\documentclass{article}
\usepackage[utf8]{inputenc}

\usepackage{amsfonts}
\usepackage{amssymb}
\usepackage{amsmath}
\usepackage{amsthm}
\usepackage{mathtools}
\usepackage{cleveref}
\usepackage{url}
\usepackage{wasysym}
\usepackage{tikz}
\usepackage{tikz-cd}
\usepackage{adjustbox}
\usepackage{stmaryrd}
\usepackage{enumitem}
\usepackage{subcaption}
\usepackage{graphicx}
\usepackage{afterpage}
\usepackage{changepage}
\usepackage[pass]{geometry}

\usepackage[
    backend=biber,
    style=alphabetic,
    sorting=nyt,
    doi=false,
    url=false,
    isbn=false,
]{biblatex}

\numberwithin{equation}{section}

\theoremstyle{plain}
    \newtheorem{theorem}{Theorem}
    \newtheorem{proposition}{Proposition}[section]
    \newtheorem{corollary}[proposition]{Corollary}
    \newtheorem{lemma}[proposition]{Lemma}
    \newtheorem{conjecture}{Conjecture}

\theoremstyle{definition}
    \newtheorem{definition}[proposition]{Definition}
    \newtheorem{algorithm}[proposition]{Algorithm}
    \newtheorem{condition}[proposition]{Condition}
    \newtheorem{example}[proposition]{Example}

\theoremstyle{remark}
	\newtheorem{remark}[proposition]{Remark}%

\newtheoremstyle{repeat-theorem}
    {\topsep}{\topsep}              
    {\itshape}                      
    {}                              
    {\bfseries}                     
    {.}                             
    { }                             
    {\thmname{#1}\thmnote{ \bfseries #3}}
    
\theoremstyle{repeat-theorem}
    \newtheorem{repeat-theorem}{Theorem}
    \newtheorem{repeat-proposition}{Proposition}
    \newtheorem{repeat-corollary}{Corollary}
    
\newcommand{\ZZ}{\mathbb{Z}}
\newcommand{\QQ}{\mathbb{Q}}
\newcommand{\RR}{\mathbb{R}}

\newcommand{\FF}{\mathbb{F}}
\renewcommand{\SS}{\mathbb{S}}

\newcommand{\id}{\mathit{id}}
\newcommand{\isom}{\cong}
\newcommand{\htpy}{\simeq}
\newcommand{\homeo}{\approx}
\newcommand{\union}{\cup}

\renewcommand{\smash}{\wedge}

\renewcommand{\epsilon}{\varepsilon}
\newcommand\restr[2]{\ensuremath{\left.#1\right|_{#2}}}

\DeclareMathOperator{\Ob}{Ob}
\DeclareMathOperator{\Hom}{Hom}
\DeclareMathOperator{\Mor}{Mor}

\DeclareMathOperator{\Ima}{Im}

\DeclareMathOperator{\Sq}{Sq}
\DeclareMathOperator{\fchar}{char}

\DeclareMathOperator{\ind}{ind}

\DeclareMathOperator{\gr}{gr}

\newcommand{\rightarrowdbl}{\rightarrow\mathrel{\mkern-14mu}\rightarrow}
\newcommand{\xrightarrowdbl}[2][]{%
  \xrightarrow[#1]{#2}\mathrel{\mkern-14mu}\rightarrow
}

\makeatletter
\newcommand*\fatcdot{\mathpalette\fatcdot@{.5}}
\newcommand*\fatcdot@[2]{\mathbin{\vcenter{\hbox{\scalebox{#2}{$\m@th#1\bullet$}}}}}
\makeatother

\newcommand{\abrac}[1]{\ensuremath{\langle{#1}\rangle}}

\DeclarePairedDelimiterX\set[1]\lbrace\rbrace{\,\setaux#1\,}
 \def\setaux#1|{#1\;\delimsize\vert\;}

\newcommand{\C}{\mathcal{C}}

\newcommand{\M}{\mathcal{M}}
\newcommand{\F}{\mathcal{F}}

\newcommand{\X}{\mathcal{X}}

\newcommand{\pt}{\mathit{pt}}

\newcommand{\Cube}[1]{{\C_\mathit{cube}(#1)}}

\newcommand{\Kh}{\mathit{Kh}}

\newcommand{\BN}{\mathit{BN}}

\renewcommand{\a}{\mathbf{a}}
\renewcommand{\b}{\mathbf{b}}
\newcommand{\x}{\mathbf{x}}
\newcommand{\y}{\mathbf{y}}
\newcommand{\z}{\mathbf{z}}
\newcommand{\w}{\mathbf{w}}

\newcommand{\ca}{\alpha}
\newcommand{\cb}{\beta}

\title{
    A Bar-Natan homotopy type
}
\author{Taketo Sano}

\begin{document}

    \maketitle
    \begin{abstract}
    A spatial refinement of Bar-Natan homology is given, that is, for any link diagram $D$ we construct a CW-spectrum $\mathcal{X}_{\mathit{BN}}(D)$ whose reduced cellular cochain complex gives the Bar-Natan complex of $D$. The stable homotopy type of $\mathcal{X}_{\mathit{BN}}(D)$ is a link invariant and is described as the wedge sum of the canonical cells. We conjecture that the quantum filtration of Bar-Natan homology also lifts to the spatial level, and that it leads us to a cohomotopical refinement of the $s$-invariant. 
\end{abstract}
    
    \setcounter{tocdepth}{2}
    \tableofcontents
    
    \section{Introduction}
\label{sec:intro}

Khovanov \cite{Khovanov:2000} introduced a categorification of the Jones polynomial, which is now known as \textit{Khovanov homology}. Lipshitz and Sarkar \cite{LS:2014} went further and constructed a spatial refinement of Khovanov homology, called \textit{Khovanov homotopy type}, a CW-spectrum whose reduced cohomology recovers Khovanov homology. The construction is based on the procedure proposed by Cohen-Jones-Segal \cite{CJS:1995}, which aims at realizing Floer homology as the ordinary homology of some naturally associated space. 

There are variants of Khovanov homology, obtained by deformations of the defining Frobenius algebra. \textit{Lee homology} \cite{Lee:2005} and \textit{Bar-Natan homology} \cite{Bar-Natan:2005} are the well-known ones. These variants are important in that they give rise to Rasmussen's \textit{$s$-invariant} \cite{Rasmussen:2010}, an integer valued knot invariant that has the strength of reproving the Milnor conjecture \cite{Milnor:1968} in a purely combinatorial manner. Recently Piccirillo \cite{Piccirillo:2020} used $s$ to disprove the sliceness of the Conway knot, a problem that remained unsolved for half a century. 

Upon this background, a question arises naturally: \textit{are there spatial refinements for the variants?} This question have been open since the construction of Khovanov homotopy type (\cite[Question 9]{LS:2018-note}). The difficulty of applying the construction to the variants is due to the increase of non-zero coefficients in the multiplication and the comultiplication of the defining Frobenius algebra. 

\vspace{0.5em}
\begin{figure}[h]
    \centering
    \resizebox{0.85\textwidth}{!}{
        \tikzset{every picture/.style={line width=0.75pt}} 

\begin{tikzpicture}[x=0.75pt,y=0.75pt,yscale=-1,xscale=1]

\draw   (134,102) -- (266.49,102) -- (266.49,137.24) -- (134,137.24) -- cycle ;

\draw   (388,102) -- (520.49,102) -- (520.49,137.24) -- (388,137.24) -- cycle ;

\draw    (274,119.99) -- (379.49,119.25) ;
\draw [shift={(381.49,119.24)}, rotate = 179.6] [color={rgb, 255:red, 0; green, 0; blue, 0 }  ][line width=0.75]    (10.93,-3.29) .. controls (6.95,-1.4) and (3.31,-0.3) .. (0,0) .. controls (3.31,0.3) and (6.95,1.4) .. (10.93,3.29)   ;
\draw   (137,15) -- (269.49,15) -- (269.49,50.24) -- (137,50.24) -- cycle ;

\draw   (389,15) -- (521.49,15) -- (521.49,50.24) -- (389,50.24) -- cycle ;

\draw    (275,32.99) -- (380.49,32.24) ;
\draw [shift={(273,33)}, rotate = 359.59] [color={rgb, 255:red, 0; green, 0; blue, 0 }  ][line width=0.75]    (10.93,-3.29) .. controls (6.95,-1.4) and (3.31,-0.3) .. (0,0) .. controls (3.31,0.3) and (6.95,1.4) .. (10.93,3.29)   ;
\draw  [dash pattern={on 4.5pt off 4.5pt}]  (197,98.24) -- (197,56.24) ;
\draw [shift={(197,54.24)}, rotate = 90] [color={rgb, 255:red, 0; green, 0; blue, 0 }  ][line width=0.75]    (10.93,-3.29) .. controls (6.95,-1.4) and (3.31,-0.3) .. (0,0) .. controls (3.31,0.3) and (6.95,1.4) .. (10.93,3.29)   ;
\draw    (456,98.24) -- (456,56.24) ;
\draw [shift={(456,54.24)}, rotate = 90] [color={rgb, 255:red, 0; green, 0; blue, 0 }  ][line width=0.75]    (10.93,-3.29) .. controls (6.95,-1.4) and (3.31,-0.3) .. (0,0) .. controls (3.31,0.3) and (6.95,1.4) .. (10.93,3.29)   ;

\draw (152,112) node [anchor=north west][inner sep=0.75pt]    {$C_{BN} \ \text{($1X$-based})$};
\draw (400,112) node [anchor=north west][inner sep=0.75pt]    {$C_{BN} \ \text{($XY$-based})$};
\draw (290,97) node [anchor=north west][inner sep=0.75pt]  [font=\normalsize] [align=left] {basis change};
\draw (290,11) node [anchor=north west][inner sep=0.75pt]  [font=\normalsize] [align=left] {handle slide};
\draw (400,24) node [anchor=north west][inner sep=0.75pt]    {$\mathcal{X}_{BN} \ \text{($XY$-based})$};
\draw (152,24) node [anchor=north west][inner sep=0.75pt]    {$\mathcal{X}_{BN} \ \text{($1X$-based})$};
\draw (20,112) node [anchor=north west][inner sep=0.75pt]   [align=left] {\underline{Chain complex}};
\draw (20,26) node [anchor=north west][inner sep=0.75pt]   [align=left] {\underline{CW-spectrum}};
\draw (463,67.43) node [anchor=north west][inner sep=0.75pt]  [font=\normalsize] [align=left] {spatial refinement};

\end{tikzpicture}
    }
    \vspace{0.5em}
    \caption{Strategy of construction}
    \label{fig:intro-idea}
\end{figure}
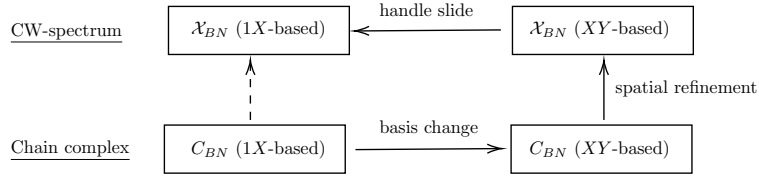

In this paper we give an affirmative answer especially for Bar-Natan's theory. The strategy of the construction is described in \Cref{fig:intro-idea}. Instead of trying to spatially refine the Bar-Natan complex from the standard generators ($1X$-based generators), we start from the diagonalized generators ($XY$-based generators). With the $XY$-based generators the structure of the complex becomes extremely simple and the construction of \cite{LS:2014} becomes applicable. Having obtained a spatial refinement for the $XY$-based complex, we recover the $1X$-based generators in the spatial level by performing \textit{handle slides}, one of the \textit{flow category moves} introduced by Lobb et.\ al.\ in \cite{JLS:2017,LOS:2018}.

\begin{theorem} \label{thm:1}
    For any link diagram $D$, there is a CW-spectrum $\X_\BN(D)$ whose reduced cellular cochain complex gives the Bar-Natan complex $C_{\BN}(D)$. The cells of $\X_\BN(D)$ correspond one-to-one to the standard generators of $C_{\BN}(D)$.
\end{theorem}

Recall that the graded module structure of Bar-Natan homology is completely known: it is freely generated by the \textit{canonical classes}. This also lifts to the spatial level.

\begin{theorem} \label{thm:2}
    $\X_\BN(D)$ is stably homotopy equivalent to the wedge sum the canonical cells. In particular, the stable homotopy type of $\X_\BN(D)$ is a link invariant.
\end{theorem}

\Cref{thm:2} tells us that $\X_\BN$ itself gives no interesting information about the corresponding link, as is true for Bar-Natan homology. What makes Bar-Natan homology an interesting object is the \textit{quantum filtration} induced from the \textit{quantum grading} defined on the standard generators. It is natural to expect that the quantum filtration also lifts to the spatial level. Although the quantum grading can be assigned to the cells, there are many attaching maps that obstruct the filtration. This is caused by the handle slides.

We conjecture that the obstructive attachings can be eliminated in a certain sense, so that the quantum filtration lifts to the spatial level. Moreover we conjecture that the correspondence between the Khovanov complex and the Bar-Natan complex also lifts to the spatial level. 

\begin{conjecture} \label{conj:1}
    There is an ascending filtration on the spectrum $\X = \X_\BN$:
    \[
        \{\pt\} = F_m \X \subset \cdots \subset F_j \X \subset F_{j + 2} \X  \subset \cdots \subset F_M \X = \X
    \]
    that refines the quantum filtration on the Bar-Natan complex $C = C_\BN$:
    \[
        C = F^m C \supset \cdots \supset F^j C \supset F^{j + 2} C \supset \cdots \supset F^M C = \{ 0 \}
    \]
    in the sense that $\tilde{C}^*(\X / F_{j - 2} \X) \isom F^j C$.
\end{conjecture}

\begin{conjecture} \label{conj:2}
    For each $j$, the quotient spectrum $F_j \X / F_{j - 2} \X$ is stably homotopy equivalent to the $j$-th wedge summand $\X^j_\Kh$ of the Khovanov spectrum $\X_\Kh$. 
\end{conjecture}

Currently, we have the following result that supports \Cref{conj:1}.

\begin{proposition}[\Cref{prop:eliminate-0dim-moduli-spaces}]
    $\X_\BN(D)$ can be modified by a stable homotopy equivalence so that for any two cells $\sigma, \tau$ of $\X_\BN(D)$ with $\dim \sigma = \dim \tau + 1$, $\sigma$ attaches to $\tau$ non-trivially if and only if the quantum grading of $\sigma$ is greater or equal to that of $\tau$. 
\end{proposition}

The problem remains for the attachings between cells of relative dimension greater than $1$. If \Cref{conj:1} is true, then it is natural to expect that we also get a  \textit{(co)homotopical refinement} of the $s$-invariant. The latter half of the paper is dedicated to build the ground that would be necessary to pursue this objective. So far we have the following:

\begin{proposition}[\Cref{def:s-bar,cor:bar-s-bound}]
    If \Cref{conj:1} is true, then there is an integer valued knot invariant $\bar{s}$ defined for each knot $K$ by
    \[
        \bar{s}(K) = \gr_q[p_\ca(K)] + 1
    \]
    where 
    \[
        [p_\ca(K)] \in \pi^0(\X_\BN(K))
    \]
    is the canonical (stable) cohomotopy class of $K$, and $\gr_q$ is the quantum grading function on $\pi^0(\X_\BN(K))$ induced from the quantum filtration of $\X_\BN(K)$. The invariant $\bar{s}$ satisfies the following:
    \begin{enumerate}
        \item $\bar{s}$ is a knot concordance invariant.
        \item $|\bar{s}(K)| \leq 2g_4(K)$.
        \item If $K$ is positive, then $\bar{s}(K) = 2g(K) = 2g_4(K)$.
    \end{enumerate}
    Here $g(K)$ denotes the genus, $g_4(K)$ denotes the smooth slice genus of $K$ respectively. In particular, property 3.\ implies the Milnor conjecture. 
\end{proposition}

The definition of $\bar{s}$ suggests that with any general cohomology theory $h^*$ we get a similar invariant $s(K; h^*)$. In particular with $h^* = \tilde{H}(-; \QQ)$ we get the original $s$. With a good choice of $h^*$, there is a hope that the invariant $s(K; h^*)$ is more tractable than $\bar{s}$, and gives a better lower bound for the slice genus, or detects the non-sliceness of some knot that $s$ fails to detect. 

\subsection*{Organization}

This paper is organized as follows. In Section 2, we review the basics of Khovanov homology and flow categories. In Section 3, we give the construction of the $XY$-based Bar-Natan flow category $\C_{XY}$ and the associated spectrum $\X_{XY}$. We also determine the stable homotopy type of $\X_{XY}$ as in \Cref{thm:2}. In Section 4, we perform handle slides on $\C_{XY}$ to obtain the $1X$-based Bar-Natan flow category $\C_{1X}$ and the associated spectrum $\X_{1X}$. This gives the desired refinement stated in \Cref{thm:1}. After stating the two conjectures, we proceed towards cohomotopically refining the $s$-invariant. Along the way, space level analogues of the standard topics of Bar-Natan homology are given, that are cobordism maps (Section 5), duality (Section 6) and canonical classes (Section 7). In the final section, we discuss future prospects regarding the two conjectures, in particular a possible definition of a cohomotopically refined $s$-invariant.

\setcounter{theorem}{0}
\setcounter{conjecture}{0}
    \subsection*{Acknowledgements}

The author is grateful to his supervisor Mikio Furuta for the support. He thanks Tomohiro Asano and Kouki Sato for helpful suggestions. He thanks members of his \textit{academist fanclub}\footnote{\url{https://taketo1024.jp/supporters}}. This work was supported by JSPS KAKENHI Grant Number 20J15094. 
    
    \section{Preliminaries}

Throughout this paper we work in the smooth category, and assume that links and their diagrams are oriented. 

\subsection{Khovanov homology theory} \label{subsec:khovanov}

\begin{definition}
    Let $R$ be a commutative ring with unity. A \textit{Frobenius algebra} over $R$ is a quintuple $(A, m, \iota, \Delta, \epsilon)$ such that: 
    \begin{enumerate}
        \item $(A, m, \iota)$ is an associative $R$-algebra with multiplication $m$ and unit $\iota$,
        \item $(A, \Delta, \epsilon)$ is a coassociative $R$-coalgebra with comultiplication $\Delta$ and counit $\epsilon$,
        \item the Frobenius relation holds: 
        \[
            \Delta \circ m = (\id \otimes m) \circ (\Delta \otimes \id) = (m \otimes \id) \circ (\id \otimes \Delta).
        \]
    \end{enumerate}
\end{definition}

\begin{definition}
    Let $R$ be a commutative ring with unity, and $h, t$ be elements in $R$. Define $A_{h, t} = R[X]/(X^2 - hX - t)$, and endow it a Frobenius algebra structure as follows: the $R$-algebra structure is inherited from $R[X]$. Regarding $A_{h, t}$ as a free $R$-module with basis $\{1, X\}$, the counit $\epsilon: A_{h, t} \rightarrow R$ is defined by
    \[
        \epsilon(1) = 0,\quad
        \epsilon(X) = 1.
    \]
    Then the comultiplication $\Delta$ is uniquely determined so that $(A_{h, t}, m, \iota, \Delta, \epsilon)$ becomes a Frobenius algebra. Explicitly, the operations $m$ and $\Delta$ are given by 
    \begin{equation}
    \label{eq:1X-operations}
        \begin{gathered}
            m(1 \otimes 1) = 1, \quad 
            m(X \otimes 1) = m(1 \otimes X) = X, \quad
        	m(X \otimes X) = hX + t, \\
            \Delta(1) = X \otimes 1 + 1 \otimes X - h (1 \otimes 1), \quad 
        	\Delta(X) = X \otimes X + t (1 \otimes 1).
        \end{gathered}
    \end{equation}
\end{definition}

\begin{definition}
    Suppose $(R, h, t)$ are taken as above. Given a link diagram $D$ with $n$ crossings, the complex $C_{h, t}(D; R)$ is defined by following the construction given in \cite{Khovanov:2000}, except that the defining Frobenius algebra $A = R[X]/(X^2)$ is replaced by $A_{h, t} = R[X]/(X^2 - hX - t)$. $C_{h, t}(D; R)$ is called the \textit{(generalized) Khovanov complex}, and its homology $H_{h, t}(D; R)$ the \textit{(generalized) Khovanov homology} of $D$ with respect to $(R, h, t)$. 
\end{definition}

Recall that Khovanov's original theory \cite{Khovanov:2000} is given by $(h, t) = (0, 0)$, Lee's theory \cite{Lee:2005} by $(h, t) = (0, 1)$, and (the filtered version of) Bar-Natan's theory \cite{BarNatan:2004} by $(h, t) = (1, 0)$. 
It is proved in \cite{Bar-Natan:2005, Khovanov:2004} that the isomorphism class of $H_{h, t}(D)$ is invariant under the Reidemeister moves. Thus it is justified to refer to the isomorphism class of $H_{h, t}(D)$ as the \textit{(generalized) Khovanov homology} of the corresponding link.

\begin{definition}
    Let $D$ be a link diagram. A \textit{state} $u$ of $D$ is an assignment of $0$ or $1$ to each crossing of $D$. When a total ordering of the crossings of $D$ is given, then $u$ is identified with an element $u \in \{0, 1\}^n$. A state $u$ yields a crossingless diagram $D(u)$ by resolving all crossings accordingly. 
\end{definition}

\begin{definition}
    Let $S$ be an arbitrary set. An \textit{$S$-enhanced state} of a link diagram $D$ is a pair $x = (u, a)$ such that $u$ is a state of $D$ and $a$ is an assignment of an element in $S$ to each circle of $D(u)$. 
    
    When $S$ is a subset of $A_{h, t}$, then an $S$-enhanced state is identified with an element of $C_{h, t}(D)$ by the corresponding tensor product of the elements of $S$. Moreover when $S$ is a basis of $A_{h, t}$, then the set of all $S$-enhanced states form a basis of $C_{h, t}(D)$. In particular for $S = \{1, X\} \subset A_{h, t}$, we call the set of all $1X$-enhanced states of $D$ the \textit{standard generators} of $C_{h, t}(D)$.
\end{definition}

\begin{definition}
\label{def:homol-quantum-grading}
    Let $D$ be a link diagram, and $n^\pm$ be the number of positive / negative crossings of $D$. For any $1X$-enhanced state $x = (u, a)$ of $D$, the \textit{homological grading} of $x$ is given by 
    \[
        \gr_h(x) = |u| - n^-,
    \]
    where $|u|$ is the Manhattan norm of $u$. The \textit{quantum grading}%
    \footnote{ 
        This definition follows \cite{Khovanov:2000}. Regarding $R[X]$ as a graded algebra, it seems more natural to declare $\deg(1) = 0$, $\deg(X) = -2$ and define $\gr_q(x) = \gr_h(x) + \deg(a) + |D(u)| + n^+ - 2n^-$.
    }
    of $x$ is defined as follows: declare $\deg(1) = 1,\ \deg(X) = -1$ and let $\deg(a)$ be the sum of the degrees of the labels on the circles of $D(u)$. Define
    \[
        \gr_q(x) = \gr_h(x) + \deg(a) + n^+ - 2n^-.
    \]
    The complex $C_{h, t}(D; R)$ is given the homological grading and the quantum grading from the standard generators. It can be easily seen that $d$ increases the homological grading by $1$. 
\end{definition}

\begin{remark}
    Depending on the degrees of $h, t \in R$, the differential $d$ of $C_{h, t}(D)$ also respects the quantum grading. When $h = t = 0$ or when $R$ is graded and $\deg(h) = -2, \deg(t) = -4$, the differential $d$ preserves the quantum grading. In this case the homology becomes bigraded. When $R$ is trivially graded and $(h, t) \neq (0, 0)$, then $d$ is quantum grading non-decreasing. In this case the homology admits the \textit{quantum filtration}. It is known that the isomorphisms corresponding to the Reidemeister moves respect the quantum grading, hence the (generalized) Khovanov homology becomes an invariant as a bigraded or a filtered module, accordingly. 
\end{remark}

\begin{remark}
    Occasionally we consider the \textit{unnormalized gradings}
    \[
        \overline{\gr}_h(x) = |u|,
        \quad
        \overline{\gr}_q(x) = \overline{\gr}_h(x) + \deg(a).
    \]
    The complex equipped with the unnormalized gradings is called the \textit{unnormalized Khovanov complex} and is denoted $\bar{C}_{h, t}(D; R)$.
\end{remark}

Lee proved that for any link diagram $D$, its $\QQ$-Lee homology has dimension $2^{|D|}$ with explicit generators called the \textit{canonical generators} (here $|D|$ denotes the number of components of the corresponding link). In particular for a knot diagram $D$, its $\QQ$-Lee homology is always $2$ dimensional. This is contrasting to Khovanov's theory, where the complexity of the module increases with respect to the complexity of the link. As discussed in \cite{MTV:2007}, the construction of the canonical generators can be generalized for any $(R, h, t)$ that satisfies the following condition.

\begin{condition} \label{cond:diagonalizable}
	The quadratic polynomial $X^2 - hX - t \in R[X]$ factors as $(X - u)(X - v) \in R[X]$ for some $u, v \in R$. 
\end{condition}

With the assumption that \Cref{cond:diagonalizable} holds, put $c = v - u$. Note that $c$ is a square root of $h^2 + 4t$. We have $c = 0$ for Khovanov's theory, $c = 2$ for Lee's theory, and $c = 1$ for Bar-Natan's theory. Define two elements in $A_{h, t}$ by
\[
    \a = X - u,\quad 
    \b = X - v
\]
If we restrict the operations $m, \Delta$ of $A_{h, t}$ to the submodule spanned by $\a$ and $\b$, then they are diagonalized as
\begin{equation}
\label{eq:ab-operations}
    \begin{gathered}
    	m(\a \otimes \a) = c\a,\ m(\a \otimes \b) = m(\b \otimes \a) = 0,\ m(\b \otimes \b) = -c\b, \\
    	\Delta(\a) = \a \otimes \a,\ \Delta(\b) = \b \otimes \b.
    \end{gathered}
\end{equation}

\begin{algorithm} \label{algo:ab-coloring}
    Given a link diagram $D$, label each of its Seifert circles by $\a$ or $\b$ according to the following algorithm: separate $\mathbb{R}^2$ into regions by the Seifert circles of $D$, and color the regions in the checkerboard fashion, with the unbounded region colored white. For each Seifert circle, let it inherit the orientation from $D$, and label it $\a$ if it sees a black region to the left, otherwise $\b$ (see \Cref{fig:ab}).
\end{algorithm}

\begin{figure}[t]
	\centering
    \includegraphics[scale=0.35]{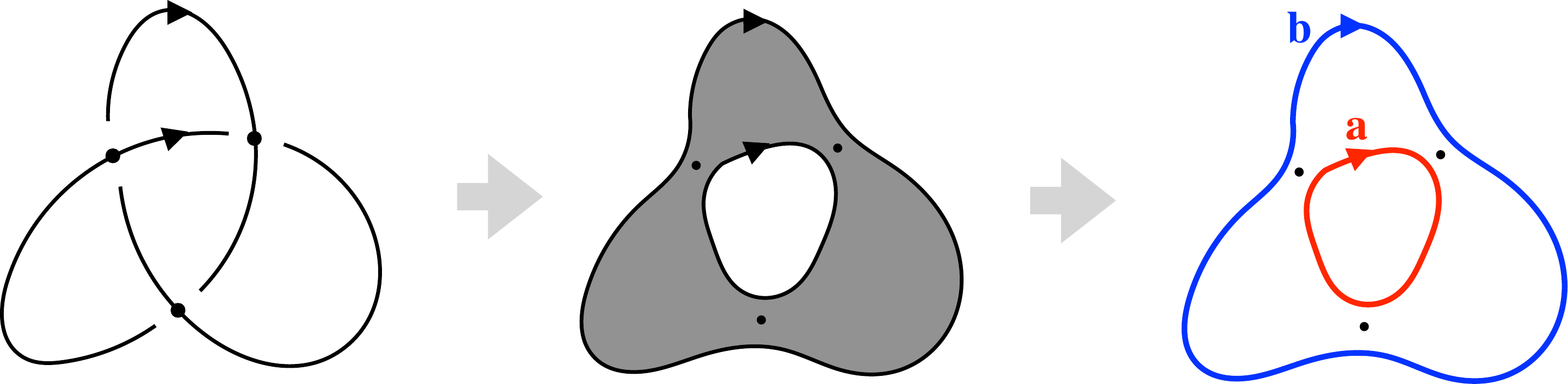}
	\caption{Coloring the Seifert circles by $\a$, $\b$.}
    \label{fig:ab}
\end{figure}

Since there is a unique state $u$ such that $D(u)$ gives the Seifert circles of $D$, the $\a\b$-labeling of \Cref{algo:ab-coloring} determines a unique $\a\b$-enhanced state $\ca(D) \in C_{h, t}(D)$. Furthermore, on the underlying unoriented diagram of $D$ there are $2^{|D|}$ possible orientations, and for each such orientation $o$, we can apply the same algorithm to obtain an $\a\b$-enhanced state $\ca(D, o) \in C_{h, t}(D)$. From \eqref{eq:ab-operations} it is easily seen that these elements are cycles.

\begin{definition}
    The cycles $\{\ca(D, o) \}_o$ are called the \textit{canonical cycles} of $D$, and those homology classes the \textit{canonical classes} of $D$. 
\end{definition}



The following is a generalization of \cite[Theorem 4.2]{Lee:2005}, which is proved in \cite[Theorem 4.2]{Turner:2020} and in \cite[Proposition 2.9]{Sano:2020}. 

\begin{proposition} \label{prop:ab-gen}
	If $c = \sqrt{h^2 + 4t}$ is invertible in $R$, then $H_{h, t}(D; R)$ is freely generated over $R$ by the canonical classes. In particular $H_{h, t}(D; R)$ has rank $2^{|D|}$.
\end{proposition}

The following proposition is proved in \cite[Proposition 4.3]{Lee:2005}, which holds verbatim in the general setting.

\begin{proposition} \label{prop:alpha-homol-gr}
    Let $D$ be an $l$-component link diagram and $D_1, \cdots, D_l$ be the component diagrams. For any orientation $o$ on the underlying unoriented diagram of $D$, let $I \subset \{1, \ldots, l\}$ be the set of indices $i$ such that $o$ is opposite to the given orientation on $D_i$. The homological grading of $\ca(D, o)$ is given by
	\[
    	2\sum_{i \in I, j \notin I} \mathit{lk}(D_i, D_j).
    \]
    where $\mathit{lk}$ denotes the linking number. In particular, when $o$ is the given orientation of $L$, then $\gr(\ca(D, o)) = 0$.
\end{proposition}

Thus when \Cref{cond:diagonalizable} is satisfied and $c = \sqrt{h^2 + 4t}$ is invertible, then from \Cref{prop:ab-gen,prop:alpha-homol-gr}, the graded module structure of $H_{h, t}(D)$ is completely determined. This is the case for $\QQ$-Lee homology, and also for $\ZZ$-Bar-Natan homology.

\subsection{Resolution configurations}

Resolution configurations were introduced in \cite{LS:2014} for the construction of the Khovanov flow category. Here, for later use, we mildly generalize some of the concepts.

\begin{definition}
    A \textit{resolution configuration} is a pair $D = (Z(D), A(D))$, where $Z(D)$ is a set of pairwise-disjoint embedded circles in $S^2$, and $A(D)$ is a totally ordered collection of disjoint arcs embedded in $S^2$ with $\bigcup Z(D) \cap \bigcup A(D) = \partial A(D)$. The cardinality of $A(D)$ is called the \textit{index} of $D$ and is denoted $ind(D)$. A resolution configuration $D$ is \textit{basic} if every circle of $Z(D)$ intersects an arc in $A(D)$, and $D$ is \textit{connected} if the underlying planar graph of $D$ is connected.
\end{definition}

\begin{definition}
    For resolution configurations $D, E$, let $D \setminus E$ be the resolution configuration with circles $Z(D) \setminus Z(E)$ and arcs $\set{a \in A(D) | a \cap \bigcup Z(E) = \varnothing }$.
\end{definition}

\begin{definition}
    Let $D$ be an index $n$ resolution configuration. For any subset $B \subset A(D)$, the \textit{surgery of $D$ along $B$}, denoted $s_B(D)$, is the resolution configuration obtained by performing embedded surgeries along all arcs $a \in B$. In particular the maximal surgery on $D$ is denoted $s(D) = s_{A(D)}(D)$. When $E = s_B(D)$ with $|B| = k$, we write $E \succeq_k D$. With the ordering of the arcs $A(D) = \{a_1, \ldots, a_n\}$, a subset $B \subset A(D)$ can be identified with a vertex $u \in \{0, 1\}^n$ under the correspondence
    \[
        u_i = 1 \Leftrightarrow a_i \in B.
    \]
    Thus we also write $D(u)$ for $s_B(D)$. 
\end{definition}

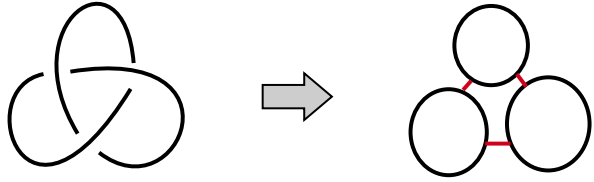
\begin{figure}[t]
    \centering
    \tikzset{every picture/.style={line width=0.75pt}} 

\begin{tikzpicture}[x=0.75pt,y=0.75pt,yscale=-1,xscale=1]

\begin{scope}
    \clip(0,0) rectangle (100, 90);
    \draw [line width=1.5]    (67.58,34.39) .. controls (62.8,-25.4) and (5.36,13.84) .. (39.36,70.01) ;
    \draw [line width=1.5]    (50.15,79.58) .. controls (90.85,111.24) and (126.02,24.21) .. (35.66,38.73) ;
    \draw [line width=1.5]    (22.13,39.95) .. controls (-15.3,47.35) and (12.68,134.29) .. (66.06,47.44) ;
\end{scope}

\draw  [fill={rgb, 255:red, 155; green, 155; blue, 155 }  ,fill opacity=0.5 ] (132.65,45.44) -- (153.61,45.44) -- (153.61,39.54) -- (167.58,51.34) -- (153.61,63.14) -- (153.61,57.24) -- (132.65,57.24) -- cycle ;
\draw  [line width=1.5]  (229.35,25.56) .. controls (229.35,14.52) and (237.65,5.57) .. (247.89,5.57) .. controls (258.12,5.57) and (266.42,14.52) .. (266.42,25.56) .. controls (266.42,36.6) and (258.12,45.56) .. (247.89,45.56) .. controls (237.65,45.56) and (229.35,36.6) .. (229.35,25.56) -- cycle ;
\draw  [line width=1.5]  (255.87,65.11) .. controls (255.87,52.13) and (265.19,41.6) .. (276.69,41.6) .. controls (288.19,41.6) and (297.52,52.13) .. (297.52,65.11) .. controls (297.52,78.09) and (288.19,88.61) .. (276.69,88.61) .. controls (265.19,88.61) and (255.87,78.09) .. (255.87,65.11) -- cycle ;
\draw  [line width=1.5]  (207.23,69.3) .. controls (207.23,57.31) and (215.84,47.6) .. (226.46,47.6) .. controls (237.08,47.6) and (245.69,57.31) .. (245.69,69.3) .. controls (245.69,81.28) and (237.08,91) .. (226.46,91) .. controls (215.84,91) and (207.23,81.28) .. (207.23,69.3) -- cycle ;
\draw [color={rgb, 255:red, 208; green, 2; blue, 27 }  ,draw opacity=1 ][line width=1.5]    (237.82,43.33) -- (233.74,47.95) ;
\draw [color={rgb, 255:red, 208; green, 2; blue, 27 }  ,draw opacity=1 ][line width=1.5]    (257.65,75.05) -- (244.9,75.05) ;
\draw [color={rgb, 255:red, 208; green, 2; blue, 27 }  ,draw opacity=1 ][line width=1.5]    (265.63,46.35) -- (260.84,39.97) ;

\end{tikzpicture}
    \caption{A resolution configuration associated to a link diagram}
    \label{fig:res-conf}
\end{figure}

\begin{definition}
    Given a link diagram $D$, we get a resolution configuration $D_0$ by performing $0$-resolutions on all crossings of $D$, and inserting arcs one for each crossing so that it is perpendicular to the resolution (see \Cref{fig:res-conf}). We call $D_0$ the \textit{resolution configuration associated} to $D$. When there is no confusion, we simply write $D$ for $D_0$. 
\end{definition}

\begin{definition}
\label{def:labeled-resolution-configuration}
    Let $S$ be an arbitrary set. An \textit{$S$-labeling} of a resolution configuration $D$ is a map
    \[
        x: Z(D) \rightarrow S.
    \]
    Such pair $(D, x)$ is called an \textit{$S$-labeled resolution configuration}.
\end{definition}

\begin{definition}\label{def:decorated-resolution-configuration}
    An \textit{$S$-decorated resolution configuration} is a triple $(D; x, y)$ such that $D$ is a resolution configuration, $y$ is an $S$-labeling of $D$, and $x$ is an $S$-labeling of $s(D)$. An $S$-decorated resolution configuration is \textit{admissible} when $(D, y) \preceq (s(D), x)$ holds.  
\end{definition}

\begin{remark}
    In \cite[Definition 2.11]{LS:2014}, admissibility is imposed in the definition of decorated resolution configurations.
\end{remark}

\begin{definition} \label{def:basic-relation}
    Let $S$ be a label set and $\mathcal{D}_S$ be the set of all (isotopy classes of) $S$-labeled resolution configurations. A \textit{basic relation} on $\mathcal{D}_S$ is a binary relation $\prec_1$ on $\mathcal{D}_S$ such that $(E, y) \prec_1 (D, x)$ holds only if (i) $E$ is basic, (ii) $\ind(E) = 1$ and (iii) $D = s(E)$ (see \Cref{fig:khovanov-relation} for an example). A basic relation $\prec_1$ extends to a binary relation between resolution configurations of relative index $1$, and its reflexive transitive closure gives a partial order $\preceq$ on $\mathcal{D}_S$, which we call the \textit{partial order generated by} $\prec_1$. 
\end{definition}

\begin{definition}
\label{def:res-conf-poset}
    Suppose $\mathcal{D}_S$ is given a partial order $\preceq$. For any resolution configuration $D$, define a subposet $P_{S}(D)$ of $\mathcal{D}_S$ by
    \[
        P_{S}(D) = \set{ (E, x) \in \mathcal{D}_S | E = s_B(D),\ B \subset A(D) }.
    \]
    For an $S$-decorated resolution configuration $(D; x, y)$, define a (possibly empty) subposet $P_{S}(D; x, y)$ of $P_{S}(D)$ by 
    \[
        P_{S}(D; x, y) 
        = \set{ (E, z) \in P_{S}(D) | (D, y) \preceq (E, z) \preceq (s(D), x) }.
    \]
    Obviously $P_{S}(D; x, y)$ is non-empty if and only if $(D; x, y)$ is admissible. 
\end{definition}

\afterpage{
\begin{figure}[t]
    \centering
    \input{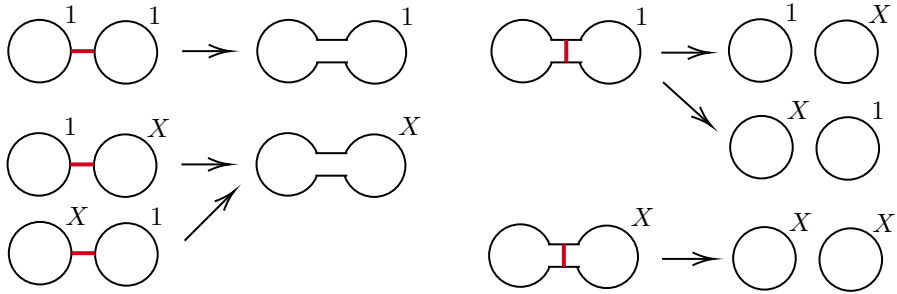}
    \caption{The basic relation for $\C_\Kh$}
    \label{fig:khovanov-relation}
\end{figure}
}

With the above setup, we may reinterpret the Khovanov chain complex using resolution configurations. Take the label set $S = \{1, X\}$ and define the basic relation $\prec_1$ as in \Cref{fig:khovanov-relation}. Given a link diagram $D$, the Khovanov chain complex $C_\Kh(D; \ZZ)$ is generated by $1X$-labeled resolution configurations $\x = (D(u), x)$ with differential $d$ defined as
\[
    d \y = \sum_{\y \prec_1 \x} (-1)^{s(e_{u, v})} \x.
\]
Here, $s$ is a sign assignment (see \Cref{def:sign-assignment}), and $u, v \in \{0, 1\}^n$ are the states of $D$ corresponding to $\x, \y$ respectively.

\subsection{Flow categories}

\textit{Flow categories} were introduced by Cohen-Jones-Segal \cite{CJS:1995} as an object that encapsulates the data of critical points and the moduli spaces of a given Floer functional. They proposed a way to build a CW-spectrum from a framed flow category, that refines the Floer chain complex. Here we review the formalization given by Lipshitz and Sarkar \cite{LS:2014}.

\begin{definition}
    A \textit{flow category} $\C$ is a topological category equipped with a grading function
    \[
        |\cdot|: \Ob(\C) \rightarrow \ZZ
    \]
    that satisfies the following conditions:
    \begin{enumerate}
        \item For any $x \in \C$, $\Hom_\C(x, x) = \{\id_x\}$. For any distinct $x, y \in \C$, $\Hom_\C(x, y)$ is a (possibly empty) compact $(|x| - |y| - 1)$-dimensional $\abrac{|x| - |y| - 1}$-manifold\footnote{An \abrac{n}-manifold $M$ is a manifold with corners, together with \textit{faces} $(\partial_1 M, \ldots, \partial_n M)$ that give a well-behaved combinatorial structure on the boundary $\partial M = \bigcup_i \partial_i M$ and the corners. See \cite{LS:2014} for the precise definition.}. $\Hom_\C(x, y)$ is denoted $\M_\C(x, y)$ and is called the \textit{moduli space} from $x$ to $y$. 
        
        \item For any distinct $x, y, z \in \C$ with $|z| - |y| = k$, the composition map
        \[
            \circ: \M_\C(z, y) \times \M_\C(x, z) \rightarrow \M_\C(x, y)
        \]
        is an embedding into $\partial_k \M_\C(x, y)$, which is also a $\abrac{|x| - |y| - 2}$-map, i.e.\ 
        \[
            \circ^{-1}(\partial_i \M_\C(x, y)) = 
            \begin{cases}
                \partial_i \M_\C(z, y) \times \M_\C(x, z) & \text{for } i < k,\\
                \M_\C(z, y) \times \partial_{i - k} \M_\C(x, z) & \text{for } k > i.
            \end{cases}
        \]

        \item For any distinct $x, y \in \C$, the composition map $\circ$ induces a diffeomorphism
        \[
            \partial_i \M_\C(x, y) \isom \bigsqcup_{z:\ |z| - |y| = i} \M_\C(z, y) \times \M_\C(x, z).
        \]
    \end{enumerate}
\end{definition}

\begin{definition}
    A full subcategory $\C'$ of a flow category $\C$ is \textit{downward-closed} (resp.\ \textit{upward-closed}) if for any $x, y \in \Ob(\C)$ with $\M_\C(x, y) \neq \varnothing$, $x \in \Ob(\C')$ implies $y \in \Ob(\C')$ (resp.\ $y \in \Ob(\C')$ implies $x \in \Ob(\C')$).
\end{definition}

\begin{definition}
    A \textit{framed flow category} is a triple $(\C, \iota, \varphi)$ such that $\C$ is a flow category, $\iota$ is a \textit{neat embedding} of $\C$ into some Euclidean space with corners, and $\varphi$ is a \textit{coherent framing} for $\iota$ (see \cite{LS:2014} for the precise definitions of the terminologies). We usually make $\iota$ and $\varphi$ implicit and regard $\C$ as a framed flow category.
\end{definition}

\begin{definition}
    Given a framed flow category $\C = (\C, \iota, \varphi)$, the \textit{associated chain complex} $C_*(\C)$ of $\C$ is defined as follows: 
    \begin{enumerate}
        \item The $i$-th chain group is freely generated by the objects of grading $i$:
        \[
            C_i(\C) = \bigoplus_{|x| = i} \ZZ x.
        \]
        \item The differential $\partial$ is defined as 
        \[
            \partial x = \sum_{|y| = |x| - 1} \# \M_\C(x, y) \cdot y
        \]
        where $\# \M_\C(x, y)$ is the signed counting of the $0$-dimensional moduli space $\M_\C(x, y)$, i.e.\ for each point $p \in \M_\C(x, y)$, its sign is given by the orientation of the framing $\varphi_{x, y}(p)$ at the point $\iota_{x, y}(p)$. 
    \end{enumerate}
    The \textit{associated cochain complex} $C^*(\C)$ is the dual complex $\Hom(C_*(\C), \ZZ)$. Given a (co)chain complex $C$, we say $\C$ \textit{refines} $C$ if the associated (co)chain complex of $\C$ is isomorphic to $C$. 
\end{definition}

\begin{proposition}[{\cite[Section 3.3]{LS:2014}}] \label{prop:CJS-const}
    Given a framed flow category $\C$, there is a CW-complex $|\C|$ satisfying the following conditions: 
    \begin{enumerate}
        \item The cells of $|\C|$ correspond bijectively to the objects of $\C$ (except for the basepoint),
        \item there is an integer $\ell \in \ZZ$ such that for each object $x \in \Ob(\C)$ the corresponding cell $\sigma_x$ has $\dim(\sigma_x) = |x| + \ell$,
        \item there is an isomorphism between the reduced cellular chain complex $\tilde{C}_*(|\C|)$ of $|\C|$ and the associated chain complex $C_*(\C)$ such that the cells of $|\C|$ are mapped to the generators of $C_*(\C)$, and
        \item the formal desuspension $\Sigma^{-\ell}|\C|$ is independent, up to stable homotopy equivalence, of the choice of the framed neat embedding of $\C$ and other auxiliary choices made for the construction of $|\C|$.
    \end{enumerate}
\end{proposition}

\begin{definition}
    The CW-complex $|\C|$ of \Cref{prop:CJS-const} is called the \textit{realization} of the framed flow category $\C$ and the spectrum $\X(\C) := \Sigma^{-\ell}|\C|$ is called the \textit{associated spectrum} of $\C$. 
\end{definition}

\begin{remark}
    A downward-closed (resp.\ upward-closed) subcategory $\C'$ of a framed flow category $\C$ inherits the framing from $\C$, and induces the inclusion $|\C'| \hookrightarrow |\C|$ (resp.\ the projection $|\C| \rightarrowdbl |\C'|$) between the realizations.
\end{remark}

The definition of \textit{cube flow categories} follows. They will be necessary in the constructions of Khovanov homotopy type and our Bar-Natan homotopy type. We first fix some notations. Let $K(n)$ denote the $n$-dimensional unit cube $[0, 1]^n$ endowed with the standard CW-complex structure. Vertices of $K(n)$ are given the obvious partial order, that is $u \geq v$ if $u_i \geq v_i$ for each $1 \leq i \leq n$. Denote $u \geq_k v$ if $u \geq v$ and $|u - v| = k$. For vertices $u \geq_k v$, the $k$-face spanned by $u, v$ is denoted $e_{u, v}$. It is convenient to express $e_{u, v}$ as coordinates with $k$ stars
\[
    (u_1, \ldots, u_{i_1 - 1}, \star, u_{i_1 + 1}, \ldots, u_{i_k - 1}, \star, u_{i_k + 1}, \ldots, u_n)
\]
where the stars are placed at coordinates $i$ such that $u_i > v_i$. 

\begin{definition}
\label{def:sign-assignment}
    A 1-cochain $s \in C^1(K(n); \FF_2)$ is called a \textit{sign assignment} for $K(n)$ if $\delta s = 1$ (a constant map). The \textit{standard sign assignment} $s$ is defined as 
    \[
        s(e) = v_1 + \cdots + v_{i - 1} \pmod{2}
    \]
    for each edge $e = (v_1, \ldots, v_{i - 1}, \star, v_{i + 1}, \ldots, v_n)$ of $K(n)$. 
\end{definition}

\begin{definition}
    Given any sign assignment $s$ for $K(n)$, the $n$-dimensional \textit{cube complex} $C_*(n, s)$ is the chain complex freely generated over $\ZZ$ by $\{0, 1\}^n$ and the differential $\partial$ given by
    \[
        \partial u = \sum_{u >_1 v} (-1)^{s(e_{u, v})} v.
    \]
    Obviously $C_*(n, s)$ is acyclic.
\end{definition}

\begin{definition}
\label{def:cube-flow-cat}
    The \textit{$n$-dimensional cube flow category} $\Cube{n}$ is defined as follows: the object set is $\{0, 1\}^n$, and the moduli space between vertices $u >_k v$ is the $(k - 1)$-dimensional \textit{permutohedron} $\Pi_{k - 1}$. The composition map is defined to coherently so that $\Cube{n}$ is actually a flow category (see \cite[Definition 3.16]{LLS:2020} for the precise definition). 
\end{definition}

\begin{remark}
    Originally in \cite[Definition 4.1]{LS:2014}, the cube flow category $\Cube{n}$ is defined as the \textit{Morse flow category} of the Morse function
    \[
        f_n: \RR^n \rightarrow \RR, \quad
        (x_1, \ldots, x_n) \mapsto f(x_1) + \cdots + f(x_n)
    \]
    where $f: \RR \rightarrow \RR$ is a self-indexing Morse function with critical points $\{0, 1\}$.
\end{remark}

\begin{proposition}[{\cite[Proposition 4.12]{LS:2014}}] \label{prop:cube-flow-cat-framing}
    Given any sign assignment $s$ for $K(n)$, there is a framed neat embedding $(\iota, \varphi)$ of $\Cube{n}$ such that the framed flow category $(\Cube{n}, \iota, \varphi)$ refines the cube complex $C_*(n, s)$.
\end{proposition}

Here we briefly review the construction of the \textit{Khovanov flow category} $\C_\Kh(D)$ given in \cite{LS:2014}. Suppose $D$ is a link diagram with $n$ crossings. The objects of $\C_\Kh(D)$ are given by the $1X$-labeled resolution configurations of $D$. With the basic relation of \Cref{fig:khovanov-relation}, the 0-dimensional moduli spaces are given by poset $P(D)^{op}$. Higher dimensional moduli spaces are constructed inductively so that each moduli space $\M_\Kh(\x, \y)$ trivially covers $\M_{\Cube{|\x| - |\y|}}(u, v)$ where $u, v$ are the states of $\x, \y$ respectively. This yields a covering functor $\F : \C_\Kh(D) \rightarrow \Cube{n}$. Take any sign assignment $s$ for $K(n)$, and a framed neat embedding of $\Cube{n}$ relative to $s$. This is pulled back to $\C_\Kh(D)$ by $\F$. Then the resulting framed flow category refines the Khovanov complex $C_\Kh(D)$. The \textit{Khovanov spectrum} $\X_\Kh(D)$ is defined as the spectrum associated to $\C_\Kh(D)$, and is proved that its stable homotopy type is invariant under the Reidemeister moves. 

\subsection{Flow category moves} \label{subsec:morse-moves}

Finally we summarize the three \textit{flow category moves} given in \cite{JLS:2017,LOS:2018}. 


\begin{proposition}[Handle cancellation, {\cite[Theorem 2.17]{JLS:2017}}] 
\label{prop:handle-cancel}
    \begin{equation*}
        \resizebox{0.6\textwidth}{!}{
\begin{tikzcd}
a\ \bullet\ \  \arrow[rd] \arrow[ddd] \arrow[rdd] &                                      &  & a\ \bullet\ \  \arrow[ddd] \arrow[rdd] \arrow[rd] &                                    \\
                                                  & \ \ \bullet\ x  \arrow[d] \arrow[ldd] &  &                                                   & \circ \arrow[ldd]                  \\
                                                  & \ \ \bullet\ y  \arrow[ld]            &  &                                                   & \circ \arrow[u] \arrow[ld] \\
b\ \bullet\ \                                     &                                      &  & b\ \bullet\ \                                     &                                   
\end{tikzcd}
        }
    \end{equation*}
    Let $\C$ be a framed flow category. Suppose there are objects $x, y \in \C$ such that $|x| = |y| + 1$ and $\M_\C(x, y) = \{ \pt \}$. Then there is a framed flow category $\C'$ satisfying the following conditions:
    \begin{enumerate}
        \item 
        $\Ob(\C') = \Ob(\C) \setminus \{x, y\}$.

        \item 
        For any objects $a, b \in \C'$, 
        \[
            \M_{\C'}(a, b) = (\M_\C(a, b)\ \sqcup \ \M_\C(x, b) \times \M_\C(a, y)) / \sim
        \]
        where $\sim$ identifies the subspace of $\partial \M_\C(a, b)$:
        \[
            \M_\C(x, b) \times \M_\C(a, x) \ \union\  \M_\C(y, b) \times \M_\C(a, y)
        \]
        and the subspace of $\partial(\M_\C(x, b) \times \M_\C(a, y))$:
        \[
            \M_\C(x, b) \times (\M_\C(x, y) \times \M_\C(a, x)) \ \union\ (\M_\C(y, b) \times \M_\C(x, y)) \times \M_\C(a, y).
        \]
        \item
        $|\C'|$ is homotopy equivalent to $|\C|$.
    \end{enumerate}
\end{proposition}

\begin{remark}
    In particular if, for every pair $(a, b)$ in $\Ob(\C) \setminus \{x, y\}$, either one of $\M_\C(x, b)$ or $\M_\C(a, y)$ is empty, then from the above formula $\M_{\C}(a, b)$ remains unchanged and $\C'$ is obtained by simply removing the two objects $x, y$.
\end{remark}

\begin{proposition}[Handle sliding, {\cite[Theorem 3.1]{LOS:2018}}]
\label{prop:handle-slide}
    \begin{equation*}
        \resizebox{0.6\textwidth}{!}{
\begin{tikzcd}[row sep={2cm,between origins}]
a\ \bullet\ \  \arrow[d, "\alpha"'] \arrow[rd, "\beta"]         &                                    &  & a\ \bullet\ \  \arrow[rd, "\beta"] \arrow[d, "\alpha"'] \arrow[rd, "\mp \alpha", bend right=32]        &                                    \\
x\ \bullet\ \  \arrow[r, "(\pm)", dashed] \arrow[rd, "\gamma"'] & \ \bullet\ y  \arrow[d, "\delta"] &  & x'\ \bullet\ \  \arrow[r, no head, dotted] \arrow[rd, "\gamma"'] \arrow[rd, "\pm \delta"', bend left=32] & \ \bullet\ y  \arrow[d, "\delta"] \\
                                                                & \ \bullet\ b                     &  &                                                                                                      & \ \bullet\ b                    
\end{tikzcd}
        }
    \end{equation*}
    Let $\C$ be a framed flow category and $\epsilon \in \{\pm 1\}$. For any objects $x, y \in \C$ of the same grading, there is a framed flow category $\C'$ satisfying the following:
    \begin{enumerate}
        \item 
        $\Ob(C') = (\Ob(\C) \setminus \{x\}) \union \{x'\}$.
        
        \item 
        For any objects $a, b \in \C'$, 
        \begin{align*}
            \M_{\C'}(a, y) &= (-\epsilon) \M_\C(a, x) \sqcup \M_\C(a, y), \\
            \M_{\C'}(x', b) &= \M_\C(x, b) \sqcup (\epsilon)\M_\C(y, b), \\
            \M_{\C'}(a, b) &= \M_\C(a, b) \sqcup (\M_\C(y, b) \times [0, 1] \times \M_\C(a, x)).
        \end{align*}
        Here, the signs $(\pm \epsilon)$ indicate that the framings are changed accordingly. 

        \item
        $|\C'|$ is homotopy equivalent to $|\C|$.
    \end{enumerate}
\end{proposition}

\begin{proposition}[Whitney trick, {\cite[Theorem 2.8]{JLS:2017}}]
\label{prop:whitney-trick}
    \begin{equation*}
        \resizebox{0.6\textwidth}{!}{
\begin{tikzcd}
a\ \bullet\ \  \arrow[rd] \arrow[ddd] \arrow[rdd, bend right=25] &                                                                                  &  & a\ \bullet\ \  \arrow[ddd] \arrow[rdd, bend right=25] \arrow[rd] &                                                       \\
                                                  & \ \ \bullet\ x  \arrow[d, "+"', bend right=15] \arrow[d, "-", bend left=15] \arrow[ldd, bend right=25] &  &                                                   & \ \bullet\ x  \arrow[ldd, bend right=25] \\
                                                  & \ \bullet\ y  \arrow[ld]                                                        &  &                                                   & \ \ \bullet\ y  \arrow[ld]                             \\
b\ \bullet\ \                                     &                                                                                  &  & b\ \bullet\ \                                     &                                                      
\end{tikzcd}
        }
    \end{equation*}
    Let $\C$ be a framed flow category. Suppose there are objects $x, y \in \C$ such that $|x| = |y| + 1$ and the $0$-dimensional moduli space $\M_\C(x, y)$ contains two points $P, Q$ with opposite framings. Then there is a framed flow category $\C'$ satisfying the following:
    \begin{enumerate}
        \item 
        $\Ob(\C') = \Ob(\C)$.
        
        \item 
        For the two objects $x, y$,
        \[
            \M_{\C'}(x, y) = \M_\C(x, y) \setminus \{ P, Q \}.
        \]
        For any objects $a, b$, 
        \begin{align*}
            \M_{\C'}(a, y) &= (\M_\C(a, y) \setminus \{ P, Q \} \times \M_\C(a, x)) / \sim, \\
            \M_{\C'}(x, b) &= (\M_\C(x, b) \setminus \M_\C(y, b) \times \{ P, Q \}) / \sim, \\
            \M_{\C'}(a, b) &= (\M_\C(a, b) \setminus \M_\C(y, b) \times \{ P, Q \} \times \M_\C(a, x)) / \sim.
        \end{align*}
        Here, the identifications $\sim$ are defined so that the corners of the moduli spaces are glued appropriately. See \cite[Definition 2.5]{JLS:2017} for the precise definition.
        
        \item
        $|\C'|$ is homotopy equivalent to $|\C|$.
    \end{enumerate}
\end{proposition}

\begin{remark}\label{rem:extended-whitney-trick}
    In \cite[Section 4]{LOS:2018}, a more generalized \textit{extended Whitney trick} is introduced, where $x, y$ need not satisfy $|x| = |y| + 1$, and the moduli space $\M_\C(x, y)$ can be replaced by any framed cobordant manifold rel boundary. 
\end{remark}

\begin{definition}
    Two framed flow categories are \textit{move equivalent} if they can be interchanged by a finite sequence of flow category moves (and isotopies and stabilizations of the framed neat embeddings).
\end{definition}

\begin{remark}
    It is conjectured in \cite[Conjecture 6.4]{JLS:2017} that two framed flow categories determine the same stable homotopy type if and only if they are move equivalent.
    
\end{remark}

\clearpage
    \section{Basic construction} \label{sec:bar-natan-spectrum}

In this section we construct a spatial refinement of the Bar-Natan complex $C_\BN(D)$. Recall from \Cref{subsec:khovanov} that the defining Frobenius algebra for Bar-Natan homology is $A = A_{1, 0}$ whose operations are given by 
\begin{equation}
\label{eq:1X-operations-BN}
    \begin{gathered}
        m(1 \otimes 1) = 1, \quad 
        m(X \otimes 1) = m(1 \otimes X) = X, \quad
    	m(X \otimes X) = X, \\
        \Delta(1) = X \otimes 1 + 1 \otimes X - 1 \otimes 1, \quad 
    	\Delta(X) = X \otimes X.
    \end{gathered}
\end{equation}
In order to apply the construction of \cite{LS:2014} to other variants of Khovanov homology, it is necessary that the following conditions hold: (i) the non-zero coefficients appearing in the multiplication and the comultiplication of $A_{h, t}$ must be $1$, and (ii) the poset associated to the chain complex must be cubical in a certain sense (see \cite[Proposition 5.2]{LS:2014} for the precise condition). Unfortunately, Bar-Natan's theory do not satisfy the two conditions. (Neither does Lee's, see \cite[Remark 4.22]{LLS:2020}.)

However, if we put $Y = X - 1$ and use basis $\{X, Y\}$ instead of $\{1, X\}$, then the operations diagonalize as
\begin{equation}
\label{eq:XY-operations}
    \begin{gathered}
        m(X \otimes X) = X, \quad 
        m(X \otimes Y) = m(Y \otimes X) = 0, \quad
        m(Y \otimes Y) = -Y,\\
        \Delta(X) = X \otimes X, \quad 
    	\Delta(Y) = Y \otimes Y.
    \end{gathered}
\end{equation}
We see that the first condition is almost satisfied, except for the negative sign appearing in $m(Y \otimes Y) = -Y$. We will see that the second condition also holds. Thus after taking a specific framing so that the negative sign is taken care of, we will be able to apply the construction of \cite{LS:2014} to the $XY$-based Bar-Natan complex. The procedure of recovering the standard basis $\{1, X\}$ in the flow category level is discussed in the next section.

\subsection{$XY$-based Bar-Natan flow category} \label{subsec:C_BN}

\begin{figure}[t]
    \centering
    \tikzset{every picture/.style={line width=0.75pt}} 

\begin{tikzpicture}[x=0.75pt,y=0.75pt,yscale=-1,xscale=1]

\draw   (16,27.02) .. controls (16,18.3) and (23.06,11.24) .. (31.78,11.24) .. controls (40.49,11.24) and (47.55,18.3) .. (47.55,27.02) .. controls (47.55,35.73) and (40.49,42.79) .. (31.78,42.79) .. controls (23.06,42.79) and (16,35.73) .. (16,27.02) -- cycle ;
\draw    (103.71,27.02) -- (126,27.02) ;
\draw [shift={(128,27.02)}, rotate = 180] [color={rgb, 255:red, 0; green, 0; blue, 0 }  ][line width=0.75]    (10.93,-3.29) .. controls (6.95,-1.4) and (3.31,-0.3) .. (0,0) .. controls (3.31,0.3) and (6.95,1.4) .. (10.93,3.29)   ;
\draw   (59.54,27.65) .. controls (59.54,18.93) and (66.6,11.87) .. (75.32,11.87) .. controls (84.03,11.87) and (91.09,18.93) .. (91.09,27.65) .. controls (91.09,36.36) and (84.03,43.42) .. (75.32,43.42) .. controls (66.6,43.42) and (59.54,36.36) .. (59.54,27.65) -- cycle ;
\draw   (141.57,27.02) .. controls (141.57,18.3) and (148.63,11.24) .. (157.35,11.24) .. controls (166.06,11.24) and (173.12,18.3) .. (173.12,27.02) .. controls (173.12,35.73) and (166.06,42.79) .. (157.35,42.79) .. controls (148.63,42.79) and (141.57,35.73) .. (141.57,27.02) -- cycle ;
\draw   (185.11,27.65) .. controls (185.11,18.93) and (192.17,11.87) .. (200.89,11.87) .. controls (209.6,11.87) and (216.66,18.93) .. (216.66,27.65) .. controls (216.66,36.36) and (209.6,43.42) .. (200.89,43.42) .. controls (192.17,43.42) and (185.11,36.36) .. (185.11,27.65) -- cycle ;
\draw  [draw opacity=0][fill={rgb, 255:red, 255; green, 255; blue, 255 }  ,fill opacity=1 ] (164.29,21.34) -- (193,21.34) -- (193,32.69) -- (164.29,32.69) -- cycle ;
\draw    (171.23,21.34) -- (187.32,21.34) ;
\draw    (171.23,32.69) -- (187.32,32.69) ;
\draw [color={rgb, 255:red, 208; green, 2; blue, 27 }  ,draw opacity=1 ][line width=1.5]    (47.55,27.02) -- (59.54,27.02) ;
\draw   (382.58,78.59) .. controls (382.58,69.88) and (389.64,62.82) .. (398.35,62.82) .. controls (407.06,62.82) and (414.13,69.88) .. (414.13,78.59) .. controls (414.13,87.31) and (407.06,94.37) .. (398.35,94.37) .. controls (389.64,94.37) and (382.58,87.31) .. (382.58,78.59) -- cycle ;
\draw    (345.33,27.78) -- (367.62,27.78) ;
\draw [shift={(369.62,27.78)}, rotate = 180] [color={rgb, 255:red, 0; green, 0; blue, 0 }  ][line width=0.75]    (10.93,-3.29) .. controls (6.95,-1.4) and (3.31,-0.3) .. (0,0) .. controls (3.31,0.3) and (6.95,1.4) .. (10.93,3.29)   ;
\draw   (426.12,79.22) .. controls (426.12,70.51) and (433.18,63.45) .. (441.89,63.45) .. controls (450.6,63.45) and (457.67,70.51) .. (457.67,79.22) .. controls (457.67,87.94) and (450.6,95) .. (441.89,95) .. controls (433.18,95) and (426.12,87.94) .. (426.12,79.22) -- cycle ;
\draw   (259.69,27.02) .. controls (259.69,18.3) and (266.76,11.24) .. (275.47,11.24) .. controls (284.18,11.24) and (291.25,18.3) .. (291.25,27.02) .. controls (291.25,35.73) and (284.18,42.79) .. (275.47,42.79) .. controls (266.76,42.79) and (259.69,35.73) .. (259.69,27.02) -- cycle ;
\draw   (303.23,27.65) .. controls (303.23,18.93) and (310.3,11.87) .. (319.01,11.87) .. controls (327.72,11.87) and (334.79,18.93) .. (334.79,27.65) .. controls (334.79,36.36) and (327.72,43.42) .. (319.01,43.42) .. controls (310.3,43.42) and (303.23,36.36) .. (303.23,27.65) -- cycle ;
\draw  [draw opacity=0][fill={rgb, 255:red, 255; green, 255; blue, 255 }  ,fill opacity=1 ] (282.41,21.34) -- (311.12,21.34) -- (311.12,32.69) -- (282.41,32.69) -- cycle ;
\draw    (289.35,21.34) -- (305.44,21.34) ;
\draw    (289.35,32.69) -- (305.44,32.69) ;
\draw   (379.44,26.7) .. controls (379.44,17.99) and (386.5,10.92) .. (395.21,10.92) .. controls (403.93,10.92) and (410.99,17.99) .. (410.99,26.7) .. controls (410.99,35.41) and (403.93,42.47) .. (395.21,42.47) .. controls (386.5,42.47) and (379.44,35.41) .. (379.44,26.7) -- cycle ;
\draw   (422.98,27.33) .. controls (422.98,18.62) and (430.04,11.55) .. (438.75,11.55) .. controls (447.47,11.55) and (454.53,18.62) .. (454.53,27.33) .. controls (454.53,36.04) and (447.47,43.1) .. (438.75,43.1) .. controls (430.04,43.1) and (422.98,36.04) .. (422.98,27.33) -- cycle ;
\draw   (259.49,77.21) .. controls (259.49,68.5) and (266.55,61.43) .. (275.27,61.43) .. controls (283.98,61.43) and (291.04,68.5) .. (291.04,77.21) .. controls (291.04,85.92) and (283.98,92.98) .. (275.27,92.98) .. controls (266.55,92.98) and (259.49,85.92) .. (259.49,77.21) -- cycle ;
\draw   (303.03,77.84) .. controls (303.03,69.13) and (310.09,62.06) .. (318.81,62.06) .. controls (327.52,62.06) and (334.58,69.13) .. (334.58,77.84) .. controls (334.58,86.55) and (327.52,93.61) .. (318.81,93.61) .. controls (310.09,93.61) and (303.03,86.55) .. (303.03,77.84) -- cycle ;
\draw  [draw opacity=0][fill={rgb, 255:red, 255; green, 255; blue, 255 }  ,fill opacity=1 ] (282.21,71.53) -- (310.92,71.53) -- (310.92,82.89) -- (282.21,82.89) -- cycle ;
\draw    (289.15,71.53) -- (305.24,71.53) ;
\draw    (289.15,82.89) -- (305.24,82.89) ;
\draw    (346.7,78.88) -- (370.99,78.88) ;
\draw [shift={(372.99,78.88)}, rotate = 180] [color={rgb, 255:red, 0; green, 0; blue, 0 }  ][line width=0.75]    (10.93,-3.29) .. controls (6.95,-1.4) and (3.31,-0.3) .. (0,0) .. controls (3.31,0.3) and (6.95,1.4) .. (10.93,3.29)   ;
\draw [color={rgb, 255:red, 208; green, 2; blue, 27 }  ,draw opacity=1 ][line width=1.5]    (297.4,21.34) -- (297.4,32.69) ;
\draw [color={rgb, 255:red, 208; green, 2; blue, 27 }  ,draw opacity=1 ][line width=1.5]    (297.19,71.53) -- (297.19,82.89) ;
\draw   (17,78.02) .. controls (17,69.3) and (24.06,62.24) .. (32.78,62.24) .. controls (41.49,62.24) and (48.55,69.3) .. (48.55,78.02) .. controls (48.55,86.73) and (41.49,93.79) .. (32.78,93.79) .. controls (24.06,93.79) and (17,86.73) .. (17,78.02) -- cycle ;
\draw    (104.71,78.02) -- (127,78.02) ;
\draw [shift={(129,78.02)}, rotate = 180] [color={rgb, 255:red, 0; green, 0; blue, 0 }  ][line width=0.75]    (10.93,-3.29) .. controls (6.95,-1.4) and (3.31,-0.3) .. (0,0) .. controls (3.31,0.3) and (6.95,1.4) .. (10.93,3.29)   ;
\draw   (60.54,78.65) .. controls (60.54,69.93) and (67.6,62.87) .. (76.32,62.87) .. controls (85.03,62.87) and (92.09,69.93) .. (92.09,78.65) .. controls (92.09,87.36) and (85.03,94.42) .. (76.32,94.42) .. controls (67.6,94.42) and (60.54,87.36) .. (60.54,78.65) -- cycle ;
\draw   (142.57,78.02) .. controls (142.57,69.3) and (149.63,62.24) .. (158.35,62.24) .. controls (167.06,62.24) and (174.12,69.3) .. (174.12,78.02) .. controls (174.12,86.73) and (167.06,93.79) .. (158.35,93.79) .. controls (149.63,93.79) and (142.57,86.73) .. (142.57,78.02) -- cycle ;
\draw   (186.11,78.65) .. controls (186.11,69.93) and (193.17,62.87) .. (201.89,62.87) .. controls (210.6,62.87) and (217.66,69.93) .. (217.66,78.65) .. controls (217.66,87.36) and (210.6,94.42) .. (201.89,94.42) .. controls (193.17,94.42) and (186.11,87.36) .. (186.11,78.65) -- cycle ;
\draw  [draw opacity=0][fill={rgb, 255:red, 255; green, 255; blue, 255 }  ,fill opacity=1 ] (165.29,72.34) -- (194,72.34) -- (194,83.69) -- (165.29,83.69) -- cycle ;
\draw    (172.23,72.34) -- (188.32,72.34) ;
\draw    (172.23,83.69) -- (188.32,83.69) ;
\draw [color={rgb, 255:red, 208; green, 2; blue, 27 }  ,draw opacity=1 ][line width=1.5]    (48.55,78.02) -- (60.54,78.02) ;

\draw (42.81,2.61) node [anchor=north west][inner sep=0.75pt]    {$X$};
\draw (85.09,3.24) node [anchor=north west][inner sep=0.75pt]    {$X$};
\draw (212.55,3.87) node [anchor=north west][inner sep=0.75pt]    {$X$};
\draw (448.16,2.29) node [anchor=north west][inner sep=0.75pt]    {$X$};
\draw (406.25,1.66) node [anchor=north west][inner sep=0.75pt]    {$X$};
\draw (329.47,50.54) node [anchor=north west][inner sep=0.75pt]    {$Y$};
\draw (330.68,3.87) node [anchor=north west][inner sep=0.75pt]    {$X$};
\draw (451.43,54.82) node [anchor=north west][inner sep=0.75pt]    {$Y$};
\draw (409.16,54.19) node [anchor=north west][inner sep=0.75pt]    {$Y$};
\draw (43.81,53.61) node [anchor=north west][inner sep=0.75pt]    {$Y$};
\draw (86.09,54.24) node [anchor=north west][inner sep=0.75pt]    {$Y$};
\draw (213.55,54.87) node [anchor=north west][inner sep=0.75pt]    {$Y$};

\end{tikzpicture}
    \caption{The basic relation for $\X_{XY}$}
    \label{fig:basic-rel-XY}
\end{figure}
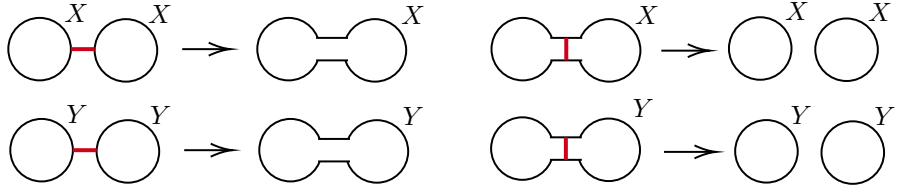

Consider $XY$-labeled resolution configurations with the basic relation defined as in \Cref{fig:basic-rel-XY}. Note that the relation comes from \eqref{eq:XY-operations}, except that that the negative sign appearing in $m(Y \otimes Y) = -Y$ has disappeared in the corresponding arrow. Due to the simplicity of the basic relation, the structure of the associated poset turns out to be extremely simple. 

\begin{proposition} \label{prop:P(Dxy)-locally-constant}
    Let $D$ be a resolution configuration with connected components $D_1, \cdots, D_l$. An $XY$-decorated resolution configuration $(D; x, y)$ is admissible if and only if (i) $y$ is constant on each $D_i$, (ii) $x$ is constant on each $s(D_i)$ and (iii) $\restr{y}{D_i} = \restr{x}{s(D_i)}$ for each $D_i$.
\end{proposition}

\begin{proof}
    If there is an arc that connects differently labeled circles, there will be no available labeling after the merge of the two circles.
\end{proof}

\begin{proposition} \label{prop:P(Dxy)-structure}
    Let $(D; x, y)$ be an index $k$ admissible $XY$-decorated resolution configuration. The poset $P_{XY}(D; x, y)$ is isomorphic to the $k$-dimensional cube poset $\{0, 1\}^k$.
\end{proposition}

\begin{proof}
    For any $(E, z) \in P_{XY}(D; x, y)$, $E$ corresponds one-to-one to a subset $B \subset A(D)$ and $z$ is automatically determined by $E$.
\end{proof}

\begin{proposition}\label{prop:P(D)-decomp}
    For any resolution configuration $D$, the associated poset $P_{XY}(D)$ decomposes into a disjoint union of cube posets.
\end{proposition}

\begin{proof}
    Consider the set of all minimal objects $\{ (D_i, y_i) \}_i$ in $P_{XY}(D)$. For each $i$, there is a unique maximal object $(D'_i, x_i)$ that succeeds or equals $(D_i, y_i)$, which is obtained by surgering all arcs of $D_i$ that connect circles labeled the same. With $k_i = \ind(D_i) - \ind(D'_i)$, it is obvious that the subposet spanned by $(D_i, y_i) \preceq (D'_i, x_i)$ is isomorphic to $P_{XY}(D_i \setminus D'_i;\ x_i, y_i) \isom \{0, 1\}^{k_i}$.
\end{proof}

\begin{example}
    \begin{figure}[t]
        \centering
        \resizebox{\textwidth}{!}{
            \input{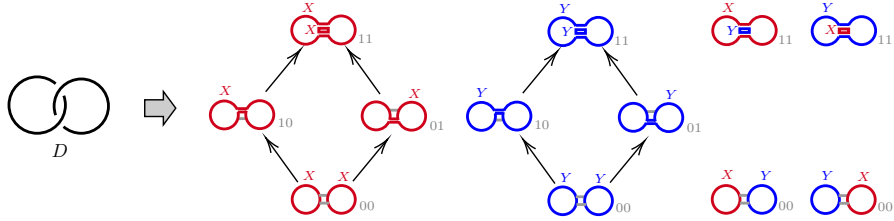}
        }
        \vspace{0.8em}
        \caption{An example of $D$ and its associated poset $P_{XY}(D)$}
        \label{fig:poset-hopf-link}
    \end{figure}
    An example of $D$ and its associated poset $P_\BN(D)$ is given in \Cref{fig:poset-hopf-link}. We see that $P_\BN(D)$ is the disjoint union of two $2$-cubes and four $0$-cubes.
\end{example}

Suppose $D$ is an index $n$ resolution configuration. Endow a grading on the objects of $P_{XY}(D)$ by 
\[
    \gr{(E, z)} = n - \ind(E),
\]
and on the objects of $\{0, 1\}^n$ by 
\[
    \gr(u) = |u|.
\]
Then the decomposition of \Cref{prop:P(D)-decomp} gives a grading preserving poset homomorphism
\[
    F: P_{XY}(D) \rightarrow \{0, 1\}^n,
    \quad 
    (D(u), z) \mapsto u
\]
which is injective on each cube subposet. We aim to construct a flow category from $P_{XY}(D)$ by fattening the arrows into moduli spaces. Unlike the case of the Khovanov flow category, the construction is done at once.

\begin{definition}\label{def:C_BN_star}
    Let $D$ be a non-empty index $n$ resolution configuration. The flow category $\bar{\C}_{XY}(D)$ is defined as follows: the object set is given by 
    \[
        \Ob(\bar{\C}_{XY}(D)) = \Ob(P_{XY}(D))
    \]
    together with the grading function defined as above. For each $\x, \y \in \bar{\C}_{XY}(D)$, the moduli space $\M_{XY}(\x, \y)$ is given by 
    \[
        \M_{XY}(\x, \y) = \M_\Cube{n}(F(\x), F(\y)).
    \]
    When $D$ is empty, define $\C_{XY}(\varnothing) = \mathbf{1}$, the discrete category with a single object. 
\end{definition}

From the above definition, there is an obvious isomorphism 
\[
    \bar{\C}_{XY}(D) \ \isom \ \bigsqcup_i \Cube{k_i}[d_i]
\]
where $[d_i]$ are appropriate grading shifts. There is also a grading preserving functor $\F$ that $F$ lifts as
\begin{equation*}
    \begin{tikzcd}
    \bar{\C}_{XY}(D) \arrow[r, "\mathcal{F}"] \arrow[d, two heads] & \Cube{n} \arrow[d, two heads] \\
    P_{XY}(D)^{\mathit{op}} \arrow[r, "F^{\mathit{op}}"]                                             & (\{0, 1\}^n)^{\mathit{op}}.            
    \end{tikzcd}
\end{equation*}
The two vertical arrows are identities on the object sets. If we restrict $\F$ to the subcategory $\Cube{k_i}[d_i]$ of $\bar{\C}_{XY}(D)$, we get inclusions
\begin{equation*}
    \begin{tikzcd}
    \Cube{k_i}[d_i] \arrow[r, hook, "\mathcal{I}_i"] \arrow[d, two heads] & \Cube{n} \arrow[d, two heads] \\
    (\{0, 1\}^{k_i}[d_i])^{\mathit{op}} \arrow[r, hook, "I_i^{\mathit{op}}"]                                             & (\{0, 1\}^n)^{\mathit{op}}.            
    \end{tikzcd}
\end{equation*}

\subsection{Framing}
\label{subsec:framing}

Next we endow $\bar{\C}_{XY}(D)$ a framing relative to a sign assignment $s$. Recall from \Cref{prop:cube-flow-cat-framing} that, given a sign assignment $s$ for the $n$-cube $K(n)$, there is a framed neat embedding $(\iota, \varphi)$ of $\Cube{n}$ so that framed flow category $(\Cube{n}, \iota, \varphi)$ refines the cube complex $C_*(n, s)$. Here, we will not simply pull $(\iota, \varphi)$ back to $\bar{\C}_{XY}(D)$ by $\F$, because we had to take care of the negative sign appearing in $m(Y \otimes Y) = -Y$.

\begin{definition}
    Let $(D; x, y)$ be an index $k$ admissible decorated resolution configuration. Define a 1-cochain $t \in C^1(K(k); \FF_2)$ as follows: for each edge $e$ between vertices $u >_1 v$, assign $t(e) = 1$ if the corresponding surgery $D(v) \prec_1 D(u)$ merges two circles labeled $Y$, or $t(e) = 0$ otherwise. The 1-cochain $t$ is called the \textit{sign adjustment} for $(D; x, y)$.
\end{definition}

\begin{lemma}
    A sign adjustment is a cocycle.
\end{lemma}

\begin{proof}
    Let $(D; x, y)$ be an index $k$ admissible decorated resolution configuration. We have $P_{XY}(D; x, y) \isom \{0, 1\}^k$. Any 2-face $\sigma$ of $K(k)$ corresponds to a pair of arcs $a, b$ of $D$. It suffices to prove that the total number of merges of $Y$-labeled circles that occurs along the boundary $\partial \sigma$ is even. If the two arcs $a, b$ initially belongs to different components, the claim is obvious. Otherwise, the label must be constant on the same component and again the claim is obvious.
\end{proof}

Now let $D$ be an index $n$ resolution configuration and $s$ be any sign assignment for $K(n)$. Consider the cubic decomposition 
\[
    P_{XY}(D) \ \isom \ \bigsqcup_i P_i
\]
where each $P_i$ is a cube subposet of $P_{XY}(D)$ between objects $(D(u_i), x_i) \succeq_{k_i} (D(v_i), u_i)$. For each $i$, let
\[
    I_i: K(k_i) \hookrightarrow K(n)
\]
be the inclusion such that $I_i(\bar{0}) = v_i$ and $I_i(\bar{1}) = u_i$. Put $s_i = I_i^* s$, and let $t_i$ be the sign adjustment for $(D(v_i) \setminus D(u_i);\ \restr{x_i}{}, \restr{y_i}{})$. Put $s'_i = s_i + t_i$. Then 
\[
    \delta s'_i = \delta s_i + \delta t_i = 1
\]
so $s'_i$ another sign assignment on $K(k_i)$. Let $\C_i \isom \Cube{k_i}$ be the subcategory of $\bar{\C}_{XY}(D)$ that corresponds to the subposet $P_i$. From \Cref{prop:cube-flow-cat-framing} there is a framed neat embedding $(\iota_i, \varphi_i)$ of $\C_i$ relative to $s'_i$. Thus, 
\[
    (\iota, \varphi) = \bigsqcup_i (\iota_i, \varphi_i)
\]
gives a framed neat embedding of $\bar{\C}_{XY}(D)$ after some perturbation if necessary. 

\begin{proposition} \label{prop:C_BN_XY-isom-unnormalized}
    The above obtained framed flow category $(\bar{\C}_{XY}(D), \iota, \varphi)$ refines the unnormalized Bar-Natan complex $\bar{C}_\BN(D)$ by the correspondence
    \[
        C^*(\bar{\C}_{XY}(D)) \rightarrow \bar{C}_\BN(D),
        \quad 
        \x^* \mapsto \x
    \]
    where $\x^*$ is a generator of $C^*(\bar{\C}_{XY}(D))$ dual to an object $\x$.
\end{proposition} 

\begin{proof}
    Put $C^* = C^*(\bar{\C}_{XY}(D))$. The coboundary map $\delta$ of $C^*$ is given by
    \[
        \delta \y^* 
            = \sum_{\y \prec_1 \x} \# \M(\x, \y) \cdot \x^*
    \]
    where $\# \M(\x, \y)$ is the signed counting of the 0-dimensional moduli space $\M(\x, \y)$ which in this case is a single point. Take any object $\y$, and suppose $P_i$ is the subposet that contains $\y$. Take any $\x \succ_1 \y$. From \Cref{prop:cube-flow-cat-framing} we have
    \[
        \# \M(\x, \y) = (-1)^{(s_i + t_i)(\bar{e})}
    \]
    where $\bar{e}$ is an edge of $K(k_i)$ corresponding to $\y \prec_1 \x$ in $P_i$. For $s_i$ we have 
    \[
        s_i(\bar{e}) = s(e)
    \]
    where $e = I_i(\bar{e})$ is an edge of $K(n)$ corresponding to $\y \prec_1 \x$ in $P_{XY}(D)$. For $t_i$ we have $t_i(e) = 1$ if and only if $\y \prec_1 \x$ corresponds to a merge of two circles labeled $Y$. Thus we have
    \[
        \delta \y^* = \sum_{\y \prec_1 \x} (-1)^{s(e)} (-1)^{t_i(\bar{e})} \x^*.
    \]
    Comparing equations \eqref{eq:XY-operations} and the basic relation given in \Cref{fig:basic-rel-XY}, we see that $\delta$ is identical to the differential $d$ of $\bar{C}_\BN(D)$.
\end{proof}

\begin{definition}
    Let $D$ be a link diagram with $n^-$ negative crossings. With the above constructed framed flow category $\bar{\C}_{XY}(D)$, define 
    \[
        \C_{XY}(D) = \bar{\C}_{XY}(D)[-n^-].
    \]
    $\C_{XY}(D)$ is called the \textit{$XY$-based Bar-Natan flow category} of $D$, and its associated spectrum $\X_{XY}(D)$ is called the \textit{$XY$-based Bar-Natan spectrum} of $D$.
\end{definition}

\begin{remark}
    When $D$ is empty,  $\X_{XY}(\varnothing) = \SS$, the sphere spectrum.
\end{remark}

\begin{corollary}  \label{prop:C_BN_XY-isom}
    There is an isomorphism
    \[
        \tilde{C}^*(\X_{XY}(D)) \xrightarrow{\ \isom\ } C^*_\BN(D)
    \]
    that maps the dual cells of $\X_{XY}(D)$ to the $XY$-based generators of $C^*_\BN(D)$. \qed
\end{corollary}

\subsection{Determining the stable homotopy type}
\label{subsec:C_BN-stable-htpy-type}

On the algebraic level, the graded module structure of $H_{\BN}(D)$ is completely known: it is freely generated by the canonical classes $\{ [\ca(D, o)] \}_o$ (\Cref{prop:ab-gen,prop:alpha-homol-gr}). We lift this structure theorem to the spacial level.

\begin{definition}
\label{def:canon-objects}
    For each orientation $o$ on the underlying unoriented link diagram of $D$, \Cref{algo:ab-coloring} yields an $XY$-labeled resolution configuration $\x_\ca(D, o)$. The objects $\{\x_\ca(D, o)\}_o$ are called the \textit{canonical objects} of $\C_{XY}(D)$. The cells $\{\sigma_\ca(D, o)\}$ corresponding to the canonical cells are called the \textit{canonical cells} of $\X_{XY}(D)$.
\end{definition}

\begin{example}
    For the diagram $D$ given in \Cref{fig:poset-hopf-link}, the four isolated $XY$-labeled diagrams placed at the right are the canonical objects. 
\end{example}

\begin{proposition}
    The framed flow category $\C_{XY}(D)$ is move equivalent to the discrete category formed by the canonical objects.
\end{proposition}

\begin{remark}
    This is a flow category level analogue of  \cite[Remark 5.4]{Wehrli:2008}. A more detailed proof for the algebraic counterpart can be found \cite{Lewark:2009}, \cite{Turner:2020} and \cite{Sano:2020}.
\end{remark}

\begin{proof}
    Recall that $\C_{XY}(D)$ is constructed as a disjoint union of cube flow categories. From the basic relation it is obvious that there is no object that precedes or succeeds a canonical object. Furthermore, we can prove that the canonical objects are the only objects with such property. We claim that the remaining subcategories can be removed by handle cancellations, which follows from the following lemma.
\end{proof}

\begin{lemma} \label{lem:cube-contract}
    If $n > 0$, then the framed cube flow category $\Cube{n}$ is move equivalent to an empty category. 
\end{lemma}

\begin{proof}
    Pair all of objects of $\Cube{n}$ by
    \[
        v = (v_1, \cdots, v_{n - 1}, 0) <_1 u = (v_1, \cdots, v_{n - 1}, 1)
    \]
    Obviously $\M(u, v) = \{ \pt \}$ for each such pair $(u, v)$. By handle cancellation (\Cref{prop:handle-cancel}), the pair $(u, v)$ can be removed from $\Cube{n}$ with the effect of modifying other relevant moduli spaces. For any other pair $(u', v')$, the effect on $\M(u', v')$ is the attaching of the product
    \[
        \M(u, v') \times \M(u', v).
    \]
    For the two multiplicands to be non-empty, we must have $u > v'$ and $u' > v$, which implies
    \[
        |u| > |v'| = |u'| - 1,\ |u'| > |v| = |u| - 1.
    \]
    Thus $|u| = |u'|$ and $|v| = |v'|$, so we have 
    \[
        v <_1 u >_1 v'.
    \]
    This forces $v = v'$, which is a contradiction. Thus the product moduli space is empty, and $\M(u', v')$ remains unchanged. Therefore all pairs can be cancelled and we will be left with an empty category.
\end{proof}

\begin{corollary} \label{cor:X_BN-structure}
    $\X_{XY}(D)$ is stably homotopy equivalent to the wedge sum of the canonical cells
    \[
        \X_{XY}(D) \htpy \bigvee_{o} \sigma_\ca(D, o)
    \]
    where $o$ runs over all orientations on $D$. \qed
\end{corollary}

Since the homological gradings of the canonical objects (and the canonical cells) are given by \Cref{prop:alpha-homol-gr}, we see that the stable homotopy type of $\X_{XY}(D)$ is a link invariant, in particular independent of the auxiliary choices made for the construction.
Obviously, the dual cells $\sigma^*_\ca(L, o) \in \tilde{C}^*(\X_{XY}(L, s))$ correspond to the canonical cycles $\ca(L, o) \in C_\BN(L, s)$ under the identification of \Cref{prop:C_BN_XY-isom}. Thus by applying $\tilde{H}^*$ on both sides, we immediately recover the structure theorem for Bar-Natan homology
\[
    H_\BN(D) \isom \bigoplus_o \ZZ[\ca(D, o)].
\]

\begin{corollary}
\label{cor:disj-union-and-mirror}
    For link diagrams $D, D'$, there are stable homotopy equivalences:
    \begin{enumerate}
        \setlength{\itemsep}{0pt} 
        \item $\X_{XY}(D \sqcup D') \htpy \X_{XY}(D) \smash \X_{XY}(D')$.
        \item $\X_{XY}(m(D)) \htpy \X_{XY}(D)^{\vee}$.
    \end{enumerate}
    Here $m(\cdot)$ denotes the mirror, and $(\cdot)^{\vee}$ denotes the Spanier-Whitehead dual.
\end{corollary}

\begin{proof}
    \textit{1.} Let $X, X', X''$ be realizations of the Bar-Natan flow categories corresponding to $D, D'$ and $D \sqcup D'$ such that $\X_{XY}(D) = \Sigma^{-l}X $,\ $\X_{XY}(D') = \Sigma^{-l'}X'$ and $\X_{XY}(D \sqcup D') = \Sigma^{-l - l'}X''$ for some integers $l, l' > 0$. This is possible from the construction of the spectra, since $l, l'$ can be taken arbitrarily large. From \Cref{thm:2} it suffices to prove 
    \[
        \sigma_\ca(D \sqcup D', o\sqcup o') \homeo \sigma_\ca(D, o) \smash \sigma_\ca(D', o')
    \]
    for arbitrary orientations $o, o'$ of $D, D'$ respectively. Put $\sigma_\ca = \sigma_\ca(D, o),\ \sigma'_\ca = \sigma_\ca(D', o')$ and $\sigma''_\ca = \sigma_\ca(D \sqcup D', o \sqcup o')$. From \Cref{prop:alpha-homol-gr}, we have 
    \[
        \gr_h(\sigma_\ca) + \gr_h(\sigma'_\ca) = \gr_h(\sigma''_\ca)
    \]
    hence
    \[
        \dim(\sigma_\ca) + \dim(\sigma'_\ca) = (\gr_h(\sigma_\ca) + l) + (\gr_h(\sigma'_\ca) + l') = \dim(\sigma''_\ca)
    \]
    and $\sigma_\ca \smash \sigma'_\ca \homeo \sigma''_\ca$.
    
    \medskip
    
    \textit{2.} From \Cref{prop:alpha-homol-gr}, we have 
    \[
        \gr_h(\sigma_\ca(m(D), m(o))) = -\gr_h(\sigma_\ca(D, o)).
    \]
    The standard $S$-duality map $\SS^p \smash \SS^{-p} \rightarrow \SS^0$ gives the desired map.
\end{proof}

\begin{remark}
    The corresponding properties for Khovaov homotopy type are proved in \cite{LLS:2020} by introducing alternative constructions of the spectrum.
\end{remark}
    \section{Change of basis}
\label{sec:1X-basis}

Next, we recover the standard basis $\{1, X\}$ from $\{X, Y\}$ on the framed flow category level. Let us roughly consider how this should proceed, with the help of analogies in classical Morse theory. First, since we take the cochain complex of the Bar-Natan spectrum to get the Bar-Natan complex, the label set corresponding to a basis $S$ for $A$ should actually be identified with the dual basis of $S$ for $A^*$. Let $\{X^*, Y^*\}$ be the dual of $\{X,\ Y\}$, and $\{1', X'\}$ be the dual of $\{1, X\}$. From
\[
    (X,\ Y) = (X,\ 1)
    \begin{pmatrix}
        1 &  1 \\ 
        0 & -1
    \end{pmatrix},
\]
we have 
\[
    (X', 1') = (X^*,\ Y^*)
    \begin{pmatrix}
        1 &  0 \\ 
        1 & -1
    \end{pmatrix}.
\]
The transition from $\{X^*, Y^*\}$ to $\{X',\ 1'\}$ is realized by 
\begin{align*}
    \{X^*, Y^*\}
    \xrightarrow{\text{add $Y^*$ to $X^*$}} 
    &\{X^* + Y^*, Y^*\} \\
    \xrightarrow{\text{negate $Y^*$}}
    &\{X^* + Y^*, -Y^*\} = \{X', 1'\}.
\end{align*}
On the framed flow category level, adding $Y^*$ to $X^*$ should be realized by sliding an object labeled $X$ over $Y$. We describe this process by
\[
    X \dashrightarrow Y.
\]
The basis change for $(A^{\otimes k})^* \isom (A^*)^{\otimes k}$ can be performed by repeating this process for each tensor factor. On the framed flow category level, this should be realized by a $k$-times batch of parallel handle slides. The following diagram describes this process when $k = 2$. 
\begin{equation*}
\begin{tikzcd}[column sep={2cm,between origins}]
X \otimes Y \arrow[r, dashed]                            & Y \otimes Y                            & X' \otimes Y                          & Y \otimes Y \arrow[l, no head, dotted]                         & X' \otimes Y                             & Y \otimes Y \arrow[l, no head, dotted]                            \\
X \otimes X \arrow[r, dashed] \arrow[u, no head, dotted] & Y \otimes X \arrow[u, no head, dotted] & X' \otimes X \arrow[u, dashed] & Y\otimes X \arrow[l, no head, dotted] \arrow[u, dashed] & X' \otimes X' \arrow[u, no head, dotted] & Y\otimes X' \arrow[l, no head, dotted] \arrow[u, no head, dotted]
\end{tikzcd}
\end{equation*}
Finally, negating $Y^*$ should be realized by reversing the orientations of the corresponding cells. The above considerations are formalized as follows.

\subsection{$1X$-based Bar-Natan flow category} \label{subsec:basis-change}

Let $D$ be an index $n$ resolution configuration. For each $u \in \{0, 1\}^n$, put $r(u) = \#Z(D(u))$ and fix a total ordering of $Z(D(u))$ as $\{Z_i\}_{i=1}^{r(u)}$. Then any $XY$-labeling $x$ of $D(u)$ can be identified with a vertex $v \in \{0, 1\}^{r(u)}$ under the correspondence
\[
    v_i = 0 \ \Leftrightarrow \ x(Z_i) = X.
\]
Thus an object $\x = (D(u), x)$ is identified with a pair of vertices $(u, v)$, and $\Ob(\bar{\C}_{XY}(D))$ is identified with the union of the vertices of the subcubes each placed at a vertex of the $n$-cube. 

Take one $u \in \{0, 1\}^n$ and focus on the $r(u)$-subcube placed at $u$. Let $\{e_i\}$ be the standard basis of $\RR^{r(u)}$. For each $1 \leq i \leq r(u)$, perform $2^{r(u) - 1}$ positive handle slides parallel to $e_i$, so that each $\x = (u, v)$ with $v_i = 0$ is slid over $\x' = (u, v + e_i)$. In total we perform $r(u) \cdot 2^{r(u) - 1}$ handle slides. We call such procedure the \textit{cubic handle slide} performed at $u$. 

By performing cubic handle slides at every $u \in \{0, 1\}^n$ (in any order), we obtain a new framed flow category, which we temporarily denote by $\bar{\C}'_{XY}(D)$. Note that the object set remains unchanged from $\bar{\C}_{XY}(D)$. However it is convenient to regard an object $\x$ of $\bar{\C}'_{XY}(D)$ placed at $(u, v)$ as an $X'Y$-labeled resolution configuration $(D(u), x)$ under the correspondence
\[
    v_i = 0 \ \Leftrightarrow \ x(Z_i) = X'.
\]
In order to describe the $0$-dimensional moduli spaces of $\bar{\C}'_{XY}(D)$, let us introduce a few notations. 

Let $\x = (u, v),\ \y = (u', v')$ be objects in the original category $\bar{\C}_{XY}(D)$. Denote $\x \rightarrow \y$ if $\M(\x, \y)$ is non-empty, and call this a \textit{vertical arrow}. Denote $\x \dashrightarrow \y$ if $u = u'$ and $v < v'$, and call this a \textit{horizontal arrow}. The \textit{length} of a vertical (resp.\ horizontal arrow) is given by $|u| - |u'|$ (resp.\ $|v| - |v'|$). We also write $\rightarrow_k$ and $\dashrightarrow_l$ so that the subscript indicates the length of the arrow. A \textit{chain} is a sequence of vertical and horizontal arrows of $\bar{\C}_{XY}(D)$.

\begin{lemma}
\label{lem:0dim-chain}
    For any pair of objects $\x, \y$ of $\bar{\C}'_{XY}(D)$ with $|\x| = |\y| + 1$, each point of the 0-dimensional moduli space $\M'(\x, \y)$ corresponds bijectively to a chain in $\bar{\C}_{XY}(D)$ of the form
    \[
        \x \dashrightarrow_{l_1} \x' \rightarrow_1 \y' \dashrightarrow_{l_2} \y.
    \]
    Moreover, the sign of such point is given by $(-1)^{l_2}$ times the sign of $\x' \rightarrow_1 \y'$ where $l_2$ is the length of $\y' \dashrightarrow \y$. 
\end{lemma}

\begin{proof}
    This follows from the formula given in \Cref{prop:handle-slide}. First, the vertical arrow $\x' \rightarrow_1 \y'$ is pushed to $\x' \rightarrow_1 \y$ and then pulled to $\x \rightarrow_1 \y$ by the cubic handle slides.
\end{proof}

\begin{figure}[t]
    \centering
    \resizebox{0.9\textwidth}{!}{
\begin{tikzcd}
                                                               &                                       & Y \arrow[dd, "-", blue] &                                                                              &                                                                     & Y \arrow[lld, no head, dotted] \arrow[dd, "-", blue] \\
X \arrow[rru, dashed] \arrow[dd, red]                        &                                       &                   & X' \arrow[dd, red] \arrow[rrd, "-", bend left=25, blue] \arrow[rrd, "+"', bend left=15, red] \arrow[rdd, "-", red] \arrow[rd, "-", red] &                                                                     &                                                \\
                                                               & X \otimes Y \arrow[r, dashed]  & Y \otimes Y       &                                                                              & X' \otimes Y \arrow[ld, no head, dotted] \arrow[r, no head, dotted] & Y \otimes Y                                    \\
X \otimes X \arrow[r, dashed] \arrow[ru, dashed] & Y \otimes X \arrow[ru, dashed] &                   & X' \otimes X' \arrow[r, no head, dotted]                                     & Y \otimes X' \arrow[ru, no head, dotted]                            &                                               
\end{tikzcd}
    }
    \caption{When $D$ is a basic merge configuration}
    \label{fig:XY-to-X1-merge}
    
    \vspace{1em}
    
    \resizebox{0.9\textwidth}{!}{
\begin{tikzcd}
                                                                          & X \otimes Y \arrow[r, dashed]  & Y \otimes Y \arrow[dd, blue] &                                                                                    & X' \otimes Y \arrow[ld, no head, dotted] \arrow[r, no head, dotted] \arrow[rdd, blue] & Y \otimes Y \arrow[ld, no head, dotted] \arrow[dd, blue] \\
X \otimes X \arrow[r, dashed] \arrow[ru, dashed] \arrow[dd, red] & Y \otimes X \arrow[ru, dashed] &                        & X' \otimes X' \arrow[r, no head, dotted] \arrow[dd, red] \arrow[rrd, "+", bend right=5, blue] \arrow[rrd, "-", bend right=20, red] & Y \otimes X' \arrow[rd, blue]                                                         &                                                    \\
                                                                          &                                       & Y                      &                                                                                    &                                                                                 & Y \arrow[lld, no head, dotted]                     \\
X \arrow[rru, dashed]                                              &                                       &                        & X'                                                                                 &                                                                                 &                                                   
\end{tikzcd}
    }
    \caption{When $D$ is a basic split configuration}
    \label{fig:XY-to-X1-split}
\end{figure}

\begin{lemma}
\label{lem:index1-label-change-basic}
    Suppose $D$ is a basic index $1$ resolution configuration, which is either a connected merge configuration or a connected split configuration. The category $\bar{\C}'_{XY}(D)$ is given as in the right hand side of \Cref{fig:XY-to-X1-merge} or of \Cref{fig:XY-to-X1-split}, depending on whether $D$ is a merge or a split configuration.
\end{lemma}

\begin{proof}
    Straightforward from \Cref{prop:handle-slide}.
\end{proof}

\begin{lemma} \label{lem:index1-label-change}
    Suppose $D$ is an index $1$ resolution configuration, not necessarily basic. The category $\bar{\C}'_{XY}(D)$ is given by copies of arrows depicted in the right hand side of \Cref{fig:XY-to-X1-merge} or of \Cref{fig:XY-to-X1-split}, depending on whether $D$ is a merge or a split configuration, modulo pairs of arrows having opposite signs. Here the circles that are not involved in the surgery are assumed to have identical labels before and after the surgery.
\end{lemma}

\begin{proof}
    We prove when $D$ is a merge configuration. The proof for the other case proceeds similarly. Put $D = D_0 \sqcup D_1$ where $D_0$ is basic index $1$ merge configuration and $D_1$ is an index $0$ configuration with $k$ disjoint circles. Then any object of $\bar{C}_{XY}(D)$ can be identified with a triple $(u, v, w)$ where $(u, v) \in \{0\} \times \{0, 1\}^2 \union \{1\} \times \{0, 1\}$ and $w \in \{0, 1\}^k$. Before the handle slides, there are $2^k$ positive vertical arrows
    \[
        (1, 0, w) \xrightarrow{+} (0, (0, 0), w)
    \]
    and $2^k$ negative arrows
    \[
        (1, 1, w) \xrightarrow{-} (0, (1, 1), w).
    \]
    
    Take two objects $\x, \y$ with $|\x| = |\y| + 1$. Put $\x = (1, v, w),\ \y = (0, v', w')$ where $v \in \{0, 1\}$, $v' \in \{0, 1\}^2$ and $w, w' \in \{0, 1\}^k$. We may assume that there is a chain of the form of \Cref{lem:0dim-chain}. This forces $w \leq w'$. The case when $w = w'$ is done in \Cref{lem:index1-label-change-basic} so we assume $w < w'$. Take any $w'' \in \{0, 1\}^k$ such that $w <_1 w'' \leq w'$ and put $e = w'' - w$. The chains can be paired as
    \begin{equation*}
        \begin{tikzcd}
            \x \arrow[r, dashed] & \x' \arrow[r, dashed] \arrow[d] & \x'' \arrow[d] & \\
             & \y' \arrow[r, dashed] & \y'' \arrow[r, dashed] & \y
        \end{tikzcd}
    \end{equation*}
    where $\x'' - \x' = \y'' - \y' = e$. Obviously the two vertical arrows have equal signs. Thus from the formula of \Cref{lem:0dim-chain}, the resulting two arrows of $\bar{\C}'_{XY}(D)$ will have opposite signs. 
\end{proof}

\begin{proposition}
\label{prop:C_BN_1X-isom-flow-cat}
    The framed flow category $\bar{\C}'_{XY}(D)$ refines the unnormalized Bar-Natan complex $\bar{C}_\BN(D)$ by the correspondence
    \[
        C^*(\bar{\C}'_{XY}(D)) \rightarrow \bar{C}_\BN(D),
        \quad 
        \x^* \mapsto i(\x).
    \]
    Here $i$ is a map 
    \[
        i: \Ob(\bar{\C}'_{XY}(D)) \rightarrow \bar{C}_\BN(D)
    \]
    defined by extending the correspondence
    \[
        X' \mapsto X, \quad Y \mapsto -1
    \]
    multilinearly. 
\end{proposition}

\begin{proof}
    Take any edge $u >_1 u'$ of $K(n)$. The surgery $D(u') \prec_1 D(u)$ is described by an index $1$ resolution configuration, so the $0$-dimensional moduli spaces of $\bar{\C}'_{XY}(D)$ between the two subcubes placed at $u, u'$ are described as in \Cref{lem:index1-label-change}, up to an overall sign determined by the sign assignment $s(e_{u, u'})$. Now identify the chain complex associated to $\bar{\C}'_{XY}(D)$ with its image under the chain isomorphism induced by $i$. The differential $\partial$ of $C_*(\bar{\C}'_{XY}(D))$ is given by
    \begin{equation*}
\begin{tikzcd}[row sep={0.8cm,between origins}]
X \arrow[rd, bend right=20] \arrow[rdd, bend right] \arrow[r] & X \otimes X & X \otimes X \arrow[r]                    & X \\
                                                           & 1 \otimes X & 1 \otimes X \arrow[r]                    & 1 \\
                                                           & X \otimes 1 & X\otimes1 \arrow[ru, bend right=20]         &   \\
1 \arrow[r]                                                & 1 \otimes 1 & 1 \otimes 1 \arrow[ruu, "-", bend right] &  
\end{tikzcd}

    \end{equation*}
    If we reverse the arrows, then they are exactly the multiplication and the comultiplication of the defining Frobenius algebra of Bar-Natan complex (compare \eqref{eq:1X-operations-BN}). Thus the associated cochain complex of $\bar{\C}'_{XY}(D)$ is isomorphic to $\bar{C}_\BN(D)$.
\end{proof}

Hereinafter, we rewrite $\bar{\C}'_{XY}(D)$ to $\bar{\C}_\BN(D)$ and identify any object $\x$ of $\bar{\C}_\BN(D)$ placed at $(u, v)$ as an $1X$-labeled resolution configuration $(D(u), x)$ under the correspondence
\[
    v_i = 0 \ \Leftrightarrow \ x(Z_i) = X.
\]

\begin{definition}
    Let $D$ be a link diagram $D$ with $n^-$ negative crossings. With the above constructed framed flow category $\bar{\C}_\BN(D)$, define
    \[
        \C_\BN(D) = \bar{\C}_\BN(D)[-n^-].
    \]
    $\C_\BN(D)$ is called the \textit{($1X$-based) Bar-Natan flow category}, and its associated spectrum $\X_\BN(D)$ is called the \textit{($1X$-based) Bar-Natan spectrum} of $D$. Here the cells of $\X_\BN(D)$ are reoriented after the realization, so that the orientations are reversed for cells corresponding to objects $\x = (u, v)$ with $|v|$ odd.
\end{definition}

Thus we obtain:

\begin{theorem}
    There is an isomorphism
    \[
        \tilde{C}^*(\X_\BN(D)) \xrightarrow{\ \isom\ } C^*_\BN(D)
    \]
    that maps the dual cells of $\X_\BN(D)$ to the standard generators of $C^*_\BN(D)$. \qed
\end{theorem}

\begin{proposition}
\label{prop:XY-to-1X}
    Let
    \[
        f: \X_{XY}(D) \xrightarrow{\htpy} \X_\BN(D),
    \]
    be the composition of all of the stable homotopy equivalences corresponding to the handle slides performed on $\C_{XY}(D)$ to obtain $\C_\BN(D)$. Then the following diagram commutes:
    \begin{equation*}
        \begin{tikzcd}
            {\tilde{C}^*(\X_{XY}(D))} \arrow[rd, "\isom"'] & &
            {\tilde{C}^*(\X_\BN(D))} \arrow[ll, "f^*"'] \arrow[ld, "\isom"] \\
            & {C^*_\BN(D)} &
        \end{tikzcd}
    \end{equation*}
    Here left diagonal arrow is the isomorphism given in \Cref{prop:C_BN_XY-isom}, and the right diagonal arrow is the one given in \Cref{thm:1}.
\end{proposition}

\begin{proof}
    Consider the dual bases $\{X^*, Y^*\}$ of $\{X, Y\}$ and $\{X', 1'\}$ of $\{X, 1\}$. From the construction of $f$, we see that the following diagram commutes. 
    \begin{equation*}
        \begin{tikzcd}
        {\tilde{C}_*(\X_{XY}(D))} \arrow[r, "f_*"] \arrow[d, "\isom"] & {\tilde{C}_*(\X_\BN(D))} \arrow[d, "\isom"] \\
        {\bigoplus_u \ZZ\{X^*, Y^*\}^{\otimes r(u)}} \arrow[r, equal] & {\bigoplus_u \ZZ\{X', 1'\}^{\otimes r(u)}}    
        \end{tikzcd}
    \end{equation*}
    The result follows by duality.
\end{proof}

\begin{theorem}
    $\X_\BN(D)$ is stably homotopy equivalent to the wedge sum of the canonical cells
    \[
        \X_\BN(D) \htpy \bigvee_{o} \sigma_\ca(D, o)
    \]
    where $o$ runs over all orientations on $D$. In particular, the stable homotopy type of $\X_\BN(D)$ is a link invariant.
\end{theorem}

\begin{proof}
    Immediate from \Cref{prop:XY-to-1X} and \Cref{cor:X_BN-structure}.
\end{proof}

Thus the following definition is justified.

\begin{definition}
    Let $L$ be a link with diagram $D$. The stable homotopy type of $\X_\BN(D)$ is denoted $\X_\BN(L)$ and is called the \textit{Bar-Natan homotopy type} of $L$. 
\end{definition}

\subsection{Quantum grading}

\begin{definition}
    Let $D$ be a link diagram with $n^\pm$ positive / negative crossings. The \textit{quantum grading} on $\C_\BN(D)$
    \[
        \gr_q: \Ob(\C_\BN(D)) \rightarrow \ZZ
    \]
    is defined as
    \[
        \gr_q(\x) = |u| + 2|v| - r(u) + n^+ - 3n^-
    \]
    for each object $\x = (u, v)$.
\end{definition}

Note that this definition coincides with the one given in \Cref{def:homol-quantum-grading} under the correspondence of \Cref{prop:C_BN_1X-isom-flow-cat}. It is natural to expect that the quantum filtration on $C_\BN(D)$ also lifts to the flow category level. Recall that the (descending) filtration on $C_\BN$ comes from the fact that the differential $d$ is quantum grading non-decreasing. The corresponding condition on $\C_\BN(D)$ should be that all moduli spaces are quantum grading non-increasing, i.e.\ $\M(\x, \y)$ is non-empty if and only if $\gr_q(\x) \geq \gr_q(\y)$. In other words, there is a sequence of downward-closed subcategories
\[
    \varnothing = F_m \C \subset \cdots \subset F_j \C \subset F_{j + 2} \C  \subset \cdots \subset F_M \C = \C_\BN(D)
\]
where $F_j \C$ contains all of objects $\gr_q \leq j$. However this is not true, as we can see in \Cref{fig:XY-to-X1-merge} or in \Cref{fig:XY-to-X1-split}, where the diagonal arrows in pair increase the quantum grading by $2$. Nonetheless, for $0$-dimensional moduli spaces, they can be eliminated by the \textit{Whitney trick} (see \Cref{prop:whitney-trick}).

\begin{proposition}
\label{prop:eliminate-0dim-moduli-spaces}
    $\C = \C_\BN(D)$ is move equivalent to a framed flow category $\C'$ such that $\Ob(\C) = \Ob(\C')$ and for every object $\x, \y$ with $|\x| = |\y| + 1$, the 0-dimensional moduli space $\M(\x, \y)$ is non-empty if and only if $\gr_q(\x) \geq \gr_q(\y)$. 
\end{proposition}

\begin{proof}
    From \Cref{lem:index1-label-change}, by performing Whitney tricks on pairs of arrows with opposite signs, we will be left with arrows corresponding to the differential of $C_\BN(D)$, all of which are quantum grading non-increasing.
\end{proof}

In general, even after eliminating the 0-dimensional moduli spaces, there still remains higher dimensional moduli spaces that increase the quantum grading. Here are two examples.

\begin{example}
    Consider the resolution configuration $D$ of \Cref{fig:res-conf-ladybug} (\cpageref{fig:res-conf-ladybug}) which is called the \textit{ladybug configuration}.  \Cref{fig:slide-ladybug} depicts $\C_\BN(D)$, where the thick vertical arrows correspond to the $0$-dimensional moduli spaces of $\C_{XY}(D)$ and the thin ones are those that are produced by the handle slides. Before the handle slides, there are two $1$-dimensional moduli spaces $\M(X_{(YY)}, X_{(00)}) = I$ and $\M(Y_{(YY)}, Y_{(00)}) = I$, as depicted in the figure by red and blue double arcs. The handle slides in the top and bottom levels copy the two intervals, which produces $I \sqcup I$ in $\M(X_{(11)}, Y_{(00)})$. Moreover, the handle slides in the middle level produce four more $1$-dimensional moduli spaces in $\M(X_{(11)}, Y_{(00)})$. These are depicted by the gray double arcs. 
    
    Now there are four oppositely signed pairs of $0$-dimensional vertical arrows, namely in $\M(X_{(11)}, YY_{(10)})$, $\M(X_{(11)}, YY_{(01)})$, $\M(XX_{(10)}, Y_{(00)})$ and $\M(XX_{(01)}, Y_{(00)})$. By applying Whitney tricks, the six intervals in $\M(X_{(11)}, 1_{(00)})$ concatenate together to form two disjoint intervals $I \sqcup I$. We have $\gr_q(X_{(11)}) = \gr_q(1_{(00)}) = 1$ and the four endpoint objects $X1_{(10)}, X1_{(01)}, 1X_{(10)}, 1X_{(01)}$ also lies in $\gr_q = 1$. In this case, there are no other quantum grading preserving $1$-dimensional moduli spaces.
\end{example}

\begin{remark}
    The ladybug configuration is exceptional in the construction of Khovanov homotopy type, for it is the only index $2$ basic resolution configuration whose Khovanov flow category contains $I \sqcup I$ in its $1$-dimensional moduli space. 
\end{remark}

\begin{example} \label{eg:2}
    Let $D$ be the resolution configuration given in \Cref{fig:res-conf-hopf}. $\C_\BN(D)$ is depicted in \Cref{fig:slide-hopf}. Let us focus on the moduli space between $XX_{(11)}$ and $YY_{(00)}$ where $\gr_q(XX_{(11)}) = 0 < \gr_q(YY_{(11)}) = 2$. In $\C_\BN(D)$, we see that the $1$-dimensional moduli space $\M(XX_{(11)}, YY_{(00)})$ consists of four disjoint intervals. By performing the Whitney tricks on the four pairs of oppositely signed arrows, the intervals concatenate to form a single circle $C$. Thus in this case, there remains a quantum grading increasing $1$-dimensional moduli space after the Whitney tricks.
\end{example}

There are higher dimensional version of the Whitney trick, called the \textit{extended Whitney trick} (see \cite[Section 4]{LOS:2018}), that allows us to replace a moduli space of any dimension with a framed cobordant manifold rel boundary. We may expect that \Cref{prop:eliminate-0dim-moduli-spaces} can be generalized to all dimensions by use of extended Whitney tricks, and that the quantum filtration can be defined in the flow category level.

\begin{conjecture}
    The framed flow category $\C_\BN(D)$ can be modified by a sequence of extended Whitney tricks, possibly after altering the framing, so that there are no moduli spaces that increase the quantum grading. This gives a filtration on the associated spectra $\X$ as
    \[
        \{\pt\} = F_m \X \subset \cdots \subset F_j \X \subset F_{j + 2} \X  \subset \cdots \subset F_M \X = \X.
    \]
    The filtration $\{ F_j \X \}$ on $\X$ and the quantum filtration $\{ F^j C\}$ on $C = C_\BN(D)$ correspond as 
    \[
        \tilde{C}^*(F_j \X) = C / F^{j + 2} C, \quad
        \tilde{C}^*(\X / F_j \X) = F^{j + 2} C.
    \]
\end{conjecture}

Recall that the Khovanov complex and the Bar-Natan complex are related by $F^j C_\BN / F^{j - 2} C_\BN = C^{*, j}_\Kh$ for each $j \in 2\ZZ + |D|$. Our second conjecture is that this correspondence also lifts to the spatial level.

\begin{conjecture}
    For each $j \in 2\ZZ + |D|$, the quotient spectrum $F_j \X / F_{j - 2} \X$ is stably homotopy equivalent to the $j$-th wedge summand $\X^j_\Kh(D)$ of the Khovanov spectrum $\X_\Kh(D)$. 
\end{conjecture}

Further considerations are given in \Cref{subsec:quantum-filt}.

\newgeometry{margin=3cm}
\begin{figure}[p]
    \centering
    \begin{subfigure}{.4\textwidth}
        \centering
        \tikzset{every picture/.style={line width=0.75pt}} 

\begin{tikzpicture}[x=0.75pt,y=0.75pt,yscale=-1,xscale=1]

\draw  [line width=1.5]  (36,59.1) .. controls (36,46.26) and (46.41,35.85) .. (59.25,35.85) .. controls (72.09,35.85) and (82.5,46.26) .. (82.5,59.1) .. controls (82.5,71.94) and (72.09,82.35) .. (59.25,82.35) .. controls (46.41,82.35) and (36,71.94) .. (36,59.1) -- cycle ;
\draw [color={rgb, 255:red, 208; green, 2; blue, 27 }  ,draw opacity=1 ][line width=1.5]    (59.25,35.85) -- (59.25,82.35) ;
\draw  [draw opacity=0][line width=1.5]  (45.85,40.93) .. controls (45.25,39.35) and (44.92,37.64) .. (44.92,35.85) .. controls (44.92,27.93) and (51.33,21.52) .. (59.25,21.52) .. controls (67.17,21.52) and (73.58,27.93) .. (73.58,35.85) .. controls (73.58,36.71) and (73.51,37.55) .. (73.36,38.37) -- (59.25,35.85) -- cycle ; \draw  [color={rgb, 255:red, 208; green, 2; blue, 27 }  ,draw opacity=1 ][line width=1.5]  (45.85,40.93) .. controls (45.25,39.35) and (44.92,37.64) .. (44.92,35.85) .. controls (44.92,27.93) and (51.33,21.52) .. (59.25,21.52) .. controls (67.17,21.52) and (73.58,27.93) .. (73.58,35.85) .. controls (73.58,36.71) and (73.51,37.55) .. (73.36,38.37) ;

\end{tikzpicture}
        \caption{Ladybug configuration}
        \label{fig:res-conf-ladybug}
    \end{subfigure}
    \hspace{2em}
    \begin{subfigure}{.4\textwidth}
        \centering
        \vspace{0.5em}
        \tikzset{every picture/.style={line width=0.75pt}} 

\begin{tikzpicture}[x=0.75pt,y=0.75pt,yscale=-1,xscale=1]

\draw  [line width=1.5]  (36,45.1) .. controls (36,32.26) and (46.41,21.85) .. (59.25,21.85) .. controls (72.09,21.85) and (82.5,32.26) .. (82.5,45.1) .. controls (82.5,57.94) and (72.09,68.35) .. (59.25,68.35) .. controls (46.41,68.35) and (36,57.94) .. (36,45.1) -- cycle ;
\draw [color={rgb, 255:red, 208; green, 2; blue, 27 }  ,draw opacity=1 ][line width=1.5]    (97.5,38.35) -- (81.25,38.35) ;
\draw  [line width=1.5]  (96,45.1) .. controls (96,32.26) and (106.41,21.85) .. (119.25,21.85) .. controls (132.09,21.85) and (142.5,32.26) .. (142.5,45.1) .. controls (142.5,57.94) and (132.09,68.35) .. (119.25,68.35) .. controls (106.41,68.35) and (96,57.94) .. (96,45.1) -- cycle ;
\draw [color={rgb, 255:red, 208; green, 2; blue, 27 }  ,draw opacity=1 ][line width=1.5]    (98.5,51.35) -- (82.25,51.35) ;

\end{tikzpicture}
        \vspace{0.5em}
        \caption{Hopf link configuration}
        \label{fig:res-conf-hopf}
    \end{subfigure}
    \vspace{1em}
    \caption{Examples of basic index $1$ resolution configurations}
\end{figure}
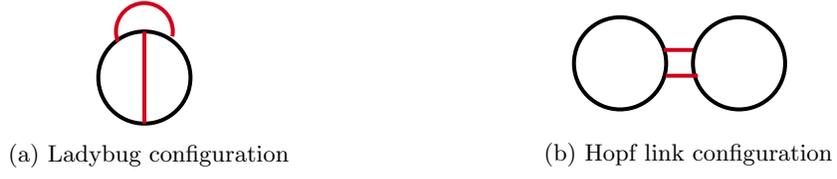
\begin{figure}[p]
    \centering
    \begin{subfigure}{0.35\textwidth}
        \resizebox{\textwidth}{!}{
        \input{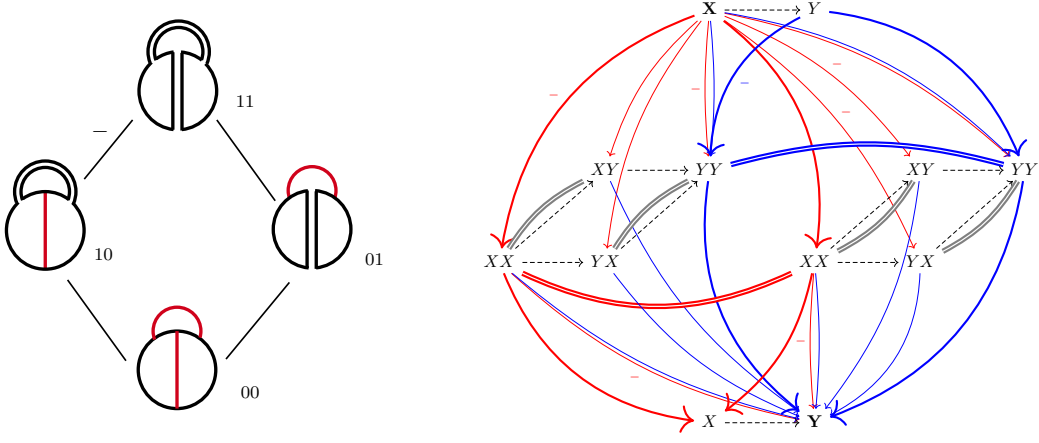}
        }
    \end{subfigure}
    \hspace{2em}
    \begin{subfigure}{0.5\textwidth}
        \resizebox{\textwidth}{!}{
\begin{tikzcd}[row sep=4em, column sep=large]
& & 
\mathbf{X} \arrow[llddd, bend right, red, line width=1.25, "-"'] \arrow[lddd, bend right=15, red] \arrow[ldd, bend right=15, red] \arrow[dd, bend right=5, red, "-"']  \arrow[dd, bend left=5, blue] \arrow[rddd, bend left, red, line width=1.25] \arrow[rrddd, bend left=15, red, "-"] \arrow[rrdd, bend left=15, red, "-"] \arrow[rrrdd, bend left=15, red] \arrow[rrrdd, bend left=17, blue] \arrow[r, dashed]& 
Y \arrow[rrdd, bend left, blue, line width=1.25] \arrow[ldd, bend right, blue, line width=1.25, "-"] & & \\
& & & & & \\
& 
XY \arrow[rrddd, bend right=15, blue] \arrow[r, dashed] & 
YY \arrow[rddd, bend right, blue, line width=1.25] \arrow[rrr, equal, bend left=15, blue, line width=1.25] & & 
XY \arrow[lddd, bend left=15, blue] \arrow[r, dashed]& 
YY \arrow[llddd, bend left, blue, line width=1.25] \\
XX \arrow[rrdd, bend right, red, line width=1.25] \arrow[rrrdd, bend right=15, red, "-"'] \arrow[rrrdd, bend right=13, blue] \arrow[r, dashed] \arrow[ru, dashed] \arrow[rrr, equal, bend right=25, red, line width=1.25] \arrow[ru, equal, bend left=15, gray, line width=1.25] & 
YX \arrow[rrdd, bend right=15, blue] \arrow[ru, dashed] \arrow[ru, equal, bend left=15, gray, line width=1.25]& & 
XX \arrow[ldd, bend left=20, red, line width=1.25] \arrow[dd, bend right=5, red, "-"'] \arrow[dd, bend left=5, blue] \arrow[ru, dashed] \arrow[r, dashed] \arrow[ru, equal, bend right=15, gray, line width=1.25] & 
YX \arrow[ldd, bend left, blue] \arrow[ru, dashed] \arrow[ru, equal, bend right=15, gray, line width=1.25] & \\
& & & & & \\
& & 
X \arrow[r, dashed] & 
\mathbf{Y} & & 
\end{tikzcd}
        }
    \end{subfigure}
    \vspace{1em}
    \caption{$\C_\BN$ for the ladybug configuration}
    \label{fig:slide-ladybug}
\end{figure}
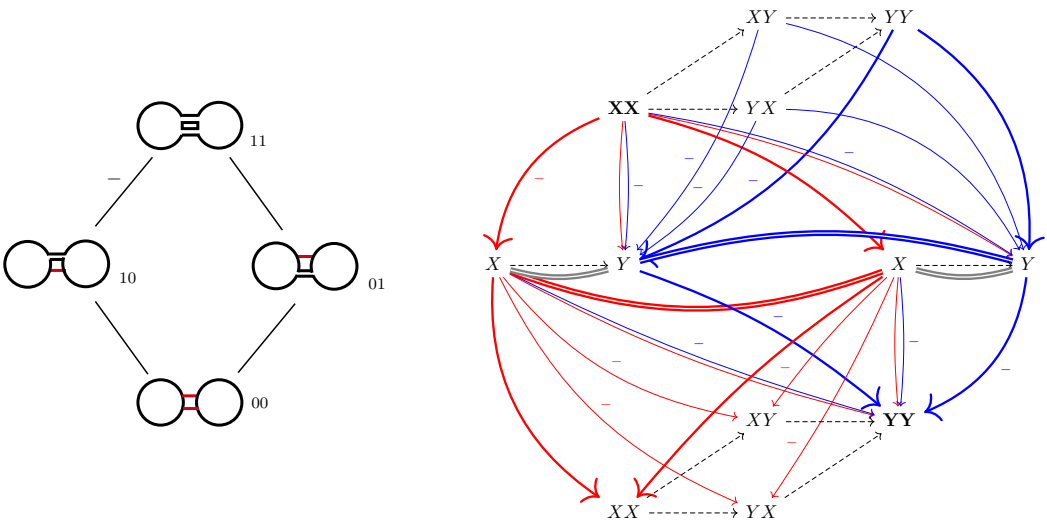
\begin{figure}[p]
    \centering
    \begin{subfigure}{0.35\textwidth}
        \resizebox{\textwidth}{!}{
        \tikzset{every picture/.style={line width=0.75pt}} 

\begin{tikzpicture}[x=0.75pt,y=0.75pt,yscale=-1,xscale=1]

\draw    (93.64,191.14) -- (128.64,239.64) ;
\draw    (183.14,95.14) -- (218.14,143.64) ;
\draw    (226.64,192.64) -- (188.64,238.64) ;
\draw  [line width=1.5]  (122,258.46) .. controls (122,250.11) and (128.77,243.35) .. (137.11,243.35) .. controls (145.46,243.35) and (152.22,250.11) .. (152.22,258.46) .. controls (152.22,266.8) and (145.46,273.57) .. (137.11,273.57) .. controls (128.77,273.57) and (122,266.8) .. (122,258.46) -- cycle ;
\draw [color={rgb, 255:red, 208; green, 2; blue, 27 }  ,draw opacity=1 ][line width=1.5]    (161.97,254.07) -- (151.41,254.07) ;
\draw  [line width=1.5]  (160.99,258.46) .. controls (160.99,250.11) and (167.76,243.35) .. (176.1,243.35) .. controls (184.45,243.35) and (191.21,250.11) .. (191.21,258.46) .. controls (191.21,266.8) and (184.45,273.57) .. (176.1,273.57) .. controls (167.76,273.57) and (160.99,266.8) .. (160.99,258.46) -- cycle ;
\draw [color={rgb, 255:red, 208; green, 2; blue, 27 }  ,draw opacity=1 ][line width=1.5]    (162.62,262.52) -- (152.06,262.52) ;

\draw [color={rgb, 255:red, 0; green, 0; blue, 0 }  ,draw opacity=1 ][line width=1.5]    (73.47,159.57) -- (62.91,159.57) ;
\draw [color={rgb, 255:red, 208; green, 2; blue, 27 }  ,draw opacity=1 ][line width=1.5]    (74.12,170.52) -- (63.56,170.52) ;
\draw [color={rgb, 255:red, 208; green, 2; blue, 27 }  ,draw opacity=1 ][line width=1.5]    (238.47,161.24) -- (227.91,161.24) ;
\draw [color={rgb, 255:red, 0; green, 0; blue, 0 }  ,draw opacity=1 ][line width=1.5]    (72.59,163.2) -- (63.25,163.2) ;
\draw  [draw opacity=0][line width=1.5]  (63.51,163.24) .. controls (63.74,164.28) and (63.85,165.35) .. (63.85,166.46) .. controls (63.85,174.88) and (57.03,181.7) .. (48.61,181.7) .. controls (40.19,181.7) and (33.37,174.88) .. (33.37,166.46) .. controls (33.37,158.04) and (40.19,151.21) .. (48.61,151.21) .. controls (54.92,151.21) and (60.34,155.05) .. (62.65,160.52) -- (48.61,166.46) -- cycle ; \draw  [line width=1.5]  (63.51,163.24) .. controls (63.74,164.28) and (63.85,165.35) .. (63.85,166.46) .. controls (63.85,174.88) and (57.03,181.7) .. (48.61,181.7) .. controls (40.19,181.7) and (33.37,174.88) .. (33.37,166.46) .. controls (33.37,158.04) and (40.19,151.21) .. (48.61,151.21) .. controls (54.92,151.21) and (60.34,155.05) .. (62.65,160.52) ;
\draw  [draw opacity=0][line width=1.5]  (72.59,161.88) .. controls (72.2,163.23) and (71.99,164.65) .. (71.99,166.13) .. controls (71.99,174.55) and (78.81,181.37) .. (87.23,181.37) .. controls (95.65,181.37) and (102.48,174.55) .. (102.48,166.13) .. controls (102.48,157.71) and (95.65,150.89) .. (87.23,150.89) .. controls (80.92,150.89) and (75.51,154.72) .. (73.19,160.19) -- (87.23,166.13) -- cycle ; \draw  [line width=1.5]  (72.59,161.88) .. controls (72.2,163.23) and (71.99,164.65) .. (71.99,166.13) .. controls (71.99,174.55) and (78.81,181.37) .. (87.23,181.37) .. controls (95.65,181.37) and (102.48,174.55) .. (102.48,166.13) .. controls (102.48,157.71) and (95.65,150.89) .. (87.23,150.89) .. controls (80.92,150.89) and (75.51,154.72) .. (73.19,160.19) ;
\draw [color={rgb, 255:red, 0; green, 0; blue, 0 }  ,draw opacity=1 ][line width=1.5]    (238.8,174.35) -- (228.24,174.35) ;
\draw [color={rgb, 255:red, 0; green, 0; blue, 0 }  ,draw opacity=1 ][line width=1.5]    (237.92,170.72) -- (228.59,170.72) ;
\draw  [draw opacity=0][line width=1.5]  (228.59,171.72) .. controls (228.98,170.37) and (229.19,168.94) .. (229.19,167.46) .. controls (229.19,159.04) and (222.36,152.22) .. (213.94,152.22) .. controls (205.52,152.22) and (198.7,159.04) .. (198.7,167.46) .. controls (198.7,175.88) and (205.52,182.71) .. (213.94,182.71) .. controls (220.25,182.71) and (225.67,178.87) .. (227.99,173.4) -- (213.94,167.46) -- cycle ; \draw  [line width=1.5]  (228.59,171.72) .. controls (228.98,170.37) and (229.19,168.94) .. (229.19,167.46) .. controls (229.19,159.04) and (222.36,152.22) .. (213.94,152.22) .. controls (205.52,152.22) and (198.7,159.04) .. (198.7,167.46) .. controls (198.7,175.88) and (205.52,182.71) .. (213.94,182.71) .. controls (220.25,182.71) and (225.67,178.87) .. (227.99,173.4) ;
\draw  [draw opacity=0][line width=1.5]  (237.92,172.05) .. controls (237.53,170.7) and (237.32,169.27) .. (237.32,167.79) .. controls (237.32,159.37) and (244.15,152.55) .. (252.57,152.55) .. controls (260.99,152.55) and (267.81,159.37) .. (267.81,167.79) .. controls (267.81,176.21) and (260.99,183.03) .. (252.57,183.03) .. controls (246.25,183.03) and (240.84,179.2) .. (238.52,173.73) -- (252.57,167.79) -- cycle ; \draw  [line width=1.5]  (237.92,172.05) .. controls (237.53,170.7) and (237.32,169.27) .. (237.32,167.79) .. controls (237.32,159.37) and (244.15,152.55) .. (252.57,152.55) .. controls (260.99,152.55) and (267.81,159.37) .. (267.81,167.79) .. controls (267.81,176.21) and (260.99,183.03) .. (252.57,183.03) .. controls (246.25,183.03) and (240.84,179.2) .. (238.52,173.73) ;
\draw [color={rgb, 255:red, 0; green, 0; blue, 0 }  ,draw opacity=1 ][line width=1.5]    (162.14,67.24) -- (151.57,67.24) ;
\draw  [draw opacity=0][line width=1.5]  (151.21,80.32) .. controls (148.84,85.65) and (143.49,89.37) .. (137.28,89.37) .. controls (128.86,89.37) and (122.03,82.54) .. (122.03,74.12) .. controls (122.03,65.71) and (128.86,58.88) .. (137.28,58.88) .. controls (143.59,58.88) and (149,62.72) .. (151.32,68.18) -- (137.28,74.12) -- cycle ; \draw  [line width=1.5]  (151.21,80.32) .. controls (148.84,85.65) and (143.49,89.37) .. (137.28,89.37) .. controls (128.86,89.37) and (122.03,82.54) .. (122.03,74.12) .. controls (122.03,65.71) and (128.86,58.88) .. (137.28,58.88) .. controls (143.59,58.88) and (149,62.72) .. (151.32,68.18) ;
\draw  [draw opacity=0][line width=1.5]  (162.1,80.29) .. controls (164.54,85.46) and (169.8,89.04) .. (175.9,89.04) .. controls (184.32,89.04) and (191.14,82.22) .. (191.14,73.8) .. controls (191.14,65.38) and (184.32,58.55) .. (175.9,58.55) .. controls (169.59,58.55) and (164.17,62.39) .. (161.86,67.86) -- (175.9,73.8) -- cycle ; \draw  [line width=1.5]  (162.1,80.29) .. controls (164.54,85.46) and (169.8,89.04) .. (175.9,89.04) .. controls (184.32,89.04) and (191.14,82.22) .. (191.14,73.8) .. controls (191.14,65.38) and (184.32,58.55) .. (175.9,58.55) .. controls (169.59,58.55) and (164.17,62.39) .. (161.86,67.86) ;
\draw [color={rgb, 255:red, 0; green, 0; blue, 0 }  ,draw opacity=1 ][line width=1.5]    (163.1,80.29) -- (151.21,80.32) ;
\draw  [line width=1.5]  (151.48,71.99) -- (161.86,71.99) -- (161.86,76.05) -- (151.48,76.05) -- cycle ;
\draw    (131.64,95.14) -- (93.64,141.14) ;

\draw (196,253.71) node [anchor=north west][inner sep=0.75pt]  [font=\footnotesize]  {$00$};
\draw (108,170.71) node [anchor=north west][inner sep=0.75pt]  [font=\footnotesize]  {$10$};
\draw (274,174.21) node [anchor=north west][inner sep=0.75pt]  [font=\footnotesize]  {$01$};
\draw (195,79.21) node [anchor=north west][inner sep=0.75pt]  [font=\footnotesize]  {$11$};
\draw (99.5,103.55) node [anchor=north west][inner sep=0.75pt]    {$-$};

\end{tikzpicture}
        }
    \end{subfigure}
    \hspace{2em}
    \begin{subfigure}{0.5\textwidth}
        \resizebox{\textwidth}{!}{
\begin{tikzcd}[row sep=3.5em, column sep=huge]
    & & 
    XY \arrow[r, dashed] \arrow[lddd, "-"', pos=0.6, bend left=12, blue] \arrow[rrddd, bend left, blue] & 
    YY \arrow[rddd, line width=1.25, bend left, blue] \arrow[llddd, "-"', pos=0.6, line width=1.25, bend left=20, blue] & \\
    & 
    \mathbf{XX} \arrow[ru, dashed] \arrow[r, dashed] \arrow[ldd, "-", line width=1.25, bend right, red] \arrow[rrdd, line width=1.25, bend left=15, red] \arrow[dd, bend right=5, red] \arrow[dd, "-", bend left=5, blue] \arrow[rrrdd, bend left=10, red] \arrow[rrrdd, "-", bend left=12, blue] & 
    YX \arrow[ru, dashed] \arrow[ldd, "-"', bend left=15, blue] \arrow[rrdd, bend left, blue] & & \\
    & & & & \\
    X \arrow[r, dashed] \arrow[rddd, line width=1.25, bend right, red] \arrow[rrddd, "-", bend right=20, red] \arrow[rrdd, "-", bend right=20, red] \arrow[rrrdd, bend right=7, red] \arrow[rrrdd, "-", bend right=5, blue] \arrow[rrr, equal, bend right=20, line width=1.25, red] \arrow[r, equal, bend right=15, line width=1.25, gray] & 
    Y \arrow[rrdd, "-", blue, line width=1.25, bend left=10] \arrow[rrr, equal, bend left=15, line width=1.25, blue] & & 
    X \arrow[llddd, line width=1.25, bend right=8, red] \arrow[r, dashed] \arrow[ldd, "-"', pos=0.8, bend right=5, red] \arrow[lddd, "-"', pos=0.8, bend left=5, red] \arrow[dd, bend right=5, red] \arrow[dd, "-", bend left=5, blue] \arrow[r, equal, bend right=15, gray, line width=1.25] & 
    Y \arrow[ldd, "-", line width=1.25, bend left, blue] \\
    & & & & \\
    & & 
    XY \arrow[r, dashed] & 
    \mathbf{YY} & \\
    & 
    XX \arrow[r, dashed] \arrow[ru, dashed] & 
    YX \arrow[ru, dashed] & & 
\end{tikzcd}
        }
    \end{subfigure}
    \vspace{1em}
    \caption{$\C_\BN$ for the Hopf link configuration}
    \label{fig:slide-hopf}
\end{figure}
\restoregeometry
    \section{Cobordism maps} \label{sec:cobordism}

A \textit{link cobordism} $S$ is a smooth, compact, oriented surface properly embedded in $\RR^3 \times [0, 1]$ with boundary links $\partial S = -L_0 \sqcup L_1$ in $\RR^3 \times \{0, 1\}$. We say $S$ is \textit{generic} if there are only finitely many $t \in [0, 1]$ such that $S \cap (\RR^3 \times \{t\})$ fails to be a link at one singular point. Under the projection $p: \RR^3 \rightarrow \RR^2$, a generic link cobordism can be represented by a finite sequence of local moves (i.e.\ Reidemeister moves and Morse moves) between the two boundary link diagrams. In this sense, we refer to \textit{generic cobordism between link diagrams}.

\begin{proposition}
\label{prop:cobordism}
    Let $S$ be a generic cobordism between link diagrams $D, D'$. There is a morphism between the corresponding spectra
    \[
        f_S: \X_\BN(D) \rightarrow \X_\BN(D')
    \]
    such that its induced map on Bar-Natan homology coincides with the standard cobordism map
    \[
        F_{\bar{S}}: H_\BN(D') \rightarrow H_\BN(D)
    \]
    associated to the reversed cobordism $\bar{S}: D' \rightarrow D$.
\end{proposition}

\begin{remark}
    In \cite{LS:2014_rasmussen} the cobordism map is defined in the opposite direction, which is $f_{\bar{S}}: \X_\BN(D') \rightarrow \X_\BN(D)$ in our notation. 
\end{remark}

\subsection{Reidemsiter moves}

\begin{proposition} \label{prop:reidemeister-moves}
    Suppose $D$ and $D'$ are two link diagrams related by one of the Reidemeister moves. Then there is a stable homotopy equivalence
    \[
        f: \X_\BN(D) \rightarrow \X_\BN(D')
    \]
    such that the induced isomorphism on Bar-Natan homology coincides with the isomorphism given in \cite[Section 4]{Bar-Natan:2005} for the reversed move.
\end{proposition}

The proof basically follows \cite[Section 6]{LS:2014}. The following \Cref{lem:cancel-acyclic} will play an essential role, which is a flow category level analogue of the cancellation principle given in \cite[pp.\ 1028-1029]{LS:2014}. First we give necessary terminologies.

\begin{definition}
    A \textit{leaf} of a resolution configuration $D$ is a circle in $Z(D)$ that intersects with only one arc at a single point. A \textit{coleaf} of $D$ is an arc in $A(D)$ that joins two points on a single circle, and at least one of the two circular arcs between the two endpoints do not intersect any other arc. 
\end{definition}

\begin{lemma} \label{lem:cancel-acyclic}
    Let $D$ be an index $n$ resolution configuration. Suppose $\{0, 1\}^n$ splits as $\{0, 1\}^k \times \{0, 1\}^{n - k}$ and there are vertices $u, v \in \{0, 1\}^k$ with $u >_1 v$ such that either one of the following holds:
    \begin{itemize}
        \setlength{\itemsep}{0pt} 
        \item $D(v, \bar{0})$ contains a leaf that intersects the arc represented by $u >_1 v$, or 
        \item $D(v, \bar{0})$ contains a coleaf represented by $u >_1 v$.
    \end{itemize}
    Let $U$ be the circle that is merged into or splitted off by the surgery $(v, \bar{0}) \rightarrow (u, \bar{0})$. Let $A, B$ be sets of objects in $\C = \C_\BN(D)$ defined as
    \begin{itemize}
        \setlength{\itemsep}{0pt} 
        \item $A = \{\ (D(u, w), x)\ \}$,\ $B = \{\ (D(v, w), y);\ y(U) = 1\ \}$,\ or 
        \item $A = \{\ (D(u, w), x);\ x(U) = X\ \},\ B = \{\ (D(v, w), y)\ \}$
    \end{itemize}
    accordingly. Then there is a complete pairing between $A$ and $B$ such that each pair $(\x, \y) \in A \times B$ satisfies $\M(\x, \y) = \{\pt\}$, and all such pairs can be removed one-by-one by handle cancellations. Moreover, if the full subcategory $\C_1$ of $\C$ spanned by $A, B$ is downward- (resp.\ upward-) closed, then the cancelled category $\C_2$ is merely the complement of $\C_1$ in $\C$ and is upward- (resp.\ downward-) closed in $\C$.
\end{lemma}

\begin{proof}
    We prove the statement for the first case. Note that for any $w \in \{0, 1\}^{n - k}$, $Z(D(v, w))$ can be identified with $Z(D(u, w)) \sqcup U$. We formally write $x \otimes x'$ for the labeling of $Z(D(v, w))$, where $x$ is the labeling of $Z(D(u, w))$ and $x'$ is the labeling of $U$. Pair each $\x = (D(u, w), x) \in A$ with $\y = (D(v, w), x \otimes 1) \in B$. Then it is obvious that $\M_{\C}(\x, \y) = \{ \pt \}$. By an argument similar to the proof of \Cref{lem:cube-contract}, one can prove that the cancellation of $(\x, \y)$ does not affect the moduli space $\M_{\C}(\x', \y')$ of any other pair $(\x', \y')$. Thus all pairs can be cancelled one-by-one. For the latter statement, suppose $\C_1$ is downward-closed. Take any pair of objects $(\z, \w)$ both not in $\C_1$. Then for any object $\x$ in $\C_1$, we have by definition $\M_{\C}(\x, \z) = \varnothing$, so the moduli space $\M_{\C}(\z, \w)$ remains unchanged by the cancellation of $\C_1$. Thus $\C_2$ is the complement of $\C_1$ in $\C$, and is obviously upward-closed.
\end{proof}

\begin{proof}[Proof of \Cref{prop:reidemeister-moves}] 
    The general strategy is as follows. Suppose \Cref{lem:cancel-acyclic} can be applied to $\C(D')$ twice, first by cancelling an upward-closed subcategory, then by cancelling a downward-closed subcategory. Moreover suppose that the resulting category is identical to $\C(D)$. Describe this process as 
    \[
        \C(D')
            \ \xhookleftarrow{d.c.}\ 
        \C_1
            \ \xhookleftarrow{u.c.}\ 
        \C_2 = \C(D).
    \]  
    Then on the spectra we obtain homotopy equivalences
    \[
        \X(D')
            \ \xhookleftarrow{\ i \ }\ 
        \X_1
            \ \xrightarrowdbl{\ p \ }\ 
        \X_2 = \X(D)
    \]
    where $i$ is the inclusion map, and $p$ is the quotient map. We define the stable homotopy equivalence  
    \[
        f: \X(D) \rightarrow \X(D')
    \]
    by the composition of $i$ and the homotopy inverse $q$ of $p$. Passing to cochain complexes, we obtain a sequence of quasi-isomorphisms
    \[
        C(D')
            \ \xrightarrowdbl{\ i^* \ }\ 
        C_1
            \ \xhookleftarrow{\ p^* \ }\ 
        C_2 = C(D).
    \]
    Now, \cite{Bar-Natan:2005} gives for each Reidemeister move $D \rightarrow D'$ an explicit chain homotopy equivalence as
    \begin{equation*}
        \begin{tikzcd}
            C(D) \arrow[r, "F", shift left] & C(D'). \arrow[l, "G", shift left]
        \end{tikzcd}
    \end{equation*}
    We will prove the following diagram commutes
    \begin{equation} \label{eq:RM-proof-commutativity}
\begin{tikzcd}
C(D') \arrow[r, "i^*"] \arrow[rrd, "G", bend right=15] & C_1 \arrow[rd, "\bar{G}"] & C_2 \arrow[d, equal] \arrow[l, "p^*"'] \\
                                                               &                                               & C(D)                                                             
\end{tikzcd}
    \end{equation}
    where $\bar{G}$ is a map induced from $G$. Then we have 
    \[
        f^* = (iq)^* = (p^*)^{-1} i^* = G_*
    \]
    as desired.

\subsubsection*{R-I}
    \begin{figure}[t]
        \centering
        \begin{subfigure}{0.45\textwidth}
            \centering
            \resizebox{\textwidth}{!}{
                \tikzset{every picture/.style={line width=0.75pt}} 

\begin{tikzpicture}[x=0.75pt,y=0.75pt,yscale=-1,xscale=1]

\draw  [dash pattern={on 0.84pt off 2.51pt}] (29,68.3) .. controls (29,49.91) and (43.91,35) .. (62.3,35) .. controls (80.68,35) and (95.59,49.91) .. (95.59,68.3) .. controls (95.59,86.68) and (80.68,101.59) .. (62.3,101.59) .. controls (43.91,101.59) and (29,86.68) .. (29,68.3) -- cycle ;
\draw    (115.57,69.63) -- (152.95,69.63) ;
\draw [shift={(154.95,69.63)}, rotate = 180] [color={rgb, 255:red, 0; green, 0; blue, 0 }  ][line width=0.75]    (10.93,-3.29) .. controls (6.95,-1.4) and (3.31,-0.3) .. (0,0) .. controls (3.31,0.3) and (6.95,1.4) .. (10.93,3.29)   ;
\draw  [dash pattern={on 0.84pt off 2.51pt}] (174.17,68.3) .. controls (174.17,49.91) and (189.08,35) .. (207.47,35) .. controls (225.85,35) and (240.76,49.91) .. (240.76,68.3) .. controls (240.76,86.68) and (225.85,101.59) .. (207.47,101.59) .. controls (189.08,101.59) and (174.17,86.68) .. (174.17,68.3) -- cycle ;
\draw  [draw opacity=0] (42.93,42.27) .. controls (52.76,47.19) and (59.44,57.13) .. (59.35,68.55) .. controls (59.25,80.1) and (52.24,90.03) .. (42.15,94.72) -- (29,68.3) -- cycle ; \draw   (42.93,42.27) .. controls (52.76,47.19) and (59.44,57.13) .. (59.35,68.55) .. controls (59.25,80.1) and (52.24,90.03) .. (42.15,94.72) ;
\draw    (187.25,43.84) .. controls (197.61,71.68) and (231.23,98.95) .. (230.18,67.25) ;
\draw  [draw opacity=0][fill={rgb, 255:red, 255; green, 255; blue, 255 }  ,fill opacity=1 ] (203.1,62.31) .. controls (207.07,62.26) and (210.34,65.44) .. (210.39,69.42) .. controls (210.44,73.39) and (207.26,76.66) .. (203.29,76.71) .. controls (199.31,76.76) and (196.04,73.58) .. (195.99,69.6) .. controls (195.94,65.63) and (199.12,62.36) .. (203.1,62.31) -- cycle ;
\draw  [draw opacity=0][fill={rgb, 255:red, 0; green, 0; blue, 0 }  ,fill opacity=1 ] (202.95,67.05) .. controls (204.22,67.03) and (205.26,68.05) .. (205.28,69.32) .. controls (205.29,70.59) and (204.28,71.64) .. (203.01,71.65) .. controls (201.74,71.67) and (200.69,70.65) .. (200.68,69.38) .. controls (200.66,68.11) and (201.68,67.07) .. (202.95,67.05) -- cycle ;

\draw    (230.18,67.25) .. controls (229.13,35.55) and (198.87,66.33) .. (189.25,94.43) ;

\end{tikzpicture}
            }
            \caption{R-I}
            \label{fig:RM1}
        \end{subfigure}
        \hspace{1em}
        \begin{subfigure}{0.45\textwidth}
            \centering
            \resizebox{\textwidth}{!}{
            \tikzset{every picture/.style={line width=0.75pt}} 

\begin{tikzpicture}[x=0.75pt,y=0.75pt,yscale=-1,xscale=1]

\draw  [dash pattern={on 0.84pt off 2.51pt}] (29,68.3) .. controls (29,49.91) and (43.91,35) .. (62.3,35) .. controls (80.68,35) and (95.59,49.91) .. (95.59,68.3) .. controls (95.59,86.68) and (80.68,101.59) .. (62.3,101.59) .. controls (43.91,101.59) and (29,86.68) .. (29,68.3) -- cycle ;
\draw    (115.57,69.63) -- (152.95,69.63) ;
\draw [shift={(154.95,69.63)}, rotate = 180] [color={rgb, 255:red, 0; green, 0; blue, 0 }  ][line width=0.75]    (10.93,-3.29) .. controls (6.95,-1.4) and (3.31,-0.3) .. (0,0) .. controls (3.31,0.3) and (6.95,1.4) .. (10.93,3.29)   ;
\draw  [dash pattern={on 0.84pt off 2.51pt}] (174.17,68.3) .. controls (174.17,49.91) and (189.08,35) .. (207.47,35) .. controls (225.85,35) and (240.76,49.91) .. (240.76,68.3) .. controls (240.76,86.68) and (225.85,101.59) .. (207.47,101.59) .. controls (189.08,101.59) and (174.17,86.68) .. (174.17,68.3) -- cycle ;
\draw  [draw opacity=0] (38.14,88.86) .. controls (43.44,82.97) and (52.3,79.12) .. (62.34,79.15) .. controls (72.54,79.17) and (81.51,83.17) .. (86.74,89.24) -- (62.3,101.59) -- cycle ; \draw   (38.14,88.86) .. controls (43.44,82.97) and (52.3,79.12) .. (62.34,79.15) .. controls (72.54,79.17) and (81.51,83.17) .. (86.74,89.24) ;
\draw  [draw opacity=0] (86.75,47.72) .. controls (81.52,53.97) and (72.56,58.11) .. (62.38,58.14) .. controls (52.03,58.18) and (42.93,53.98) .. (37.69,47.61) -- (62.3,35) -- cycle ; \draw   (86.75,47.72) .. controls (81.52,53.97) and (72.56,58.11) .. (62.38,58.14) .. controls (52.03,58.18) and (42.93,53.98) .. (37.69,47.61) ;
\draw  [draw opacity=0] (235.72,48.64) .. controls (231.17,61.93) and (220.62,71.2) .. (208.43,71.07) .. controls (196.13,70.95) and (185.69,61.29) .. (181.49,47.71) -- (208.81,34.18) -- cycle ; \draw   (235.72,48.64) .. controls (231.17,61.93) and (220.62,71.2) .. (208.43,71.07) .. controls (196.13,70.95) and (185.69,61.29) .. (181.49,47.71) ;
\draw  [draw opacity=0][fill={rgb, 255:red, 255; green, 255; blue, 255 }  ,fill opacity=1 ] (230.58,66.06) .. controls (230.55,70.04) and (227.3,73.24) .. (223.32,73.2) .. controls (219.35,73.17) and (216.15,69.92) .. (216.18,65.95) .. controls (216.21,61.97) and (219.46,58.77) .. (223.44,58.8) .. controls (227.42,58.84) and (230.61,62.09) .. (230.58,66.06) -- cycle ;
\draw  [draw opacity=0][fill={rgb, 255:red, 0; green, 0; blue, 0 }  ,fill opacity=1 ] (225.84,65.81) .. controls (225.83,67.08) and (224.8,68.1) .. (223.53,68.09) .. controls (222.25,68.08) and (221.23,67.05) .. (221.24,65.78) .. controls (221.25,64.5) and (222.29,63.48) .. (223.56,63.49) .. controls (224.83,63.5) and (225.85,64.54) .. (225.84,65.81) -- cycle ;

\draw  [draw opacity=0][fill={rgb, 255:red, 255; green, 255; blue, 255 }  ,fill opacity=1 ] (199.59,64.81) .. controls (199.56,68.79) and (196.31,71.99) .. (192.33,71.96) .. controls (188.36,71.92) and (185.16,68.67) .. (185.19,64.7) .. controls (185.22,60.72) and (188.47,57.52) .. (192.45,57.56) .. controls (196.42,57.59) and (199.62,60.84) .. (199.59,64.81) -- cycle ;
\draw  [draw opacity=0][fill={rgb, 255:red, 0; green, 0; blue, 0 }  ,fill opacity=1 ] (194.85,64.56) .. controls (194.84,65.83) and (193.81,66.86) .. (192.53,66.85) .. controls (191.26,66.84) and (190.24,65.8) .. (190.25,64.53) .. controls (190.26,63.26) and (191.3,62.23) .. (192.57,62.24) .. controls (193.84,62.25) and (194.86,63.29) .. (194.85,64.56) -- cycle ;

\draw  [draw opacity=0] (179.77,84.77) .. controls (184.03,69.45) and (194.77,58.59) .. (207.3,58.65) .. controls (219.94,58.7) and (230.66,69.85) .. (234.7,85.42) -- (207.11,99.27) -- cycle ; \draw   (179.77,84.77) .. controls (184.03,69.45) and (194.77,58.59) .. (207.3,58.65) .. controls (219.94,58.7) and (230.66,69.85) .. (234.7,85.42) ;

\end{tikzpicture}
            }
            \caption{R-II}
            \label{fig:RM2}
        \end{subfigure}
        \caption{R-I and R-II moves}
    \end{figure}
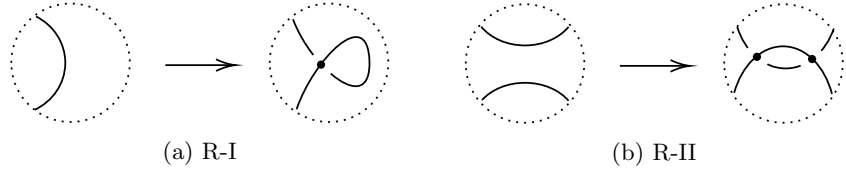

    Suppose $D$, $D'$ are related as in \Cref{fig:RM1}. Order the crossings of $D, D'$ so that $A(D') = \{a\} \cup A(D)$, where $a$ is the arc corresponding to the added crossing. Let $U$ be the leaf of $D'(0, \bar{0})$. Note that $Z(D'(0, u)) = Z(D(u)) \sqcup U$ and $Z(D'(1, u)) = Z(D(u))$ for any $u \in \{0, 1\}^n$. Consider pairs of objects $(\x, \y)$ in $\C(D')$ of the form
    \[
        \x = (D'(1, u), x) \ \succ_1 \ \y = (D'(0, u), x \otimes 1).
    \]
    From \Cref{lem:cancel-acyclic}, all such pairs can be cancelled, and the resulting flow category $\C_1$ is a downward-closed subcategory of $\C(D')$. Moreover $\C_1$ is identical to $\C(D)$ under the correspondence of objects
    \[
         (D'(0, u), x \otimes X)
         \ \leftrightarrow \ 
         (D(u), x).
    \]
    Thus we have a homotopy equivalence
    \[
        i: \X(D)
            \ \xhookrightarrow{\ \ \ }\ 
        \X(D').
    \]
    The induced map on the cochain complexes maps $x \otimes X \in C^*(D')$ to $x \in C^*(D)$ and other generators to $0$. This is exactly the chain homotopy equivalence $G$ for the R-I move given in \cite{Bar-Natan:2005}. Thus the commutativity of diagram \eqref{eq:RM-proof-commutativity} (with $C_1 = C_2$) is proved.
    
    \subsubsection*{R-II}
    
    Suppose $D$, $D'$ are related as in \Cref{fig:RM2}.  Order the crossings of $D, D'$ so that $A(D') = \{a, b\} \cup A(D)$, where $a, b$ are the arcs corresponding to the added crossings, aligned in this order. Let $U$ be the circle that will be splitted from $D'$ by the $(0, 1)$-resolution on $(a, b)$. Note that $Z(D'(01u)) = D(u) \sqcup U$ for any $u \in \{0, 1\}^n$. Using \Cref{lem:cancel-acyclic}, cancel all pairs $(\x, \y)$ of objects in $\C(D')$ of the form
    \[
        \x = (D'(11u), x) \ \succ_1 \ \y = (D'(01u), x \otimes 1),
    \]
    and then all pairs of the form
    \[
        \x' = (D'(01u), x \otimes X) \ \succ_1 \ \y' = (D'(00u), x).
    \]
    Let $\C_1$ and $\C_2$ be the result of the first and the second cancellations respectively. $\C_1$ is a downward-closed subcategory of $\C(D')$, and $\C_2$ is an upward-closed subcategory of $\C_1$. Moreover $\C_2$ is identical to $\C(D)$ under the correspondence
    \[
        (D'(10u), x)
            \ \leftrightarrow \ 
        (D(u), x).
    \]
    From the explicit description of the chain homotopy equivalence $G$ for the R-II move given in \cite{Bar-Natan:2005}, we see that $G$ annihilates the objects cancelled in the first cancellation, thus induces $\bar{G}$. Moreover $G$ restricts to identity on $C(D)$. Thus the commutativity of \eqref{eq:RM-proof-commutativity} is proved.
    
    \subsubsection*{R-III}
    
    \begin{figure}[t]
    \centering
        \input{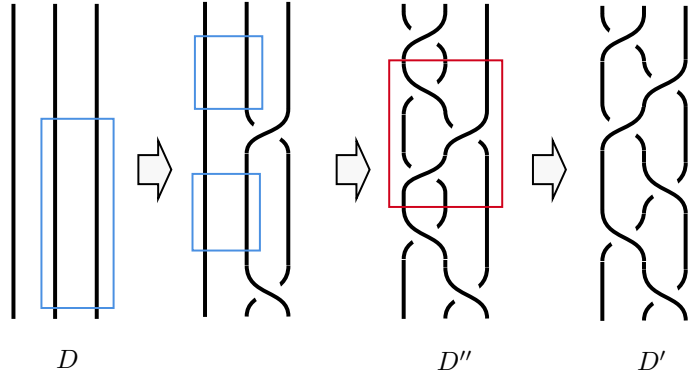}
        \caption{Generating a braid-like R-III move}
        \label{fig:R3-bR3}
    \end{figure}
    
    Suppose $D$ and $D'$ are related by a braid-like R-III move, as depicted in \Cref{fig:R3-bR3}. One can check that the procedure of cancellation given in \cite[Proposition 6.4]{LS:2014} can be applied verbatim to our case. This gives homotopy equivalences
    \[
        \X(D')
            \ \xhookleftarrow{\ i \ }\ 
        \X_1
            \ \xrightarrowdbl{\ p \ }\ 
        \X_2 = \X(D).
    \]
    Here, $\X_1$ is the result of cancelling the top half upward-closed subcategory, and $\X_2$ is the result of cancelling the bottom half downward-closed subcategory of \cite[Figure 6.4]{LS:2014}. Now consider the sequence of ordinary Reidemeister moves that gives the braid-like R-III move:
    \[
        D \xrightarrow{\text{(R-II)}^3}
        D'' \xrightarrow{\text{R-III}} 
        D'.
    \]
    Let $G$ be the composition of the chain homotopy equivalences corresponding to the moves $D' \rightarrow D'' \rightarrow D$. From the explicit description of $G$, one can manually check the commutativity of the diagram \eqref{eq:RM-proof-commutativity}.
\end{proof}

\begin{remark}
    Note that the cancelling pairs $(\x, \y)$ of \Cref{lem:cancel-acyclic} have $\gr_q(\x) = \gr_q(\y)$. Thus if \Cref{conj:1} is true, the chain homotopy equivalences of \Cref{prop:reidemeister-moves} should restrict to each subspectrum of the filtration, implying that $\X_\BN(L)$ is an invariant as a filtered CW-spectrum.
\end{remark}

\subsection{Morse moves}

\begin{proposition} \label{prop:morse-move}
    Suppose $D$ and $D'$ are two link diagrams related by one of the Morse moves. Then there is a morphism 
    \[
        f: \X_\BN(D) \rightarrow \X_\BN(D')
    \]
    such that the induced homomorphism on Bar-Natan homology coincides with the standard homomorphism for the reversed move.
\end{proposition}
    
\begin{proof} 
    The proof of \cite[Section 3.3]{LS:2014_rasmussen} works verbatim. Suppose $D, D'$ are related by a cup move, so that $D' = D \sqcup U$ with a circle $U$. Then $\C(D')$ is the disjoint union of two subcategories $\C_\pm$ where $\C_+$ (resp.\ $\C_-$) consists of objects with $U$ labeled $1$ (resp.\ $X$). Both subcategories are obviously isomorphic to $\C(D)$. The associated spectra is given as $\X(D') = \X_+ \vee \X_-$. The cup map is given by the inclusion
    \[
        f_\cup:\ \X(D) \isom \X_- \hookrightarrow \X(D'),
    \]
    and the cap map is given by the projection
    \[
        f_\cap:\ \X(D') \rightarrowdbl \X_+ \isom \X(D).
    \]
    Next, suppose $D$ and $D'$ are related by a saddle move. Then there is a diagram $D''$ such that $D, D'$ are obtained by a $1$-, $0$-resolution of a crossing of $D''$ respectively. We may regard $\C(D')$ as a downward-closed subcategory of $\C(D'')$, and $\C(D)[1]$ as the complement of $\C(D')$ in $\C(D'')$,
    \[
        \C(D')
            \ \xhookrightarrow{d.c.}\ 
        \C(D'')
            \ \xhookleftarrow{u.c.}\ 
        \C(D)[1].
    \]
    Thus on the associated spectra we have a cofibration sequence 
    \[
        \X(D') \hookrightarrow
        \X(D'') \rightarrowdbl
        \Sigma\X(D).
    \]
    The saddle map $f$ is defined as the Puppe map
    \[
        f: \Sigma \X(D) \rightarrow \Sigma \X(D').
    \]
    It is straightforward to check that these maps induce the standard homomorphisms on Bar-Natan homology.
\end{proof}

\begin{remark}
\label{rem:cobordism-map-filtered}
    Again if \Cref{conj:1} is true, the above constructed maps should be filtered maps of degree $-\chi(S)$, i.e.\ for each $j$, $f_S$ restricts to
    \[
        f_S: F^j \X(L) \rightarrow F^{j - \chi(S)} \X(L').
    \]
\end{remark}

\begin{proof}[Proof of \Cref{prop:cobordism}]
    Decompose $S$ into elementary cobordisms $S_l \circ \cdots \circ S_1$, each of which corresponds to a Reidemeister move or a Morse move. Define 
    \[
        f_S = f_{S_l} \circ \cdots \circ f_{S_1}.
    \]
    The claimed property is obvious from \Cref{prop:reidemeister-moves,prop:morse-move}.
\end{proof}

\begin{remark}
\label{rem:cobordism-degree}
    Under the assumption of \Cref{prop:cobordism}, for each pair of canonical cells $\sigma$ of $\X_\BN(D)$ and $\sigma'$ of $\X_\BN(D')$, consider the composition
    \[
        \SS^k \xhookrightarrow{i_\sigma} \X_\BN(D) \xrightarrow{f_S} \X_\BN(D') \xrightarrowdbl{p'_\sigma} \SS^{k'}
    \]
    where $k, k'$ are the (stable) dimensions of $\sigma, \sigma'$, and $i_\sigma, p'_\sigma$ are the obvious inclusion, projection maps for $\sigma, \sigma'$ respectively. This map gives an element of the stable homotopy group $\pi_{k - k'}(\SS)$, which is possibly non-trivial when $k \geq k'$. For pairs with $k > k'$, this might capture some information of $S$ that cannot be seen in the homology groups. The case $k = k'$ is considered later in \Cref{prop:canon-coh-cobordism}. 
\end{remark}

Recently Lawson, Lipshitz and Sarkar \cite{LLS:2021-functoriality} proved the functoriality of Khovanov homotopy type, i.e.\ for a cobordism $S$ between links $L, L'$, there is a map
\[
    f_S: \X^j_\Kh(L) \rightarrow \X^{j - \chi(S)}_\Kh(L')
\]
well-defined up to sign, and is functorial with respect to composition of cobordisms. The well-definedness follows by proving that the map defined on the level of diagrams is invariant under isotopies of $S$ rel boundary.

\begin{conjecture}
    For a cobordism $S$ between links $L, L'$, there is a map
    \[
        f_S: \X_\BN(L) \rightarrow \X_\BN(L')
    \]
    well-defined up to sign, and is functorial with respect to composition of cobordisms. Moreover if \Cref{conj:1} is true, $f_S$ is a filtered map of degree $-\chi(S)$.
\end{conjecture}
    \section{Mirror and duality}

So far we have considered the cellular \textit{cochain} complex of $\X_\BN$. Here we give dual statements by considering the cellular \textit{chain} complex of $\X_\BN$. Let us first recall how a link diagram $D$ and its mirror $m(D)$ are related in Bar-Natan homology. 

Let $A = (A, m, \iota, \Delta, \epsilon)$ be a Frobenius algebra over a commutative ring $R$. The \textit{dual Frobenius algebra} of $A$ is defined by $A^* = (A^*, \Delta^*, \epsilon^*, m^*, \iota^*)$ where $A^* = \Hom_R(A, R)$ and other maps are the dual maps. The \textit{twisting} of $A$ by an invertible element $\theta \in R$ is another Frobenius algebra $A_\theta = (A, m, \iota, \Delta_\theta, \epsilon_\theta)$ whose algebra structure is the same as $A$ but the coalgebra structure is twisted as
\[
    \Delta_\theta(x) = \Delta(\theta^{-1}x),
    \quad 
    \epsilon_\theta(x) = \epsilon(\theta x).
\]
The defining Frobenius algebra $A_{1, 0} = \ZZ[X]/(X^2 - X)$ for Bar-Natan homology is \textit{almost self-dual}, i.e.\ $A_{1, 0}$ is isomorphic to $A_{-1, 0}^*$ by the correspondence
\[
    1 \leftrightarrow X^*,\quad
    X \leftrightarrow 1^*
\]
where $\{1^*, X^*\}$ is the dual of the basis $\{1, X\}$ for $A_{-1, 0}$. Moreover there is an isomorphism from $A_{-1, 0}$ to $A_{1, 0; -1}$ (the (-1)-twisting of $A_{1, 0}$), induced from the underlying algebra isomorphism:
\[
    \ZZ[X]/(X^2 + X) \rightarrow \ZZ[X]/(X^2 - X), \quad 
    X \mapsto -X.
\]

Frobenius algebra isomorphisms and twistings induce chain isomorphisms on the complexes (\cite[Proposition 3]{Khovanov:2004}). Thus we get a sequence of isomorphisms
\[
    C_{1, 0}(m(D)) 
    \isom (C_{-1, 0}(D))^*
    \isom (C_{1, 0; -1}(D))^*
    \isom (C_{1, 0}(D))^*
\]
where the first isomorphism is the canonical identification (see \cite[Section 7.3]{Khovanov:2000}), the second induced from the Frobenius algebra isomorphism, and the third induced from the twisting. The composition 
\[
    \Phi: C_\BN(m(D)) \rightarrow (C_\BN(D))^*
\]
is explicitly given by
\[
    ({m(D)}(u), x) \ \longmapsto \ \epsilon(u) (D(\bar{1} - u), \varphi(x))^*
\]
where $\epsilon(u) \in \{ \pm 1 \}$ are signs determined by the twisting, and $\varphi$ is a map defined by $\varphi(1) = -X$ and $\varphi(X) = 1$.

\begin{proposition} \label{prop:homology-mirror}
    There is a canonical isomorphism
    \[
        \tilde{C}_*(\X_\BN(D)) \isom C^{-*}_\BN(m(D)).
    \]
\end{proposition}

\begin{proof}
    The isomorphism is given by the composition
    \[
        \tilde{C}_*(\X_\BN(D)) 
            \xrightarrow{{}^{**}}
        (\tilde{C}^*(\X_\BN(D)))^*
            =
        (C^*_\BN(D))^*
            \xrightarrow{{\Phi}^{-1}}
        C^*_\BN(m(D)).
    \]
    One can also check that the homological grading is reversed.
\end{proof}

\begin{proposition} \label{prop:cobordism-mirror}
    Let $S$ be a generic link cobordism between link diagrams $D, D'$. Let
    \[
        f_S: \X_\BN(D) \rightarrow \X_\BN(D')
    \]
    be the cobordism map given in \Cref{prop:cobordism}. The induced map on reduced homology
    \[
        (f_S)_*: H_\BN(m(D)) \rightarrow H_\BN(m(D'))
    \]
    coincides (up to sign) with the cobordism map on Bar-Natan homology induced from the mirrored cobordism $m(S): m(D) \rightarrow m(D')$.
\end{proposition}

\begin{proof}
    First consider the Reidemeister moves. In the proof of \Cref{prop:reidemeister-moves}, for each Reidemeister move $D \rightarrow D'$ we considered a sequence of homotopy equivalences
    \[
        \X(D')
            \ \xhookleftarrow{\ i \ }\ 
        \X_1
            \ \xrightarrowdbl{\ p \ }\ 
        \X_2 = \X(D)
    \]
    and defined $f_S = iq$ where $q$ is the homotopy inverse of $p$. Passing to reduced cellular chain complexes, we get 
    \[
        C(m(D'))
            \ \xhookleftarrow{\ i_* \ }\ 
        C_1
            \ \xrightarrowdbl{\ p_* \ }\ 
        C_2 \isom C(m(D)).
    \]
    Consider the standard chain homotopy equivalences
    \begin{equation*}
        \begin{tikzcd}
            C(m(D)) \arrow[r, "F", shift left] & C(m(D')). \arrow[l, "G", shift left]
        \end{tikzcd}
    \end{equation*}
    associated to the move $m(D) \rightarrow m(D')$. Consider the following diagram
    \begin{equation*}
\begin{tikzcd}
C(m(D')) \arrow[rrd, "G", bend right=15] & C_1 \arrow[l, "i_*"', hook'] \arrow[r, "p_*", two heads] \arrow[rd, "G|"] & C_2 \arrow[d, equal] \\
                                      &                                                                                                         & C(m(D))                                      
\end{tikzcd}
    \end{equation*}
    With the explicit isomorphism given in \Cref{prop:homology-mirror}, one can prove as in the proof of \Cref{prop:reidemeister-moves} that the above diagram commutes up to sign. The proof for the Morse moves is straightforward. 
\end{proof}

Finally we give another conjecture regarding duality for the conjectured quantum filtration. On the algebraic level, mirroring also respects the quantum filtration. That is, with the descending filtration $\{F^jC(D)\}$ of $C_\BN(D)$ and the cofiltration $\{ F_jC(m(D)) = C(m(D)) / F^{j + 2}C(m(D))\}$ of $C_\BN(m(D))$, there is a canonical isomorphism
\[
    F^j C(D) \isom (F_{-j} C(m(D)))^*
\]
as asserted in the proof of \cite[Proposition 3.10]{Rasmussen:2010}. If \Cref{conj:1} is true, we may also expect that this lifts to the spatial level. 

\begin{conjecture}
    Assuming \Cref{conj:1} is true, the $S$-duality map of \Cref{cor:disj-union-and-mirror}
    \[
        \mu: \X_\BN(D) \smash \X_\BN(m(D)) \rightarrow \SS
    \]
    also induces an $S$-duality map
    \[
        \mu: F_j \X_\BN(D) \smash F^{-j}\X_\BN(m(D)) \rightarrow \SS
    \]
    for each $j \in 2\ZZ + |D|$. Here, $\{F^j\X\}$ denotes the cofiltration $\{\X / F_{j - 2}\X\}$. 
\end{conjecture}
    \section{Canonical cohomotopy classes}

The canonical classes of Bar-Natan homology play an essential role in proving the important properties of the $s$-invariant. We define the cohomotopical counterpart in the (stable) cohomology group of $\X_\BN$, by tracing the arguments \cite[Section 5]{LNS:2015}. In the following, cohomotopy classes are denoted by brackets $[\cdot]$ and cohomology classes are denoted by double brackets $\llbracket \cdot \rrbracket$. 

\begin{definition} \label{def:coho-alpha}
    Let $D$ be a link diagram. Each canonical cell $\sigma_\ca(D, o)$ of (stable) dimension $k$, let 
    \[
        p_\ca(D, o): \X_\BN(D) \rightarrowdbl \SS^k
    \]
    be the umkehr map of $\sigma_\ca(D, o)$, i.e.\ $p_\ca(D, o)$ maps $\sigma_\ca(D, o)$ homeomorphically onto $\SS^k$ and collapses other cells to the basepoint. We call the (stable) cohomotopy class
    \[
        [p_\ca(D, o)] \in \pi^k(\X_\BN(D))
    \]
    the \textit{canonical cohomotopy class} of $D$ for $o$.
\end{definition}

Recall that there is a dual Hurewicz map for a CW-complex $X$
\[
    h': \pi^k(X) \rightarrow H^k(X)
\]
for each $k$, defined on the level of cells (see \cite[Section V.11]{Bredon:1993}). The \textit{Hopf classification theorem} asserts that it is an isomorphism when $k = \dim(X)$.

\begin{proposition} 
\label{prop:coHur-ca}
    The co-Hurewicz map
    \[
        h': \pi^k(\X_\BN(D)) \rightarrow H^k_\BN(D)
    \]
    sends the canonical cohomotopy class $[p_\ca(D, o)]$ to the canonical homology class $\llbracket \ca(D, o) \rrbracket$.
\end{proposition}

\begin{proof}
    Recall from \Cref{subsec:C_BN-stable-htpy-type} that the canonical class $\llbracket \ca \rrbracket$ is represented by the dual $\sigma_\ca^*$ of the canonical cell $\sigma_\ca$. For any cell $\sigma$ with attaching map $\varphi$, consider the following diagram:
    \begin{equation*}
        \begin{tikzcd}
        D^n \arrow[r, "\varphi"] \arrow[d, "/\partial"', two heads] & X^n \arrow[d, "p_\alpha"] \\
        S^n \arrow[r, dashed]                         & S^n                      
        \end{tikzcd}
    \end{equation*}
    The degree of the dashed arrow is $1$ if $\sigma = \sigma_\ca$ and $0$ otherwise. Thus $h'(p_\ca)$ is exactly $\sigma_\ca^*$.
\end{proof}

\begin{corollary}
    If $D$ is a knot diagram, then $\pi^0(\X_\BN(D))$ is freely generated by the two canonical cohomotopy classes
    \[
        \pi^0(\X_\BN(D)) \isom \ZZ\{[p_\ca], [p_\cb]\}
    \]
\end{corollary}

\begin{proof}
    The two canonical cells are concentrated in dimension $0$, so $h'$ is an isomorphism in grading $0$. 
\end{proof}

The following is a special case of the problem considered in \Cref{rem:cobordism-degree}.

\begin{proposition} 
\label{prop:canon-coh-cobordism}
    Let $S$ be a generic connected cobordism between knot diagrams $D, D'$. Let
    \[
        f_S: \X_\BN(D) \rightarrow \X_\BN(D')
    \]
    be the cobordism map of \Cref{prop:cobordism}. Let $\sigma_\ca, \sigma'_\ca$ be the canonical cells of $\X_\BN(D), \X_\BN(D')$ corresponding to the given orientations of $D, D'$, and $\sigma_\cb, \sigma'_\cb$ be the one corresponding to the reversed orientations respectively. For any $\sigma \in \{\sigma_\ca, \sigma_\cb\}$ and $\sigma' \in \{\sigma'_\ca, \sigma'_\cb\}$, consider the following composition
    \[
        g: \SS 
            \xhookrightarrow{i_\sigma} 
        \X_\BN(D) 
            \xrightarrow{\ f_S\ } 
        \X_\BN(D') 
            \xrightarrowdbl{p'_\sigma} 
        \SS
    \]
    where $i_\sigma, p'_\sigma$ are the obvious inclusion and the projection maps. The composition $g$ has degree $\pm 1$ if either $(\sigma, \sigma') = (\sigma_\ca, \sigma'_\ca)$ or $(\sigma_\cb, \sigma'_\cb)$, and degree $0$ otherwise.
\end{proposition}

\begin{proof}
    From the naturality of the co-Hurewicz map, we obtain a commutative diagram:
    \begin{equation*}
        \begin{tikzcd}
            \ZZ \arrow[d, equal] & \pi^0(\X_\BN(D)) \arrow[l, "i_\sigma*"'] \arrow[d, "h'"] & \pi^0(\X_\BN(D')) \arrow[l, "{f_S}^*"'] \arrow[d, "h'"] & \ZZ \arrow[d, equal] \arrow[l, "p_{\sigma'}^*"'] \\
            \ZZ                                & H^0_\BN(D) \arrow[l, "i_\sigma*"']                       & H^0_\BN(D') \arrow[l, "{f_S}^*"']                 & \ZZ \arrow[l, "p_{\sigma'}^*"']                               
        \end{tikzcd}
    \end{equation*}
    The top row is exactly $\deg(g)$. From \Cref{prop:cobordism,prop:coHur-ca} the bottom row is given by the pairing $\abrac{ {f_{\bar{S}}}_*\llbracket \ca' \rrbracket, \llbracket \ca \rrbracket}$. 
    The claimed result follows from \cite[Proposition 3.4]{Sano:2020-b} with $c = 1$. 
\end{proof}

\begin{remark}
    \Cref{prop:canon-coh-cobordism} can be generalized as in the form of \cite[Proposition 4.1]{Rasmussen:2010}, so that $S$ is a link cobordism, not necessarily connected but without closed components.
\end{remark}

Thus we obtain a refinement of \cite[Proposition 2.4]{LS:2014_rasmussen}:

\begin{corollary}
    If $S$ is a generic connected cobordism between knot diagrams $D, D'$, the cobordism map 
    \[
        f_S: \X_\BN(D) \rightarrow \X_\BN(D')
    \]
    is a stable homotopy equivalence. \qed
\end{corollary}

\begin{corollary} 
\label{cor:canon-invariant}
    For a knot diagram $D$, the canonical cohomotopy classes $[p_\ca], [p_\cb]$ are invariant (up to sign) under connected cobordisms. \qed
\end{corollary}

\begin{remark}
    In \cite{Sano:2020-b}, we proved that the sign ambiguity of the cobordism maps on Khovanov homology and its variants can be fixed uniformly by adjusting the signs of the elementary cobordism maps. By applying the same adjustment for the cobordism map of \Cref{prop:cobordism} (which is possible by suspending $f$ and reflecting the added coordinate), the map $g$ will have degree $1$ for $(\sigma, \sigma') = (\sigma_\ca, \sigma'_\ca)$. Also for $(\sigma, \sigma') = (\sigma_\cb, \sigma'_\cb)$ the degree can be described explicitly as $(-1)^j$ with
    \[
        j = \frac{\delta w(D, D') - \delta r(D, D') - \chi(S)}{2},
    \]
    where $\delta w(D, D')$ is difference of the writhes, $\delta r(D, D')$ is the difference of the number of Seifert circles, and $\chi(S)$ is the Euler number of $S$. In particular, if we restrict to transverse knot diagrams and transverse Markov moves, we always have $j = 0$. Thus the canonical cohomotopy classes are also invariants (without sign indeterminacy) of transverse knots. 
\end{remark}

\begin{remark}
    Dually, we can define the \textit{canonical homotopy classes} 
    \[
        [i_\ca(D, o)] \in \pi_k(\X_\BN(D))
    \]
    from the inclusion maps
    \[
        i_\ca(D, o): \SS^k \hookrightarrow \X_\BN(D)
    \]
    corresponding to the canonical cells $\sigma_\ca(D, o)$ of $D$. Results dual to \Cref{prop:coHur-ca} and \Cref{cor:canon-invariant} can be proved similarly. Moreover, one can prove that the canonical cohomotopy classes and the canonical homotopy classes are dual under the $S$-duality map
    \[
        \mu: \X_\BN(D) \smash \X_\BN(m(D)) \rightarrow \SS
    \]
    of \Cref{cor:disj-union-and-mirror}. 
\end{remark}

    \section{Future prospects} \label{sec:future}

\subsection{Higher dimensional moduli spaces}
\label{subsec:quantum-filt}

To prove \Cref{conj:1} it would be necessary to analyze how the higher dimensional moduli spaces are produced by the handle slides. The following lemma generalizes \Cref{lem:0dim-chain}. 

\begin{lemma}
\label{lem:k-dim-chain}
    Let $D$ be a resolution configuration. For any pair of objects $\x, \y$ of $\bar{\C}_\BN(D)$ with $|\x| = |\y| + k$, the $(k - 1)$-dimensional moduli space $\M_\BN(\x, \y)$ is given by the following: consider all chains in $\bar{\C}_{XY}(D)$ of the form
    \[
        \x \dashrightarrow 
        \x_1 \rightarrow \y_1
        \dashrightarrow 
        \x_2 \rightarrow \y_2
        \dashrightarrow \cdots \dashrightarrow 
        \x_m \rightarrow \y_m
        \dashrightarrow \y.
    \]
    For each such chain, take $2^{(l_1 - 1) + (l_2 - 1) + \cdots + (l_{m - 1} - 1)}$ copies of
    \[
        \M(\x_m, \y_m) \times I \times \cdots \times I \times \M(\x_2, \y_2) \times I \times \M(\x_1, \y_1).
    \]
    Here, $l_i\ (1 \leq i \leq m - 1)$ is the length of the horizontal arrow $\y_i \dashrightarrow \x_{i + 1}$. Then $\M_\BN(\x, \y)$ is the disjoint union of all such manifolds. 
\end{lemma}

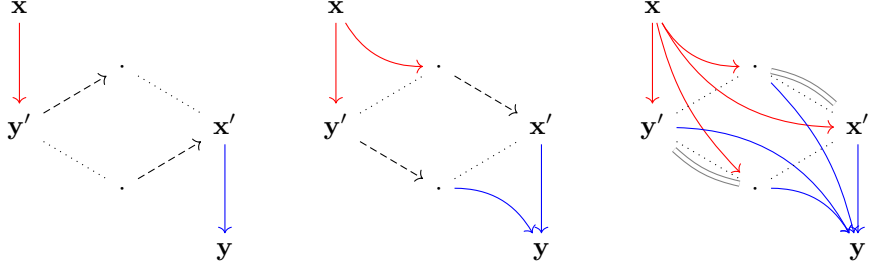
\begin{figure}
    \centering
    \resizebox{\textwidth}{!}{
\begin{tikzcd}[row sep={0.8cm,between origins}]
\x \arrow[dd, red]                             &                          &               & \x \arrow[dd, red] \arrow[rd, bend right, red]               &                                                            &               & \x \arrow[dd, red] \arrow[rd, bend right, red] \arrow[rrdd, bend right, red] \arrow[rddd, bend right=15, red]                                          &                                                                                                       &               \\
                                         & \cdot \arrow[rd, no head, dotted] &               &                                                   & \cdot \arrow[rd, dashed]                                   &               &                                                                                                                                & \cdot \arrow[rd, no head, dotted] \arrow[rddd, bend left=15, blue] \arrow[rd, equal, bend left=15, gray] &               \\
\y' \arrow[ru, dashed] \arrow[rd, no head, dotted] &                          & \x'  \arrow[dd, blue] & \y' \arrow[ru, no head, dotted] \arrow[rd, dashed] &                                                            & \x'  \arrow[dd, blue] & \y' \arrow[ru, no head, dotted] \arrow[rd, no head, dotted] \arrow[rrdd, bend left=30, blue] \arrow[rd, equal, bend right=15, gray] &                                                                                                       & \x'  \arrow[dd, blue] \\
                                         & \cdot \arrow[ru, dashed] &               &                                                   & \cdot \arrow[ru, no head, dotted] \arrow[rd, bend left, blue] &               &                                                                                                                                & \cdot \arrow[ru, no head, dotted] \arrow[rd, bend left, blue]                                               &               \\
                                         &                          & \y              &                                                   &                                                            & \y              &                                                                                                                                &                                                                                                       & \y             
\end{tikzcd}
    }
    \caption{Production of intervals}
    \label{fig:handle-slide-high-dim}
\end{figure}

\begin{proof}
    Again this follow from \Cref{prop:handle-slide}. Take one chain of the above form. The handle slides at the two ends $\x \dashrightarrow \x_1$ and $\y_m \dashrightarrow \y$ merely produces a copy of the intermediate moduli space, so we assume $\x = \x_1$ and $\y_m = \y$. For simplicity we also consider the case $m = 1$ where the chain is given by  
    \[
        \x \rightarrow \y'
        \dashrightarrow_l 
        \x' \rightarrow \y.
    \]
    Note that horizontal paths from $\y'$ to $\x'$ form an $l$-dimensional subcube. Up to the $(l - 1)$-th step of the parallel handle slides for the relevant directions, the moduli spaces $\M(\x, \y')$ and $\M(\x', \y)$ are copied along the vertices, while they do not meet at any of the vertices. Finally at the $l$-th step, the $2^{l - 1}$ pairs each connect to produce a moduli space
    \[
        \M(\x', \y) \times I \times \M(\x, \y').
    \]
    \Cref{fig:handle-slide-high-dim} depicts the case for $l = 2$. 
\end{proof}

The following observation might be useful in the induction step of eliminating higher dimensional moduli spaces.

\begin{proposition}
\label{prop:q-dim-increase-closed}
    Let $\C$ be a flow category equipped with a function 
    \[
        q: \Ob(\C) \rightarrow \ZZ.
    \]
    Suppose there exists some $k > 0$ such that for any object $x, y$ in $\C$ with $|x| - |y| \leq k$, the moduli space $\M(x, y)$ is non-empty if and only if $q(x) \geq q(y)$. Then for any object $x, y$ in $\C$ with $|x| - |y| = k + 1$ and $q(x) < q(y)$, the moduli space $\M(x, y)$ is a closed $k$-dimensional manifold. 
\end{proposition}

\begin{proof}
    Take any $x, y$ in $\C$ with $|x| - |y| = k + 1$ and $q(x) < q(y)$. If $\partial \M(x, y)$ is non-empty, then there must be an object $z$ between $x$ and $y$ such that 
    \[
        \varnothing \neq \M(z, y) \times \M(x, z) \hookrightarrow \partial \M(x, y).
    \]
    However from the hypothesis, we must have $q(x) \geq q(z) \geq q(y)$ which is a contradiction.
\end{proof}

One possible strategy for proving \Cref{conj:1} is to perform extended Whitney tricks inductively on $\C = \C_\BN(D)$  
\[
    \C \rightarrow \C^{(0)} \rightarrow \C^{(1)} \rightarrow \C^{(2)} \rightarrow \cdots
\]
so that $\C^{(k)}$ contains no quantum grading increasing moduli spaces of dimension $\leq k$ and the $(k + 1)$-dimensional ones (closed manifolds) are framed null-cobordant. The initial framing on $\C_{XY}(D)$ must be given coherently, or maybe isotoped during the process, so that the induction actually works. 

\subsection{Steenrod squares}

In \cite{LS:2014_steenrod}, Lipshitz and Sarkar gave an explicit algorithm to comute the first and second Steenrod squares of the Khovanov homotopy type. Seed implemented the algorithm and showed by direct computations that Khovanov homotopy type is a strictly stronger invariant than the Khovanov homology \cite{Seed:2012}. Lately Lobb, Orson and Sch{\"u}tz \cite{LOS:2017-algorithm} gave an algorithm to compute the second Steenrod square on the spectrum obtained from a general framed flow category. 

Here we show that the latter algorithm can be applied to our category $\C_\BN$. Although we know from \Cref{thm:2} that all Steenrod squares are trivial on $\X_\BN$, they might act non-trivially on a subspectrum (or a quotient spectrum) of $\X_\BN$. In particular if \Cref{conj:1} is true, then we might get non-trivial results by observing how the operations act on the filtration.

\begin{definition}
\label{def:frame-assign}
    Let $s$ be a sign assignment on the $n$-cube $K(n)$. A $2$-cochain $f \in C^2(K(n); \FF_2)$ is called a \textit{frame assignment compatible with} $s$ if for any 3-face $u >_3 v$ of $K(n)$,
    \[
        \delta f(e_{u, v}) = \sum_{v <_1 w <_2 u} s(e_{w, v}).
    \]
    We call such pair $(s, f)$ a \textit{frame assignment pair} on $K(n)$.
\end{definition}

In \cite{LS:2014_steenrod}, together with the \textit{standard sign assignment} $s$, the \textit{standard frame assignment} $f$ is defined as
\[
    f(e) = (v_1 + \cdots + v_{i - 1})(v_{i + 1} + \cdots + v_{j - 1}) \pmod{2}
\]
for each $2$-face $e = (v_1, \ldots, v_{i - 1}, \star, v_{i + 1}, \ldots, v_{j - 1}, \star, v_{j + 1}, \ldots, v_n)$. It is proved in \cite[Lemma 2.1]{LS:2014_steenrod} that this $(s, f)$ is a framing assignment pair in the sense of \Cref{def:frame-assign}. Moreover, it is proved in \cite[Lemma 3.5]{LS:2014_steenrod} that the $n$-dimensional cube flow category can be framed so that the $0$-dimensional moduli spaces are framed according to $s$, and the $1$-dimensional moduli spaces are framed according to $f$. In fact, the proof only uses the compatibility condition of \Cref{def:frame-assign}, so $(s, f)$ need not be standard ones. 

\begin{proposition}
\label{prop:frame-assign-exist}
    For any sign assignment $s$ on the $n$-cube $K(n)$, there exists a frame assignment $f$ compatible with $s$. 
\end{proposition}

\begin{proof}
    This follows immediately from the fact that the right hand side of the compatibility condition is a cocycle, and that $K(n)$ is acyclic. Here we give another constructive proof. Take any sign assignment $s$, and let $s_0, f_0$ be the standard sign assignment and the standard frame assignment on $K(n)$ respectively. Put $t = s - s_0$, then $t$ is a $1$-cocycle on $K(n)$. Since $K(n)$ is acyclic, there exists a $0$-cochain $b$ such that $\delta b = t$ ($b$ can be obtained directly by integrating $t$ along the edges). Define a $2$-cochain $g$ by $g(e_{u, w}) = b(w)$ for each $u >_2 w$. Then for any $u >_3 v$,
    \[
        \delta g(e_{u, v}) 
        = 3 b(v) + \sum_{v <_1 w <_2 u} b(w)
        = \sum_{v <_1 w <_2 u} \delta b(e_{w, v}).
    \]
    With $f = f_0 + g$, we have
    \[
        \delta f(e_{u, v}) 
        = \sum_{v <_1 w <_2 u} (s_0(e_{w, v}) + \delta b(e_{w, v}))
        = \sum_{v <_1 w <_2 u} s(e_{w, v})
    \]
    as desired. 
\end{proof}

The algorithm of \cite{LOS:2017-algorithm} only requires that a flow category $\C$ is equipped with a sign assignment
\[
    s: \Mor^{0}(\C) \rightarrow \FF_2
\]
and a frame assignment
\[
    f: \Mor^{1}(\C) \rightarrow \FF_2
\]
satisfying the compatibility condition of \cite[Definition 3.3]{LOS:2017-algorithm}. Here $\Mor^{0}, \Mor^{1}$ denotes the set of all components of the $0$-, $1$-dimensional moduli spaces of $\C$ respectively. One can easily check that any frame assignment pair $(s, f)$ on $K(n)$ satisfies the required condition for the category $\C = \Cube{n}$. 

Now, recall from \Cref{subsec:framing} that the $XY$-based Bar-Natan flow category $\C_{XY}(D)$ is constructed as disjoint union of cube flow categories. From \Cref{prop:frame-assign-exist} we can take a frame assignment pair for each subcube and frame each of them so that the $0$-, $1$-dimensional moduli spaces are framed according to $(s, f)$. \cite[Section 4]{LOS:2017-algorithm} describes the effect on the frame assignment pair by handle slides, so the frame assignment pair for $\C_\BN(D) = \C_\BN(D)$ and any of its downward- or upward-closed subcategory can be made explicit. Thus we summarize:

\begin{proposition}
    For any subspectrum (resp.\ quotient spectrum) $\X'$ of $\X = \X_\BN(L)$ obtained from a downward (resp.\ upward) closed subcategory of $\C_\BN(L)$, there is an algorithm that computes the first and the second Steenrod squares
    \[
        \Sq^i: \tilde{H}^*(\X'; \FF_2) \rightarrow \tilde{H}^{*+i}(\X'; \FF_2),\quad (i = 1, 2).
    \]
\end{proposition}

\subsection{Cohomotopical refinement of the $s$-invariant}

Finally we consider the possibility of cohomotopically refining the $s$-invariant. First we review the definition of $s$ given in \cite{Rasmussen:2010,MTV:2007,LS:2014_rasmussen}. The quantum filtration $\{F^j H\}$ on $H = H_\BN$ is induced from the quantum filtration $\{F^j C\}$ of $C = C_\BN$ as
\[
    F^j H = \Ima(i_*: H(F^j C) \rightarrow H(C))
\]
and the quantum grading on $H_\BN$ is defined as 
\[
    \gr_q(z) = \max\set{ j | z \in F^j H }.
\]
Let $K$ be a knot and $F$ be a field. Put
\begin{align*}
    s_{\min}(K; F) &= \min \{\ \gr_q(z) \mid z \in H_\BN(K; F) \ \}, \\ 
    s_{\max}(K; F) &= \max \{\ \gr_q(z) \mid z \in H_\BN(K; F) \ \}.
\end{align*}
Then the \textit{$s$-invariant} over $F$ is defined as
\[
    s(K; F) = \frac{s_{\min}(K; F) + s_{\max}(K; F)}{2}.
\]
Although Rasmussen's original definition uses Lee homology (originally over $\QQ$), it follows from \cite[Proposition 3.1]{MTV:2007} that when $\fchar{F} \neq 2$ the invariant can be equally defined using Bar-Natan homology. Moreover \cite[Proposition 2.6]{LS:2014_rasmussen} asserts that Bar-Natan homology allows to extend the definition of $s$ to fields $F$ of $\fchar{F} = 2$. It is known that 
\[
    s_{\max}(K; F) - s_{\min}(K; F) = 2
\]
and that the canonical classes $[\ca], [\cb] \in H_\BN(K; F)$ satisfy
\[
    \gr_q[\ca] = \gr_q[\cb] = s_{\min}(K; F).
\]
Thus we may alternatively define
\[
    s(K; F) = \gr_q[\ca] + 1 = \gr_q[\cb] + 1.
\]

Now, \textit{with the assumption that \Cref{conj:1} is true}, we define an analogous invariant $\bar{s}(K)$ using the canonical cohomotopy class of $K$. Suppose we have the quantum filtration
\[
    \{\pt\} \subset \cdots \subset F_{j - 2} \X \subset F_j \subset \cdots \subset \X.
\]
Consider the cofiltration
\[
    \X \rightarrow \cdots \rightarrow F^j \X \rightarrow F^{j + 2} \rightarrow \cdots \rightarrow \{\pt\}
\]
where $F^j \X = \X / F_{j - 2} \X$, and the arrows are the obvious projections. Then a descending filtration on $\pi^*(\X)$ is defined as 
\[
    F^j \pi^*(\X) = \Ima(p^*: \pi^*(F^j\X) \rightarrow \pi^*(\X))
\]
and the quantum grading function as
\[
    \gr_q(z) = \max \{\ j \mid z \in F^j \pi^*(\X) \ \}.
\]

\begin{definition}
\label{def:s-bar}
    (Assuming \Cref{conj:1} is true.) Define
    \[
        \bar{s}(K) = \gr_q[p_\ca] + 1
    \]
    where $[p_\ca]$ is the canonical cohomotopy class of $K$. 
\end{definition}

Equivalently, consider the following diagram in the homotopy category of CW-spectra
\begin{equation}
\label{eq:ca-factor-problem}
    \begin{tikzcd}
    \X \arrow[rr, "p_\ca"] \arrow[rd, "p"] & & \SS \\
    & F^j \X \arrow[ru, "f", dashed] & 
    \end{tikzcd}
\end{equation}
Then $\bar{s}(K)$ is given by the maximal $j + 1$ such that the above diagram has a solution $f$. We show that many of the important properties of $s$ also hold for $\bar{s}$. 

\begin{lemma}
\label{lem:bar-s-ineq}
    (Assuming \Cref{conj:1} is true.) For any field $F$, 
    \[
        \bar{s}(K) \leq s(K; F).
    \] 
\end{lemma}

\begin{proof}
    Immediate from \Cref{prop:coHur-ca}.
\end{proof}

\begin{lemma}
\label{lem:s-bar-positive}
    (Assuming \Cref{conj:1} is true.) Suppose $K$ is a positive knot. Take a positive diagram $D$ representing $K$, and let $S$ be the Seifert surface of $K$ obtained by applying Seifert's algorithm to $D$. Then
    \[
        \bar{s}(K) = 2g(S)
    \]
    where $g(S)$ is the genus of $S$.
\end{lemma}

\begin{proof}
    As argued in \cite[Section 5.2]{Rasmussen:2010}, the lowest $q$-grading of $\C_\BN(D)$ is given by $2g(S) - 1$. Thus $2g(S) \leq \bar{s}(K)$. The opposite inequality follows from \Cref{lem:bar-s-ineq} and the similar result for the ordinary $s$-invariant $s(K) = 2g(S)$.
\end{proof}

\begin{corollary}
\label{cor:s-bar-unknot}
    (Assuming \Cref{conj:1} is true.) $\bar{s}(U) = 0$ for the unknot $U$.
\qed
\end{corollary}

\begin{proposition}
\label{prop:bar-s-difference}
    (Assuming \Cref{conj:1} is true.) Let $S$ be a connected cobordism between knots $K, K'$. Then
    \[
        |\bar{s}(K) - \bar{s}(K')| \leq 2g(S). 
    \]
\end{proposition}

\begin{proof}
    Let $[p_\ca], [p'_\ca]$ be the canonical cohomotopy classes of $K, K'$ respectively.  Consider the following diagram
    \begin{equation*}
        \begin{tikzcd}
        \X_\BN(K) \arrow[rr, "p_\ca"] \arrow[rd, "f_S"] \arrow[d] &                                                 & \SS \arrow[rd, equal] &     \\
        F^j \X_\BN(K) \arrow[rru, dashed] \arrow[rd, "f_S"]       & \X_\BN(K') \arrow[rr, "p'_\ca"] \arrow[d]       &                                     & \SS \\
                                                                  & F^{j' - \chi(S)} \X_\BN(K') \arrow[rru, dashed] &                                     &    
        \end{tikzcd}
    \end{equation*}
    The top square commutes from \Cref{prop:canon-coh-cobordism} , and the left square commutes from \Cref{rem:cobordism-map-filtered}. With $\chi(S) = -2g(S)$ we have $\bar{s}(K) \geq \bar{s}(K') + 2g(S)$. By reversing $S$ we get $\bar{s}(K') \geq \bar{s}(K) + 2g(S)$.
\end{proof}

\begin{corollary}
\label{cor:bar-s-bound}
    (Assuming \Cref{conj:1} is true.)
    \begin{enumerate}
        \item $\bar{s}$ is a knot concordance invariant.
        \item $|\bar{s}(K)| \leq 2g_4(K)$. 
        \item If $K$ is a positive knot, $\bar{s}(K) = 2g(K) = 2g_4(K)$.
    \end{enumerate}
    Here $g_4(K)$ denotes the (smooth) slice genus of $K$.
\end{corollary}

\begin{proof}
    Claims 1 and 2 are immediate from \Cref{prop:bar-s-difference} and \Cref{cor:s-bar-unknot}. For claim 3, from \Cref{lem:s-bar-positive} and claim 2 we have 
    \[
        \bar{s}(K) \leq 2g_4(K) \leq 2g(K) \leq 2g(S) = \bar{s}(K). 
    \]
\end{proof}

Thus again, as a corollary, the Milnor conjecture can be reproved using $\bar{s}$.

\begin{corollary}[The Milnor conjecture \cite{Milnor:1968,KM:1993,Rasmussen:2010} ]
     The slice genus and unknotting number of the $(p, q)$ torus knot are both equal to $(p - 1)(q - 1)/2$. \qed
\end{corollary}

\begin{remark}
    We cannot expect that $\bar{s}$ is a homomorphism from the knot concordance group, in particular that $\bar{s}(m(K)) = -\bar{s}(K)$ holds. If it should, then it implies that all $s(- ; F)$ are equal to $\bar{s}$, from
    \[
        \bar{s}(K) \leq s(K; F) = -s(m(K); F) \leq -\bar{s}(m(K)) = \bar{s}(K).
    \]
    However it is known that $s(-; \QQ) \neq s(-; \FF_2)$, see \cite[Remark 6.1]{LS:2014_rasmussen}. 
\end{remark}

The above arguments also suggest that $s$ can be generalized using any general cohomology theory $h^*$. By applying $h^*$ to \eqref{eq:ca-factor-problem}, we get
\begin{equation*}
    \begin{tikzcd}
    h^*(\X) & & h^*(\SS) \arrow[ll, "p_\ca^*"']  \arrow[ld, "f^*"', dashed]  \\
    & h^*(F^j \X) \arrow[lu, "p^*"'] & 
    \end{tikzcd}
\end{equation*}
We may define $s(K; h^*)$ as the maximal $j + 1$ such that the above diagram can be solved, with $f^*$ replaced by an arbitrary map in the corresponding category. In particular with $h^* = \tilde{H}^*(-; F)$ we get $s(K; F)$. We can expect that $s(K; h^*)$ gives a more tractable invariant than $\bar{s}$, and that it gives a better lower bound for the slice genus or detect non-sliceness that $s$ fails to detect. 


    \printbibliography

@article {BarNatan:2004,
    AUTHOR = {Bar-Natan, Dror},
     TITLE = {Khovanov's homology for tangles and cobordisms},
   JOURNAL = {Geom. Topol.},
  FJOURNAL = {Geometry and Topology},
    VOLUME = {9},
      YEAR = {2005},
     PAGES = {1443--1499},
      ISSN = {1465-3060},
   MRCLASS = {57M27 (57M25 57R56)},
  MRNUMBER = {2174270},
MRREVIEWER = {Justin Sawon},
       DOI = {10.2140/gt.2005.9.1443},
       URL = {https://doi.org/10.2140/gt.2005.9.1443},
}

@ARTICLE{Bar-Natan:2005,
  title   = "Khovanov's homology for tangles and cobordisms",
  author  = "Bar-Natan, Dror",
  journal = "Volume",
  volume  =  9,
  pages   = "1443--1499",
  year    =  2005
}

@book {Bredon:1993,
    AUTHOR = {Bredon, Glen E.},
     TITLE = {Topology and geometry},
    SERIES = {Graduate Texts in Mathematics},
    VOLUME = {139},
 PUBLISHER = {Springer-Verlag, New York},
      YEAR = {1993},
     PAGES = {xiv+557},
      ISBN = {0-387-97926-3},
   MRCLASS = {55-01 (54-01 57-01)},
  MRNUMBER = {1224675},
MRREVIEWER = {Donald W. Kahn},
       DOI = {10.1007/978-1-4757-6848-0},
       URL = {https://doi-org.utokyo.idm.oclc.org/10.1007/978-1-4757-6848-0},
}

@INCOLLECTION{CJS:1995,
  title     = "Floer's infinite dimensional Morse theory and homotopy theory",
  booktitle = "The Floer memorial volume",
  author    = "Cohen, Ralph L and Jones, John D S and Segal, Graeme B",
  publisher = "Springer",
  pages     = "297--325",
  year      =  1995
}

@article {Khovanov:2000,
    AUTHOR = {Khovanov, Mikhail},
     TITLE = {A categorification of the {J}ones polynomial},
   JOURNAL = {Duke Math. J.},
  FJOURNAL = {Duke Mathematical Journal},
    VOLUME = {101},
      YEAR = {2000},
    NUMBER = {3},
     PAGES = {359--426},
      ISSN = {0012-7094},
   MRCLASS = {57M27 (57R56)},
  MRNUMBER = {1740682},
       DOI = {10.1215/S0012-7094-00-10131-7},
       URL = {https://doi.org/10.1215/S0012-7094-00-10131-7},
}

@article {Khovanov:2004,
    AUTHOR = {Khovanov, Mikhail},
     TITLE = {Link homology and {F}robenius extensions},
   JOURNAL = {Fund. Math.},
  FJOURNAL = {Fundamenta Mathematicae},
    VOLUME = {190},
      YEAR = {2006},
     PAGES = {179--190},
      ISSN = {0016-2736},
   MRCLASS = {57M27 (57R56)},
  MRNUMBER = {2232858},
MRREVIEWER = {Jacob Andrew Rasmussen},
       DOI = {10.4064/fm190-0-6},
       URL = {https://doi.org/10.4064/fm190-0-6},
}

@article {KM:1993,
    AUTHOR = {Kronheimer, P. B. and Mrowka, T. S.},
     TITLE = {Gauge theory for embedded surfaces. {I}},
   JOURNAL = {Topology},
  FJOURNAL = {Topology. An International Journal of Mathematics},
    VOLUME = {32},
      YEAR = {1993},
    NUMBER = {4},
     PAGES = {773--826},
      ISSN = {0040-9383},
   MRCLASS = {57R57 (57N13 57R40 57R55 58D29)},
  MRNUMBER = {1241873},
MRREVIEWER = {Ronald J. Stern},
       DOI = {10.1016/0040-9383(93)90051-V},
       URL = {https://doi.org/10.1016/0040-9383(93)90051-V},
}

@article {Lee:2005,
    AUTHOR = {Lee, Eun Soo},
     TITLE = {An endomorphism of the {K}hovanov invariant},
   JOURNAL = {Adv. Math.},
  FJOURNAL = {Advances in Mathematics},
    VOLUME = {197},
      YEAR = {2005},
    NUMBER = {2},
     PAGES = {554--586},
      ISSN = {0001-8708},
   MRCLASS = {57M27},
  MRNUMBER = {2173845},
MRREVIEWER = {Paola Cristofori},
       DOI = {10.1016/j.aim.2004.10.015},
       URL = {https://doi.org/10.1016/j.aim.2004.10.015},
}

@phdthesis{Lewark:2009,
  title={The Rasmussen invariant of arborescent and of mutant links},
  author={Lewark, Lukas},
  year={2009},
  school={Master thesis, ETH Z{\"u}rich}
}

@article {LS:2014,
    AUTHOR = {Lipshitz, Robert and Sarkar, Sucharit},
     TITLE = {A {K}hovanov stable homotopy type},
   JOURNAL = {J. Amer. Math. Soc.},
  FJOURNAL = {Journal of the American Mathematical Society},
    VOLUME = {27},
      YEAR = {2014},
    NUMBER = {4},
     PAGES = {983--1042},
      ISSN = {0894-0347},
   MRCLASS = {57M25 (55P42)},
  MRNUMBER = {3230817},
MRREVIEWER = {Nikolai N. Saveliev},
       DOI = {10.1090/S0894-0347-2014-00785-2},
       URL = {https://doi.org/10.1090/S0894-0347-2014-00785-2},
}

@ARTICLE{LS:2014_steenrod,
  title     = "A Steenrod square on Khovanov homology",
  author    = "Lipshitz, Robert and Sarkar, Sucharit",
  journal   = "J. Topol.",
  publisher = "Oxford Academic",
  volume    =  7,
  number    =  3,
  pages     = "817--848",
  month     =  mar,
  year      =  2014,
  language  = "en"
}

@article {LS:2014_rasmussen,
    AUTHOR = {Lipshitz, Robert and Sarkar, Sucharit},
     TITLE = {A refinement of {R}asmussen's {$S$}-invariant},
   JOURNAL = {Duke Math. J.},
  FJOURNAL = {Duke Mathematical Journal},
    VOLUME = {163},
      YEAR = {2014},
    NUMBER = {5},
     PAGES = {923--952},
      ISSN = {0012-7094},
   MRCLASS = {57M25 (55P42)},
  MRNUMBER = {3189434},
MRREVIEWER = {Laurence R. Taylor},
       DOI = {10.1215/00127094-2644466},
       URL = {https://doi.org/10.1215/00127094-2644466},
}

@ARTICLE{LLS:2020,
  title     = "Khovanov homotopy type, Burnside category and products",
  author    = "Lawson, Tyler and Lipshitz, Robert and Sarkar, Sucharit",
  journal   = "Geom. Topol.",
  publisher = "Mathematical Sciences Publishers",
  volume    =  24,
  number    =  2,
  pages     = "623--745",
  month     =  sep,
  year      =  2020
}

@article {LNS:2015,
    AUTHOR = {Lipshitz, Robert and Ng, Lenhard and Sarkar, Sucharit},
     TITLE = {On transverse invariants from {K}hovanov homology},
   JOURNAL = {Quantum Topol.},
  FJOURNAL = {Quantum Topology},
    VOLUME = {6},
      YEAR = {2015},
    NUMBER = {3},
     PAGES = {475--513},
      ISSN = {1663-487X},
   MRCLASS = {57M27 (57M25 57R17)},
  MRNUMBER = {3392962},
MRREVIEWER = {Hernando Burgos-Soto},
       DOI = {10.4171/QT/69},
       URL = {https://doi.org/10.4171/QT/69},
}

@inproceedings {LS:2018-note,
    AUTHOR = {Lipshitz, Robert and Sarkar, Sucharit},
     TITLE = {Spatial refinements and {K}hovanov homology},
 BOOKTITLE = {Proceedings of the {I}nternational {C}ongress of
              {M}athematicians---{R}io de {J}aneiro 2018. {V}ol. {II}.
              {I}nvited lectures},
     PAGES = {1153--1173},
 PUBLISHER = {World Sci. Publ., Hackensack, NJ},
      YEAR = {2018},
   MRCLASS = {57K18 (55P42)},
  MRNUMBER = {3966803},
}

@misc{LLS:2021-functoriality,
      title={Homotopy functoriality for Khovanov spectra}, 
      author={Tyler Lawson and Robert Lipshitz and Sucharit Sarkar},
      year={2021},
      eprint={2104.12907},
      archivePrefix={arXiv},
      primaryClass={math.GT}
}

@article {JLS:2017,
    AUTHOR = {Jones, Dan and Lobb, Andrew and Sch\"{u}tz, Dirk},
     TITLE = {Morse moves in flow categories},
   JOURNAL = {Indiana Univ. Math. J.},
  FJOURNAL = {Indiana University Mathematics Journal},
    VOLUME = {66},
      YEAR = {2017},
    NUMBER = {5},
     PAGES = {1603--1657},
      ISSN = {0022-2518},
   MRCLASS = {57R58},
  MRNUMBER = {3718437},
MRREVIEWER = {Kristen Hendricks},
       DOI = {10.1512/iumj.2017.66.6136},
       URL = {https://doi.org/10.1512/iumj.2017.66.6136},
}

@ARTICLE{LOS:2018,
  title     = "Framed cobordism and flow category moves",
  author    = "Lobb, Andrew and Orson, Patrick and Sch{\"u}tz, Dirk",
  journal   = "Algebr. Geom. Topol.",
  publisher = "Mathematical Sciences Publishers",
  volume    =  18,
  number    =  5,
  pages     = "2821--2858",
  month     =  aug,
  year      =  2018
}

@article{LOS:2017-algorithm,
      title={Khovanov homotopy calculations using flow category calculus}, 
      author={Andrew Lobb and Patrick Orson and Dirk Schuetz},
      year={2017},
      eprint={1710.01857},
      archivePrefix={arXiv},
      primaryClass={math.GT}
}

@article {MTV:2007,
    AUTHOR = {Mackaay, Marco and Turner, Paul and Vaz, Pedro},
     TITLE = {A remark on {R}asmussen's invariant of knots},
   JOURNAL = {J. Knot Theory Ramifications},
  FJOURNAL = {Journal of Knot Theory and its Ramifications},
    VOLUME = {16},
      YEAR = {2007},
    NUMBER = {3},
     PAGES = {333--344},
      ISSN = {0218-2165},
   MRCLASS = {57M25 (57M27)},
  MRNUMBER = {2320159},
MRREVIEWER = {Scott Morrison},
       DOI = {10.1142/S0218216507005312},
       URL = {https://doi.org/10.1142/S0218216507005312},
}

@book {Milnor:1968,
    AUTHOR = {Milnor, John},
     TITLE = {Singular points of complex hypersurfaces},
    SERIES = {Annals of Mathematics Studies, No. 61},
 PUBLISHER = {Princeton University Press, Princeton, N.J.; University of Tokyo Press, Tokyo},
      YEAR = {1968},
     PAGES = {iii+122},
   MRCLASS = {57.20 (14.00)},
  MRNUMBER = {0239612},
MRREVIEWER = {J. P. Levine},
}

@ARTICLE{Piccirillo:2020,
  title     = "The Conway knot is not slice",
  author    = "Piccirillo, Lisa",
  journal   = "Ann. Math.",
  publisher = "[Annals of Mathematics, Trustees of Princeton University on
               Behalf of the Annals of Mathematics, Mathematics Department,
               Princeton University]",
  volume    =  191,
  number    =  2,
  pages     = "581--591",
  year      =  2020
}

@ARTICLE{Rasmussen:2010,
  title   = "Khovanov homology and the slice genus",
  author  = "Rasmussen, Jacob",
  journal = "Invent. Math.",
  volume  =  182,
  number  =  2,
  pages   = "419--447",
  month   =  nov,
  year    =  2010
}

@ARTICLE{Sano:2020,
  title     = "A description of Rasmussen's invariant from the divisibility of
               Lee's canonical class",
  author    = "Sano, Taketo",
  journal   = "J. Knot Theory Ramif.",
  publisher = "World Scientific Pub Co Pte Lt",
  volume    =  29,
  number    =  06,
  pages     = "2050037",
  month     =  may,
  year      =  2020,
  language  = "en"
}

@ARTICLE{Sano:2020-b,
  title         = "Fixing the functoriality of Khovanov homology: a simple
                   approach",
  author        = "Sano, Taketo",
  month         =  aug,
  year          =  2020,
  archivePrefix = "arXiv",
  primaryClass  = "math.GT",
  eprint        = "2008.02131"
}

@ARTICLE{Seed:2012,
  title         = "Computations of the {Lipshitz-Sarkar} Steenrod Square on
                   Khovanov Homology",
  author        = "Seed, Cotton",
  month         =  oct,
  year          =  2012,
  archivePrefix = "arXiv",
  primaryClass  = "math.GT",
  eprint        = "1210.1882"
}

@ARTICLE{Turner:2020,
  title     = "Khovanov homology and diagonalizable Frobenius algebras",
  author    = "Turner, Paul",
  journal   = "J. Knot Theory Ramif.",
  publisher = "World Scientific Publishing Co.",
  volume    =  29,
  number    =  01,
  pages     = "1950095",
  month     =  jan,
  year      =  2020
}

@article {Wehrli:2008,
    AUTHOR = {Wehrli, S.},
     TITLE = {A spanning tree model for {K}hovanov homology},
   JOURNAL = {J. Knot Theory Ramifications},
  FJOURNAL = {Journal of Knot Theory and its Ramifications},
    VOLUME = {17},
      YEAR = {2008},
    NUMBER = {12},
     PAGES = {1561--1574},
      ISSN = {0218-2165},
   MRCLASS = {57M27},
  MRNUMBER = {2477595},
MRREVIEWER = {Qingtao Chen},
       DOI = {10.1142/S0218216508006762},
       URL = {https://doi.org/10.1142/S0218216508006762},
}

\end{document}